\documentclass[reqno]{amsart}%
\usepackage[utf8]{inputenc}%
\usepackage[english]{babel}%
\usepackage{amsmath,amssymb,amsthm,amsfonts}%
\usepackage[foot]{amsaddr}%
\usepackage{hyperref}%
\usepackage{mathrsfs,bbm}%
\usepackage{enumerate}%
\usepackage{stmaryrd}%
\allowdisplaybreaks%
\numberwithin{equation}{section}%
\newcommand{\la}{\lambda}
\newcommand{\te}{\theta}
\newcommand{\al}{\alpha}
\newcommand{\si}{\sigma}
\newcommand{\N}{\mathbb{Z}_{>0}}
\newcommand{\Z}{\mathbb{Z}}
\newcommand{\C}{\mathbb{C}}
\newcommand{\Sb}{\mathbb{S}}
\newcommand{\R}{\mathbb{R}}
\newcommand{\Pf}{\mathop{\mathrm{Pf}}}
\newcommand{\Cl}{\mathop{\mathit{Cl}}}
\newcommand{\gdim}{\mathop{\dim_{\rule{0.6pt}{0pt}\Sb}}}
\newcommand{\sh}[1]{#1{{{\scriptstyle\boldsymbol`}}}}
\renewcommand{\i}{\mathrm{i}}
\renewcommand{\j}{\mathrm{j}}
\newcommand{\un}[1]{\underline{{#1}}}
\newcommand{\uno}[1]{\underline{{\rule[-1pt]{0pt}{0pt}#1}}}
\newcommand{\unt}[1]{\underline{{\rule[-2pt]{0pt}{0pt}#1}}}
\newcommand{\Uf}{\mathsf{U}}

\newcommand{\Df}{\mathsf{D}}
\newcommand{\Hf}{\mathsf{H}}
\newtheorem{prop}{Proposition}[section]
\newtheorem{lemma}[prop]{Lemma}
\newtheorem{corollary}[prop]{Corollary}
\newtheorem{thm}[prop]{Theorem}
\theoremstyle{definition}
\newtheorem{mainthm}{Theorem}
\newtheorem{properties}[prop]{Properties}
\newtheorem{df}[prop]{Definition}
\newtheorem{rmk}[prop]{Remark}
\begin{document}
\title{Pfaffian Stochastic Dynamics of Strict Partitions}
\author[Leonid Petrov]{Leonid Petrov}
\thanks{The author was partially supported by the RFBR-CNRS grant  10-01-93114, by A.~Kuznetsov's graduate student scholarship and by the Dynasty foundation fellowship for young scientists}
\address{Dobrushin Mathematics Laboratory,
Kharkevich Institute for Information Transmission Problems,
Bolshoy Karetny per.~19, Moscow, 127994, Russia}
\email{lenia.petrov@gmail.com}

\begin{abstract}
  We study a family of continuous time Markov jump processes on strict partitions (partitions with distinct parts) preserving the distributions introduced by Borodin \cite{Borodin1997} in connection with projective representations of the infinite symmetric group. The one-dimensional distributions of the processes (i.e., the Borodin's measures) have determinantal structure. We express the dynamical correlation functions of the processes in terms of certain Pfaffians and give explicit formulas for both the static and dynamical correlation kernels using the Gauss hypergeometric function. Moreover, we are able to express our correlation kernels (both static and dynamical) through those of the z-measures on partitions obtained previously by Borodin and Olshanski in a series of papers.
  
  The results about the fixed time case were announced in the note \cite{Petrov2010}. A part of the present paper contains proofs of those results.
\end{abstract}

\maketitle

\setcounter{tocdepth}{1}
\tableofcontents
\setcounter{tocdepth}{2}


\section{Introduction} 
\label{sec:introduction}

We introduce and study a family\footnote{The whole picture depends on two continuous parameters $\al>0$ and $0<\xi<1$.} of continuous time Markov jump processes on the set of all strict partitions (that is, partitions in which nonzero parts are distinct). Our Markov processes preserve the family of probability measures
introduced by Borodin \cite{Borodin1997} in connection with the harmonic analysis of projective representations of the infinite symmetric group. The construction of our dynamics is similar to that of Borodin and Olshanski \cite{Borodin2006} and is based on a special coherency property\footnote{which has a representation-theoretic meaning.} of the measures on strict partitions introduced in \cite{Borodin1997}. Regarding each strict partition $\la=(\la_1>\dots>\la_\ell>0)$, $\la_j\in\Z$, as a point configuration $\left\{ \la_1,\dots,\la_{\ell}\right\}$ on the half-lattice $\N:=\left\{ 1,2,\dots \right\}$, one can say that the state space of our Markov processes is the space of all finite point configurations on $\N$. The fixed time distributions of our dynamics are probability measures on this configuration space. In other words, in the static (fixed time) picture one sees a random point process on $\N$.

The main result of the paper is the computation of the dynamical (or space-time) correlation functions for our family of Markov processes. We show that these correlation functions have certain Pfaffian form, and compute the corresponding kernel. Here the kernel is a function $\boldsymbol\Phi_{\al,\xi}(s,x;t,y)$ of two space-time variables, where $x,y\in\Z$ and $s,t\in\R$, which is explicitly expressed through the Gauss hypergeometric function. Following the common terminology (e.g., see \cite{nagao1998multilevel}, \cite{johansson2005non}, \cite{Borodin2006}), we call $\boldsymbol\Phi_{\al,\xi}$ the \textit{extended} (\textit{Pfaffian}) \textit{hypergeometric-type kernel}.

In the static case the Pfaffian formula for the correlation functions of our Markov processes can be reduced to a determinantal one. Thus, in the fixed time picture we have a determinantal point process on $\N$. Its kernel $\mathbf{K}_{\al,\xi}$ has integrable form and is also expressed through the Gauss hypergeometric function. (About integrable operators, e.g., see \cite{its1990differential}, \cite{deift1999integrable}, discrete integrable operators are discussed in \cite{Borodin2000riemann} and \cite[\S6]{Borodin2000a}.) We call this kernel the \textit{hypergeometric-type kernel}. The results about the static case were announced in the note \cite{Petrov2010}. A part of the present paper (\S\ref{sec:kerov_s_operators}--\S\ref{sec:static_determinantal_kernel}) contains complete proofs of those results.

Models with correlation functions of Pfaffian form first appeared in theory of random matrices, e.g., see \cite{Dyson1970correlations}, \cite{Pandey1983gaussian}, \cite{Mehta1983some}, \cite{nagaoWadati1991}, \cite{nagaoWadati1992}, \cite{Tracy1996orthogonal}, \cite{Nagao2007pfaffian}, and the book by Mehta \cite{mehta2004random}. An essentially time-inhomogeneous Pfaffian dynamical model of random-matrix type was considered by Nagao, Katori and Tanemura \cite{NagaoKatoriTanemura2004PfDyn}, \cite{Katori2005PfDyn}. Static Pfaffian random point processes of various origins have also been studied, e.g., see \cite{Rains2000}, \cite{ferrari2004polynuclear}, \cite{Matsumoto2005}, \cite{borodin2005eynard}, \cite{Vuletic2007shifted}, and \S10 of the survey \cite{Borodin2009}. Borodin and Strahov \cite{Borodin2006averages}, \cite{Borodin2009correlation}, \cite{Strahov2009} considered static models which are discrete analogues of Pfaffian models of random-matrix type, they involve random ordinary (i.e., not necessary strict) partitions and have a representation-theoretic interpretation (see \cite{Strahov2009z}). The dynamical model that we study in the present paper seems to be a first example of a \textit{stationary} (in contrast to the model of \cite{NagaoKatoriTanemura2004PfDyn}, \cite{Katori2005PfDyn}) Pfaffian dynamics.

\subsection*{Comparison with results for the $z$-measures}

Our model of random strict partitions and associated stochastic dynamics is very similar to the one of the $z$-measures on ordinary partitions.\footnote{The $z$-measures originated from the problem of harmonic analysis for the infinite symmetric group $\mathfrak{S}_\infty$ \cite{Kerov1993}, \cite{Kerov2004} and were studied in detail by Borodin, Okounkov, Olshanski, and other authors, e.g., see the bibliography in \cite{Borodin2007}.} The structure of static and dynamical correlation functions in that case was investigated in \cite{Borodin2000a}, \cite{Okounkov2001a}, \cite{Borodin2006}, \cite{borodin2006meixner}. Let us discuss the relationship of our results with the ones from those papers.

\begin{enumerate}[$\bullet$]
  \item The main feature of our model is that its dynamical correlation functions are expressed in terms of Pfaffians and not determinants, as it is for the $z$-measures.
  
  \item Determinantal (static) correlation kernels of random point processes often appear to be projection operators. In particular, this holds for the $z$-measures. In contrast, in our situation the kernel $\mathbf{K}_{\al,\xi}(x,y)$ ($x,y\in\N$) is not a projection operator in the corresponding coordinate Hilbert space $\ell^2(\N)$.
  
  \item On the other hand, the static Pfaffian kernel $\boldsymbol\Phi_{\al,\xi}(s,x;s,y)$ (where $x,y\in\Z$) in our model has a structure which is very similar to that of the determinantal kernel of the $z$-measures on semi-infinite point configurations on the lattice. Viewed as an operator in the Hilbert space $\ell^2(\Z)$, $\boldsymbol\Phi_{\al,\xi}(s,\cdot;s,\cdot)$, is a rank one perturbation of an orthogonal projection operator.
  
  \item Furthermore, our extended Pfaffian kernel $\boldsymbol\Phi_{\al,\xi}(s,x;t,y)$ is obtained from the static kernel $\boldsymbol\Phi_{\al,\xi}(s,x;s,y)$ in a way which is common for Markov processes on configuration spaces with determinantal dynamical correlation functions. In particular, the same extension happens in the case of the $z$-measures.
  
  \item Markov processes of \cite{Borodin2006}, as well as many dynamical determinantal models that arise in the theory of random matrices and random tilings (e.g., see \cite{nagao1998multilevel}, \cite{warren2005dyson}, \cite{johansson2006eigenvalues}, \cite{adler-nordenstam2010dyson}, \cite{johansson2002non}, \cite{johansson2005non}, \cite{borodin-gr2009q}), are closely related to orthogonal polynomials. Moreover, connections with orthogonal polynomials also arise in static Pfaffian models of random-matrix and representation-theoretic origin \cite{nagaoWadati1991}, \cite{nagaoWadati1992}, \cite{Tracy1996orthogonal}, \cite{NagaoKatoriTanemura2004PfDyn}, \cite{Katori2005PfDyn}, \cite{Nagao2007pfaffian}, \cite{Borodin2006averages}, \cite{Borodin2009correlation}, \cite{Strahov2009}. For our model there also exists a connection with orthogonal polynomials (namely, the Krawtchouk polynomials), but this connection does not help us to compute the correlation kernels as it was for the $z$-measures \cite{Borodin2006}, \cite{borodin2006meixner}.
  
  \item The expressions for our correlation kernels involve the same special functions (expressed through the Gauss hypergeometric function) which arise for the kernels in the case of the $z$-measures. These functions first appeared in the works of Vilenkin and Klimyk \cite{Vilenkin-Klimyk-DAN_UKR_1988}, \cite{Vilenkin-Klimyk-ITOGI1995-en}. In particular, certain degenerations of them lead to the classical Meixner and Krawtchouk orthogonal polynomials. 
  
  Using this fact, we are able to express our kernels directly through the corresponding kernels for the $z$-measures. These expressions seem to have no direct probabilistic meaning at the level of random point processes, but in particular they allow to study asymptotics of our kernels with the help of results of \cite{Borodin2000a}, \cite{Borodin2006}.
\end{enumerate}

\subsection*{Method}

Our technique of obtaining both static and dynamical correlation kernels in an explicit form is different from those of \cite{Borodin2000a}, \cite{borodin2006meixner}, \cite{Borodin2006}, and is based on computations in the fermionic Fock space involving so-called Kerov's operators which span a certain $\mathfrak{sl}(2, \C)$-module. Both the static and dynamical correlation kernels in our model are expressed through matrix elements related to this module. This approach is similar to the one invented by Okounkov \cite{Okounkov2001a} to calculate the (static) correlation kernel of the $z$-measures on ordinary partitions.\footnote{A possibility of use of this method in studying the dynamical model related to the $z$-measures was pointed out in \cite{Borodin2006}, and later this approach was carried out by Olshanski \cite{Olshanski-fockone}.} In computations in this paper we use the ordinary fermionic Fock space instead of the (closely related, but different) infinite wedge space of \cite{Okounkov2001a}. Moreover, our situation also requires to deal with a Clifford algebra (acting in the fermionic Fock space) of a different type. One can say that our Clifford algebra is an infinite-dimensional generalization of the Clifford algebra over an odd-dimensional quadratic space. Similar Clifford algebras were used in \cite{Date1982transformation}, \cite{Matsumoto2005}, \cite{Vuletic2007shifted}. In the latter two papers the fermionic Fock space is also used for computations of certain correlation functions. That approach is analogous to the formalism of Schur measures and Schur processes \cite{okounkov2001infinite}, \cite{okounkov2003correlation} and differs from the one used in the present paper.

\subsection*{Organization of the paper}

In \S\ref{sec:model_and_results} we give main definitions and state main results about our model. In \S\ref{sec:schur_graph_and_multiplicative_measures} we discuss combinatorial constructions from which our model arises. We also give an argument why the corresponding fixed time random point processes on $\N$ are determinantal. 

In \S\ref{sec:kerov_s_operators} we study Kerov's operators on strict partitions. These operators provide us with a convenient way of writing expectations with respect to our point processes. Such formulas are used in the computation of both static and dynamical correlation functions. In \S\ref{sec:fermionic_fock_space} we recall the formalism of the fermionic Fock space and define an action of a Clifford algebra in it. These structures are extensively used in our computations. 

In \S\ref{sec:z_measures_on_ordinary_partitions_and_an_orthonormal_basis_in_ell_2_z_} we discuss functions (matrix elements of a certain $\mathfrak{sl}(2,\C)$-module) which are used in explicit expressions for our correlation kernels. These functions are eigenfunctions of a certain second order difference operator on the lattice $\Z$. This fact allows later to interpret our kernels through orthogonal spectral projections related to that operator on the lattice. In \S\ref{sec:z_measures_on_ordinary_partitions_and_an_orthonormal_basis_in_ell_2_z_} we also recall the results of \cite{Borodin2000a}, \cite{borodin2006meixner}, and \cite{Borodin2006} about the $z$-measures on ordinary partitions which we use in the study of our model. 

In \S\ref{sec:static_correlation_functions} we prove that the static correlation functions of our Markov dynamics can be written as certain Pfaffians. We express the Pfaffian kernel through matrix elements related to Kerov's operators, and through the functions discussed in \S\ref{sec:z_measures_on_ordinary_partitions_and_an_orthonormal_basis_in_ell_2_z_}. In \S\ref{sec:static_determinantal_kernel} we write the static correlation functions as determinants and express the determinantal correlation kernel in various forms (including a so-called integrable form).

The Markov processes on strict partitions are defined in \S\ref{sec:markov_processes} in terms of their jump rates. In \S\ref{sec:dynamical_correlation_functions} we show that the dynamical correlation functions of our Markov processes have Pfaffian form, and give an explicit expression for the dynamical (Pfaffian) correlation kernel in terms of the functions discussed in \S\ref{sec:z_measures_on_ordinary_partitions_and_an_orthonormal_basis_in_ell_2_z_}. We consider the asymptotic behavior of our dynamical Pfaffian kernel under a degeneration and in two limits regimes in \S\ref{sec:dynamical_pfaffian_kernel}.

\subsection*{Acknowledgements}

The author is very grateful to Grigori Olshanski for permanent attention to the work, fruitful discussions, and access to his unpublished manuscript \cite{Olshanski-fockone}. I would also like to thank Alexei Borodin and Vadim Gorin for useful comments on my work.



\section{Model and results} 
\label{sec:model_and_results}

\subsection{Point processes on the half-lattice}\label{subsection:Point_processes}

Let us first describe the fixed time picture, that is, the random point processes on the half-lattice $\N$ that we study. They arise from a model of random strict partitions introduced in \cite{Borodin1997}.

By a \textit{strict partition} we mean a partition in which nonzero parts are distinct, that is, $\la=(\la_1>\dots>\la_{\ell(\la)}>0)$, where $\la_j\in\N$. The number $|\la|:=\la_1+\dots+\la_{\ell(\la)}$ is called the \textit{weight} of the partition, and the number of nonzero parts $\ell(\la)$ is the \textit{length} of the partition. By $\Sb_n$ denote the set of all strict partitions of weight $n=0,1,\dots$.\footnote{By agreement, the set $\Sb_0$ consists of the empty partition $\varnothing$.} Throughout the paper we identify strict partitions and corresponding shifted Young diagrams as in \cite[Ch. I, \S1, Ex. 9]{Macdonald1995}.

The description of the model of \cite{Borodin1997} starts with the \textit{Plancherel measures} on strict partitions of a fixed weight:
\begin{equation}\label{df:Pl_n}
  \mathsf{Pl}_n(\la):=\frac{2^{n-\ell(\la)}\cdot
  n!}{(\la_1!\dots\la_{\ell(\la)}!)^2}
  \prod_{1\le k<j\le\ell(\la)}
  \left( \frac{\la_k-\la_j}{\la_k+\la_j} \right)^{2},\qquad
  \la\in\Sb_n
\end{equation}
(by $\mathsf{Pl}_n(\la)$ we denote the measure of a singleton $\{ \la \}$, and the same agreement for other measures on strict partitions is used throughout the paper). The measure $\mathsf{Pl}_n$ is a probability measure on $\Sb_n$. The set $\Sb_n$ parametrizes irreducible truly projective representations of the symmetric group $\mathfrak{S}_n$ \cite{Schur1911}, \cite{Hoffman1992}, and the measures $\mathsf{Pl}_n$ on $\Sb_n$ are analogues (in the theory of projective representations of $\mathfrak{S}_n$) of the well-known Plancherel measures on ordinary partitions. The system of measures $\{ \mathsf{Pl}_n \}$ possesses the coherency property (\ref{coherency}) that has a representation-theoretic meaning, see \S\ref{subsection:coherent_systems} below. The Plancherel measures on strict partitions were studied in, e.g., \cite{Borodin1997}, \cite{IvanovNewYork3517-3530}, \cite{Ivanov2006plancherel}, \cite{Petrov2009eng}.

We consider the \textit{poissonized Plancherel measure} on the set $\Sb:=\bigsqcup_{n=0}^{\infty}\Sb_n$ of all strict partitions:
\begin{equation}\label{df:Pl_te} 
  \mathsf{Pl}_\te(\la):= 
  \frac{(\te/2)^{|\la|}e^{-\te/2}}
  {|\la|!}\mathsf{Pl}_{|\la|}(\la),\qquad\la\in\Sb,
\end{equation}
where $\te>0$ is a parameter. In other words, we mix the measures $\mathsf{Pl}_n$ on $\Sb_n$ using the Poisson distribution on the set $\Z_{\ge0}:=\left\{ 0,1,2,\dots \right\}$ of indices $n$. Regarding each strict partition $\la=(\la_1,\dots,\la_{\ell(\la)})$ as a point configuration $\left\{ \la_1,\dots,\la_{\ell(\la)} \right\}$ on $\N$ (to the empty partition $\varnothing$ corresponds the empty configuration), we view $\mathsf{Pl}_\te$ as a random point process on the half-lattice $\N$.\footnote{Throughout the paper we use this identification of strict partitions with point configurations on $\N$ whenever we speak about random point processes and their correlation functions.}

Like for the Plancherel measures on ordinary partitions \cite{Johansson1999}, \cite{okounkov2000random}, \cite{Borodin2000b} (see also \cite{baik1999distribution}, \cite{Baik1999}), the poissonization (\ref{df:Pl_te}) of the measures $\mathsf{Pl}_n$ on strict partitions leads to a determinantal point process, see Theorem \ref{thm:A} below.

Borodin \cite{Borodin1997} introduced a deformation of the measures $\mathsf{Pl}_n$ (\ref{df:Pl_n}) depending on one real parameter $\al>0$ (in \cite{Borodin1997} this parameter is denoted by $x$):
\begin{equation}\label{df:mnun}
  \mathsf{M}_{\al,n}(\la):=
  \mathsf{Pl}_n(\la)\cdot
  \frac1{\al(\al+2)\dots(\al+2n-2)}\cdot
  \prod_{k=1}^{\ell(\la)}\prod_{j=0}^{\la_k-1}
  \left( j\big(j+1\big)+\al \right).
\end{equation}
The deformations $\mathsf{M}_{\al,n}$ of the Plancherel measures $\mathsf{Pl}_n$ preserve the coherency property (\ref{coherency}). As $\al\to+\infty$, the measure $\mathsf{M}_{\al,n}$ on $\Sb_n$ converges to $\mathsf{Pl}_n$.

\begin{df}\label{df:nu(al)}
  To simplify certain formulas, instead of the parameter $\al$ we will sometimes use another parameter $\nu(\al):=\frac12\sqrt{1-4\al}$. If $0<\al\le\frac14$, then $\nu(\al)$ is real, $0\le\nu(\al)<\frac12$. If $\al>\frac14$, then $\nu(\al)$ can take arbitrary purely imaginary values. The whole picture is symmetric with respect to the replacement of $\nu(\al)$ by $-\nu(\al)$. Sometimes the argument $\al$ in $\nu(\al)$ is omitted.
\end{df}

Similarly to the poissonization of the Plancherel measures (\ref{df:Pl_te}), we consider a certain mixing of the measures $\mathsf{M}_{\al,n}$.\footnote{The mixing of Plancherel measures $\mathsf{Pl}_n$ and the measures $\mathsf{M}_{\al,n}$ over $n$ can also be viewed as a passage to the \emph{grand canonical ensemble}, cf. \cite{Vershik1996StatMech}.} But now as the mixing distribution we take a special case of the \textit{negative binomial distribution} (on $\Z_{\ge0}$)
\begin{equation}\label{df:negbinom}
  \pi_{\al,\xi}(n):=
  (1-\xi)^{\al/2}\tfrac{(\al/2)_n}{n!}\xi^n,
  \qquad n=0,1,2,\dots,
\end{equation}
with an additional parameter $\xi\in(0,1)$. Here 
\begin{equation}\label{pochhammer}
  (a)_k:=a(a+1)\dots(a+k-1)={\Gamma(a+k)}/{\Gamma(a)}
\end{equation}
is the 
\textit{Pochhammer symbol}, and $\Gamma(\cdot)$ is the Euler gamma function. As a result of the mixing, we obtain a random point process $\mathsf{M}_{\al,\xi}$ on $\N$: 
\begin{equation*}
  \mathsf{M}_{\al,\xi}(\la):=\pi_{\al,\xi}(|\la|)\cdot\mathsf{M}_{\al,|\la|}(\la),
  \qquad \la\in\Sb.
\end{equation*}
The process $\mathsf{M}_{\al,\xi}$ is supported by finite configurations. The probability of each configuration $\la=\left\{ \la_1,\dots,\la_\ell \right\}\subset\N$ has the form
\begin{equation}\label{df:mnuxi}
  \mathsf{M}_{\al,\xi}(\la)=(1-\xi)^{\al/2}\cdot
  \prod_{k=1}^{\ell}w_{\al,\xi} (\la_k)\cdot
  \prod_{1\le k<j\le \ell}\left( \frac{\la_k-\la_j}{\la_k+\la_j} \right)^{2},
\end{equation}
where 
\begin{equation}\label{df:w_al,xi}
  w_{\al,\xi}(x):=
  \frac{\xi^x\cos(\pi\nu(\al))}{2\pi}
  \frac{\Gamma(\frac12-\nu(\al)+x)
  \Gamma(\frac12+\nu(\al)+x)}{(x!)^2},
  \qquad x\in\N,
\end{equation}
and $(1-\xi)^{\al/2}$ is a normalizing constant (observe that $w_{\al,\xi}(x)>0$, $x\in\N$).

Note that in the limit
\begin{equation}\label{Pl_degen}
  \al\to+\infty\mbox{ and }\xi\to0
  \mbox{ in such a way that }\al\xi\to\te>0,
\end{equation}
the measures $\mathsf{M}_{\al,\xi}$ converge to the poissonized Plancherel measure $\mathsf{Pl}_\te$. The measure $\mathsf{Pl}_\te$ also has the form (\ref{df:mnuxi}) with $w_{\al,\xi}$ and $(1-\xi)^{\al/2}$ replaced by the limiting values $w_\te(x):=\frac{\te^x}{2(x!)^2}$ and $e^{-\te/2}$, respectively. We call the limit (\ref{Pl_degen}) the \textit{Plancherel degeneration} (see also \S\ref{subsection:whit_limit}--\S\ref{subsection:gamma_limit} for other limit regimes for the measures $\mathsf{M}_{\al,\xi}$).

Our first result is the computation of the correlation functions of the point processes $\mathsf{M}_{\al,\xi}$ and $\mathsf{Pl}_\te$. Recall that the correlation functions of a random point process on $\N$ are defined as
\begin{equation}\label{df:static_correlation_functions}
  \rho^{(n)}(x_1,\dots,x_n):=
  \mathsf{Prob}\left\{ \mbox{the random configuration
  contains $x_1,\dots,x_n$} \right\},
\end{equation}
where $n=1,2,\dots$ and $x_1,\dots,x_n$ are pairwise distinct points of $\N$. Under mild assumptions, the correlation functions determine the point process uniquely. A point process on $\N$ is called \textit{determinantal}, if there exists a function $K$ on $\N\times\N$ (called the (\textit{determinantal}) \textit{correlation kernel}) such that the correlation functions $\rho^{(n)}$, $n=1,2,\dots$, have the following form:
\begin{equation*}
  \rho^{(n)}(x_1,\dots,x_n)=
  \det\left[ K(x_k,x_j) \right]_{k,j=1}^{n}.
\end{equation*}
About determinantal point processes see, e.g., the surveys \cite{Soshnikov2000}, \cite{Peres2006determinantal}, \cite{Borodin2009}.

\begin{mainthm}\label{thm:A}
  Both the point processes $\mathsf{M}_{\al,\xi}$ and $\mathsf{Pl}_\te$ on the half-lattice $\N$ are determinantal. The correlation kernel $\mathbf{K}_{\al,\xi}$ of $\mathsf{M}_{\al,\xi}$ is expressed through the Gauss hypergeometric function (\ref{knuxi-sum}), (\ref{knuxi-integrable}). The correlation kernel $\mathbf{K}_{\te}$ of $\mathsf{Pl}_\te$ can be written in terms of the Bessel function of the first kind (\ref{kte_series}), (\ref{kte_integrable}).
\end{mainthm}

We call the kernel $\mathbf{K}_{\al,\xi}$ the \textit{hypergeometric-type} kernel. The kernel $\mathbf{K}_{\te}$ is obtained from $\mathbf{K}_{\al,\xi}$ via the Plancherel degeneration (\ref{Pl_degen}).

In (\ref{knuxi_Kzz}) we are able to express the kernel $\mathbf{K}_{\al,\xi}$ through the discrete hypergeometric kernel introduced in \cite{Borodin2000a}, \cite{borodin2006meixner}. This is done in the same spirit as is explained in \S\ref{subsection:expr_pfk_z_meas_model_and_results} below.

\subsection{Dynamical model}\label{subsection:Dynamical_model}

Let us now describe a family of continuous time Markov jump processes $(\boldsymbol\la_{\al,\xi}(t))_{t\in[0,+\infty)}$ on the space of all strict partitions $\Sb$ (which is the same as the set of all finite configurations on $\N$). These processes preserve the measures $\mathsf{M}_{\al,\xi}$. The construction of the processes $\boldsymbol\la_{\al,\xi}$ uses the same ideas as in \cite{Borodin2006}. The first key ingredient is the continuous time birth and death process on $\N$ denoted by $({\boldsymbol{n}}_{\al,\xi}(t))_{t\in[0,+\infty)}$. It depends on our parameters $\al$ and $\xi$ and has the following jump rates:
\begin{align*}
    \mathsf{Prob}\left\{ {\boldsymbol{n}}_{\al,\xi}(t+dt)=n+1\,|\, 
    {\boldsymbol{n}}_{\al,\xi}(t)=n \right\}&=
    (1-\xi)^{-1}\xi(n+\al/2)dt,
    \\
    \mathsf{Prob}\left\{ {\boldsymbol{n}}_{\al,\xi}(t+dt)=n-1\,|\, 
    {\boldsymbol{n}}_{\al,\xi}(t)=n \right\}&= 
    \displaystyle
    (1-\xi)^{-1}ndt.
\end{align*}
The process ${\boldsymbol{n}}_{\al,\xi}$ preserves the negative binomial distribution $\pi_{\al,\xi}$ (\ref{df:negbinom}) on $\N$ and is reversible with respect to it. About birth and death processes in general see, e.g., \cite{KMG57BDClassif}, \cite{KMG58Linear}.

The second key ingredient is the collection of Markov transition kernels $p^{\uparrow}_\al(n,n+1)$ from $\Sb_n$ to $\Sb_{n+1}$ and $p^\downarrow(n+1,n)$ from $\Sb_{n+1}$ to $\Sb_n$, $n=0,1,2,\dots$, such that 
\begin{equation}\label{finite_level_invariance}
  \mathsf{M}_{\al,n}\circ p^\uparrow_\al(n,n+1)=
  \mathsf{M}_{\al,n+1}
  \quad\mbox{and}\quad
  \mathsf{M}_{\al,n+1}\circ p^\downarrow(n+1,n)=
  \mathsf{M}_{\al,n}.
\end{equation}
These kernels are canonically associated with the system of measures $\{\mathsf{M}_{\al,n}\}_{n=0}^{\infty}$ (see \S\ref{subsection:coherent_systems} below, and also \cite{Borodin1997}, \cite{Petrov2009eng}), this construction follows the general formalism of Vershik and Kerov \cite{vershik1987locally}. Note that the kernels $p^{\uparrow}_\al(n,n+1)$ depend on the parameter $\al$, and the  kernels $p^\downarrow(n+1,n)$ do not depend on any parameter. The values $p^{\uparrow}_\al({n,n+1})_{\mu,\varkappa}$ and $p^\downarrow({n+1,n})_{\varkappa,\mu}$, where $\mu\in\Sb_n$ and $\varkappa\in\Sb_{n+1}$ (these are the individual transition probabilities), \textit{vanish} unless the shifted Young diagram $\varkappa$ is obtained from $\mu$ by adding a box. In other words, the transition kernels $p^{\uparrow}_\al(n,n+1)$ and $p^\downarrow(n+1,n)$ describe random procedures of adding and deleting one box, respectively.

We describe the dynamics $\boldsymbol\la_{\al,\xi}$ on strict partitions in terms of jump rates. The jumps are of two types: one can either add a box to the random shifted Young diagram, or remove a box from it (of course, the result must still be a shifted Young diagram). The events of adding and removing a box are governed by the birth and death process ${\boldsymbol{n}}_{\al,\xi}=|\boldsymbol\la_{\al,\xi}|$. Conditioned on $\boldsymbol\la_{\al,\xi}(t)=\la$ and the jump $n\to n+1$ (where $n=|\la|$) of the process ${\boldsymbol{n}}_{\al,\xi}$ during the time interval $(t,t+dt)$, the choice of the box to be added to the diagram $\la$ is made according to the probabilities $p_\al^\uparrow(n,n+1)_{\la,\varkappa}$, where $\varkappa\in\Sb_{n+1}$. Similarly, conditioned on $\boldsymbol\la_{\al,\xi}(t)=\la$ and the jump $n\to n-1$ of ${\boldsymbol{n}}_{\al,\xi}$ during $(t,t+dt)$, the choice of the box to be removed from $\la$ is made according to the probabilities $p^\downarrow({n,n-1})_{\la,\mu}$, where $\mu\in\Sb_{n-1}$.

The fact that the process ${\boldsymbol{n}}_{\al,\xi}$ preserves the mixing distribution $\pi_{\al,\xi}$ together with (\ref{finite_level_invariance}) implies that the measure $\mathsf{M}_{\al,\xi}$ on $\Sb$ is invariant for the process $\boldsymbol\la_{\al,\xi}$. Moreover, the process is reversible with respect to $\mathsf{M}_{\al,\xi}$. In this paper by $(\boldsymbol\la_{\al,\xi}(t))_{t\ge0}$ we mean the equilibrium process (that is, the process starting from the invariant distribution $\mathsf{M}_{\al,\xi}$).

\begin{rmk}
  A closely related dynamical model was considered in \cite{Petrov2009eng}, namely, a sequence of discrete time Markov chains on the sets $\Sb_n$, $n=0,1,\dots$, with transition operators $p^\uparrow_\al({n,n+1})\circ p^{\downarrow}({n+1,n})$. These chains (called the \textit{up/down Markov chains}) preserve the measures $\mathsf{M}_{\al,n}$ (\ref{df:mnun}). Similar models on ordinary partitions with various up and down transition kernels were studied in \cite{Fulman2005}, \cite{Fulman2007} (spectral properties of the chains), and \cite{Borodin2007}, \cite{Petrov2007}, \cite{Olshanski2009} (large $n$ limits).

  The $n$th up/down Markov chain on $\Sb_n$ can be reconstructed from $(\boldsymbol\la_{\al,\xi}(t))_{t\ge0}$ as follows. Condition the process $\boldsymbol\la_{\al,\xi}$ to stay in the set $\Sb_n\times\Sb_{n+1}$, and take its embedded Markov chain, that is, consider the continuous time process only at the times of jumps. We get a Markov chain on $\Sb_n\times \Sb_{n+1}$  that belongs to $\Sb_n$ at, say, even discrete time moments. Taking this chain at these moments, we reconstruct the Markov chain on $\Sb_n$ with the transition operator $p^\uparrow_\al({n,n+1})\circ p^{\downarrow}({n+1,n})$.
\end{rmk}

Let $(t_1,x_1),\dots(t_n,x_n)\in\R_{\ge0}\times\N$ be pairwise distinct space-time points. The dynamical (or space-time) correlation functions of the Markov process $\boldsymbol\la_{\al,\xi}$ are defined as 
\begin{align}
  \label{df:dynamical_correlation_functions}
  &
  \rho^{(n)}_{\al,\xi}(t_1,x_1;\dots;t_n,x_n)\\&
  \displaystyle\quad:=
  \mathsf{Prob}\left\{ 
  \mbox{the configuration $\boldsymbol\la_{\al,\xi}(t)
  $
  at 
  time $t=t_j$
  contains $x_j$, $j=1,\dots,n$}
  \right\}.
  \nonumber
\end{align}
The notion of dynamical correlation functions is a combination of finite-dimensional distributions of a stochastic dynamics and correlation functions of a random point process. Indeed, the finite-dimensional distribution of the process $\boldsymbol\la_{\al,\xi}$ at times $t_1,\dots,t_n$ (let these times be distinct for simplicity) is a probability measure on configurations on the space $\N\sqcup\dots\sqcup\N$ ($n$ copies), and $\rho^{(n)}_{\al,\xi}(t_1,x_1;\dots;t_n,x_n)$ ($t_j$'s fixed) are just the correlation functions of this measure on configurations. The dynamical correlation functions uniquely determine the dynamics $(\boldsymbol\la_{\al,\xi}(t))_{t\in[0,+\infty)}$.

The main result of the present paper is the computation of the dynamical correlation functions of $\boldsymbol\la_{\al,\xi}$. 

To formulate the result, we need a notation. By $\Z_{\ne0}$ denote the set of all nonzero integers, and for $x_1,\dots,x_n\in\N$ put, by definition, $x_{-k}:=-x_k$, $k=1,\dots,n$.
\begin{mainthm}\label{thm:B}
  The equilibrium continuous time dynamics $(\boldsymbol\la_{\al,\xi}(t))_{t\ge0}$ is Pfaffian, that is, there exists a function $\boldsymbol\Phi_{\al,\xi}(s,x;t,y)$, $x,y\in\Z$, $s\le t$, such that the dynamical correlation functions of $\boldsymbol\la_{\al,\xi}$ have the form
  \begin{equation}\label{Pfaffian_dynamical_formula}
    \rho^{(n)}_{\al,\xi}(t_1,x_1;\dots;t_n,x_n)=\Pf\left(   
    \boldsymbol\Phi_{\al,\xi}\llbracket T,X\rrbracket \right),
    \qquad 0\le t_1\le \dots\le t_n,
  \end{equation}
  where $\boldsymbol\Phi_{\al,\xi}\llbracket T,X\rrbracket$ is the $2n\times 2n$ skew-symmetric matrix with rows and columns indexed by $1,-1,\dots,n,-n$, and the $kj$-th entry in $\boldsymbol\Phi_{\al,\xi}\llbracket T,X\rrbracket$ above the main diagonal is $\boldsymbol\Phi_{\al,\xi}(t_{|k|},x_k;t_{|j|},x_j)$, where $k,j=1,-1,\dots,n,-n$ (thus, $|k|\le|j|$). The kernel $\boldsymbol\Phi_{\al,\xi}$ can be expressed through the Gauss hypergeometric function (\ref{Pfk_dyn}).
\end{mainthm}

\begin{rmk}\label{rmk:intro_xy=0}
  \textbf{1.}
  Observe that it is enough for $\boldsymbol\Phi_{\al,\xi}(s,x;t,y)$ to be defined only for $x,y\in\Z_{\ne0}$ because only such values of $\boldsymbol\Phi_{\al,\xi}(s,x;t,y)$ are used in the theorem. However, our kernel $\boldsymbol\Phi_{\al,\xi}(s,x;t,y)$ extends to $x,y\in\Z$ in a very natural way, so we always let $\boldsymbol\Phi_{\al,\xi}(s,x;t,y)$ to be defined for all $x,y\in\Z$. The same is applicable to the static Pfaffian kernel $\boldsymbol\Phi_{\al,\xi}(x,y):=\boldsymbol\Phi_{\al,\xi}(s,x;s,y)$ (see \ref{Pfk_static_xy_all} below).

\textbf{2.}
  In (\ref{Pfaffian_dynamical_formula}) we require that the time moments $t_j$ are ordered. However, Theorem \ref{thm:B} allows to compute the correlation functions $\rho^{(n)}_{\al,\xi}(t_1,x_1;\dots;t_n,x_n)$ with arbitrary order of time moments: one should simply permute the space-time points $(t_1,x_1),\dots,(t_n,x_n)$ (this does not change the value of $\rho^{(n)}_{\al,\xi}(t_1,x_1;\dots;t_n,x_n)$) in such a way that the time moments become nondecreasing, and then apply (\ref{Pfaffian_dynamical_formula}).
\end{rmk}

In \S\ref{subsection:Plancherel_degeneration_dyn} we consider the Plancherel degeneration (\ref{Pl_degen}) of the dynamical kernel $\boldsymbol\Phi_{\al,\xi}$. The resulting kernel $\boldsymbol\Phi_{\te}$ is expressed through the Bessel function of the first kind (Theorem \ref{thm:dynamical_Plancherel}). The dynamical kernel $\boldsymbol\Phi_{\te}$ has analogues related to the Plancherel measures on ordinary partitions and associated dynamics, see \cite{PhahoferSpohn2002}, \cite{Borodin2006stochastic}. The asymptotic behavior of the dynamical kernel $\boldsymbol\Phi_{\al,\xi}$ in two other limit regimes (corresponding to studying smallest resp. largest components of a random strict partition) is considered in \S\ref{subsection:whit_limit}--\S\ref{subsection:gamma_limit}. In the static case these two limit regimes were described in \cite[\S3]{Petrov2010}, where limit static (determinantal) correlation kernels were written out.

\begin{rmk}
  [Hidden determinantal structure in Pfaffian processes]
  \label{rmk:Pf_hidden_det}
  If in Theorem~\ref{thm:B} we set $t_1=\dots=t_n$, then the dynamical correlation functions turn into the (static) correlation functions of the point process $\mathsf{M}_{\al,\xi}$ on $\N$. Thus, Theorem~\ref{thm:B} implies that the point process $\mathsf{M}_{\al,\xi}$ on $\N$ is Pfaffian. To show that it is in fact determinantal requires some work (see Theorem \ref{thm:knuxi} and Proposition \ref{prop:A1_Determ_reduction} from Appendix). Thus, one can say that in the static case the determinantal structure of correlation functions is \emph{hidden} under the Pfaffian one. In particular, in this way we discover the determinantal structure of the poissonized Plancherel measure $\mathsf{Pl}_\te$ (\ref{df:Pl_te}) on strict partitions, thus strengthening a result of Matsumoto \cite{Matsumoto2005} who gave a Pfaffian formula for the correlation functions of $\mathsf{Pl}_\te$, see \S\ref{subsection:Plancherel_degeneration}.

  On the other hand, numerical computations suggest that the dynamical correlation functions of the Markov process $\boldsymbol\la_{\al,\xi}$ cannot be written as determinants. We plan to give a rigorous proof of this fact in a subsequent work.
\end{rmk} 

\subsection{Expression through the kernel for the $z$-measures}
\label{subsection:expr_pfk_z_meas_model_and_results}

The dynamical Pfaffian kernel $\boldsymbol\Phi_{\al,\xi}(s,x;t,y)$ of Theorem \ref{thm:B} has a relatively simple structure we are about to describe. 

In \S\ref{subsection:an_orthonormal_basis_phw}--\S\ref{subsection:_twisting_} we define a system of functions $\widetilde{\boldsymbol\varphi}_m(x;\al,\xi)$, where $\al,\xi$ are our parameters, and the argument $x$ and the index $m$ range over $\Z$. Each  $\widetilde{\boldsymbol\varphi}_m$ can be expressed through the Gauss hypergeometric function (\ref{wphw_al,xi}), (\ref{phw_al,xi}). For fixed $\al,\xi$, the functions $\{\widetilde{\boldsymbol\varphi}_m\}_{m\in\Z}$ form an orthonormal basis of the Hilbert space $\ell^2(\Z)$. 

Our kernel then has the form ($x,y\in\Z$ and $s\le t$):
\begin{equation}
  \label{Pfk_intro}
  \boldsymbol\Phi_{\al,\xi}(s,x;t,y)=\sum\nolimits_{m=0}^\infty
  2^{-\delta(m)}e^{-m(t-s)}
  \widetilde{\boldsymbol\varphi}_m(x)
  \widetilde{\boldsymbol\varphi}_m(-y).
\end{equation}
Here $\delta(m)=\delta_{m,0}$ is the Kronecker delta.

The functions $\widetilde{\boldsymbol\varphi}_m$ are particular cases of the functions used by Borodin and Olshanski to describe the static and dynamical correlation kernels for the $z$-measures on partitions \cite{borodin2006meixner}, \cite{Borodin2006} (see \S\ref{subsection:z-measures}--\S\ref{subsection:an_orthonormal_basis_phw}). From this fact it follows that our dynamical kernel $\boldsymbol\Phi_{\al,\xi}(s,x;t,y)$ can be expressed through the extended discrete hypergeometric kernel of \cite{Borodin2006}:
\begin{align}\nonumber
  \boldsymbol\Phi_{\al,\xi}
  (s,x;t,y)&=
  \tfrac12
  (-1)^{x\wedge0+y\vee0}
  \Big[
  e^{-\frac12(t-s)}
  \underline K_{\nu(\al)-\frac12,
  -\nu(\al)-\frac12,\xi}
  (t,x+\tfrac12;s,-y+\tfrac12)\\&\qquad+
  e^{\frac12(t-s)}
  \underline K_{\nu(\al)+\frac12,
  -\nu(\al)+\frac12,\xi}
  (t,x-\tfrac12;s,-y-\tfrac12)
  \Big],
  \label{Pfkd_Kzz}
\end{align}
where $\nu(\al)$ is given by Definition \ref{df:nu(al)}, $x,y\in\Z$, and the kernel $\underline K_{z,z',\xi}$ \cite{borodin2006meixner}, \cite{Borodin2006} lives on the lattice of (proper) half-integers $\Z':=\Z+\frac12$. Here and below for any two numbers $a$ and $b$, by $a\vee b$ and $a\wedge b$ we denote the maximum and the minimum of $a$ and $b$, respectively.

The identity (\ref{Pfkd_Kzz}) seems to be only formal and have no direct probabilistic consequences (such as probabilistic relations between our random point processes or dynamics on $\N$ and the corresponding objects for the $z$-measures). However, (\ref{Pfkd_Kzz}) can serve as a useful tool in studying the asymptotics of our dynamical kernel $\boldsymbol\Phi_{\al,\xi}(s,x;t,y)$ simply by a direct application of the results of \cite{Borodin2006} (see \S\ref{subsection:whit_limit}--\S\ref{subsection:gamma_limit} below).



\section{Schur graph and multiplicative measures} 
\label{sec:schur_graph_and_multiplicative_measures}

\subsection{Schur graph}\label{subsection:Schur_graph}

We identify \textit{strict partitions} $\la=(\la_1>\dots>\la_{\ell(\la)}>0)$, $\la_j\in\N$, and corresponding \textit{shifted Young diagrams} as in \cite[Ch. I, \S1, Example 9]{Macdonald1995}. The shifted Young diagram of the form $\la$ consists of $\ell(\la)$ rows. Each $k$th row ($k=1,\dots,\ell(\la)$) has $\la_k$ boxes, and for $j=1,\dots,\ell(\la)-1$ the first box of the $(j+1)$th row is right under the second box of the $j$th row. For example, the shifted Young diagram corresponding to the strict partition $\la=(6,4,2,1)$ looks as follows:
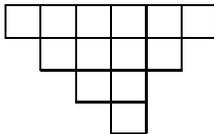
\begin{equation*}
  \begin{array}{|c|c|c|c|c|c|}
    \hline
    \ &\ &\ &\ &\ &\ \\
    \hline 
    \multicolumn{1}{c|}\ &\ &\ &\ &\ \\
    \cline{2-5}
    \multicolumn{2}{c|}\ &\ &\ \\
    \cline{3-4}
    \multicolumn{3}{c|}\ &\ \\
    \cline{4-4}
  \end{array}
\end{equation*}

Let $\mu$ and $\la$ be strict partitions. If $|\la|=|\mu|+1$ and the shifted diagram $\la$ is obtained from the shifted diagram $\mu$ by adding a box, then we write $\mu\nearrow\la$, or, equivalently, $\la\searrow\mu$. The box that is added is denoted by $\la/\mu$.

The set $\Sb=\bigsqcup_{n=0}^{\infty}\Sb_n$ of all strict partitions is equipped with a structure of a graded graph: for $\mu\in\Sb_{n-1}$ and $\la\in\Sb_n$ we draw an edge between $\mu$ and $\la$ iff $\mu\nearrow\la$. Thus, the edges in $\Sb$ are drawn only between consecutive floors. We assume the edges to be oriented from $\Sb_{n-1}$ to $\Sb_n$. In this way $\Sb$ becomes a graded graph. It is called the \textit{Schur graph}.\footnote{In \cite{Petrov2009eng} the Schur graph had multiple edges, but now it is more convenient for us to consider simple edges as in, e.g., \cite{Borodin1997}. The difference between these two choices is inessential because the down transition probabilities (\S\ref{subsection:coherent_systems}) are the same.} This graph describes the branching of (suitably normalized) irreducible truly projective characters of symmetric groups, e.g., see \cite{IvanovNewYork3517-3530}.

Let $\gdim\la$ be the total number of oriented paths in the Schur graph from the initial vertex $\varnothing$ to the vertex $\la$. This number is given by \cite[Ch. III, \S8, Example 12]{Macdonald1995}
\begin{equation}\label{gdim}
  \gdim\la=\frac{|\la|!}
  {\la_1!\dots\la_{\ell(\la)}!}\prod_{1\le k<j\le \ell(\la)}
  \frac{\la_k-\la_j}{\la_k+\la_j},\qquad \la\in\Sb.
\end{equation}
Observe that if the components of $\la$ are not distinct, then $\gdim\la$ vanishes. The numbers $\gdim\la$ satisfy the recurrence relations
\begin{equation}\label{gdim_recurrence}
  \mbox{$\displaystyle\gdim\la=
  \sum\nolimits_{\mu\colon\mu\nearrow\la}\gdim\mu$ \qquad
  for all $\la\in\Sb$,\qquad 
  $\gdim\varnothing=1$.}
\end{equation}
The number $\gdim\la$ can also be interpreted as the number of shifted standard tableaux of the form $\la$ \cite{Sag87}, \cite{worley1984theory}, and as the (suitably normalized) dimension of the irreducible truly projective representation of the symmetric group $\mathfrak{S}_{|\la|}$ corresponding to the shifted diagram $\la$ \cite{Hoffman1992}, \cite{IvanovNewYork3517-3530}.

Similarly, by $\gdim(\mu,\la)$ denote the total number of paths from $\mu$ to $\la$ in the graph $\Sb$. Clearly, $\gdim(\mu,\la)$ vanishes unless $\mu\subseteq\la$, that is, unless $\mu_k\le\la_k$ for all $k$. If $\mu\subseteq\la$, by $\la/\mu$ denote the corresponding skew shifted Young diagram, that is, the set difference of $\la$ and $\mu$. We have $\gdim\la=\gdim(\varnothing,\la)$. 

\subsection{Coherent systems of measures on the Schur graph}
\label{subsection:coherent_systems}

Following the general formalism (e.g., see \cite{Kerov1998}), one can define coherent systems of measures on the Schur graph. This definition starts from the notion of the \textit{down transition probabilities}. For $\la,\mu\in\Sb$, set
\begin{equation*}
  p^\downarrow(\la,\mu):=\left\{
  \begin{array}{ll}
    {\gdim\mu}/{\gdim\la},&\qquad\mbox{if $\mu\nearrow\la$};\\
    0,&\qquad\mbox{otherwise}.
  \end{array}
  \right.
\end{equation*}
By (\ref{gdim_recurrence}), the restriction of $p^\downarrow$ to $\Sb_{n+1}\times \Sb_n$ for all $n=0,1,\dots$ is a Markov transition kernel. We denote it by $p^\downarrow({n+1,n})=\{p^\downarrow({n+1,n})_{\la,\mu}\}_{\la\in\Sb_{n+1},\, \mu\in\Sb_n}$, and call it the \textit{down transition kernel}.

\begin{df}
  Let $M_n$ be a probability measure on $\Sb_n$, $n=0,1,\dots$. We call $\left\{ M_n \right\}$ a \textit{coherent system} of measures iff
  \begin{equation}\label{coherency}
    M_{n}(\la)=
    \sum\nolimits_{\varkappa\colon\varkappa\searrow\la}
    M_{n+1}(\varkappa)p^\downarrow(\varkappa,\la)\quad
    \mbox{for all $n$ and 
    $\la\in\Sb_n$}.
  \end{equation}
  In other words, $M_{n+1}\circ p^\downarrow({n+1,n})=M_n$ for all $n$ (cf. (\ref{finite_level_invariance})).
\end{df}

Having a \textit{nondegenerate} coherent system $\left\{ M_n \right\}$ (that is, $M_n(\la)>0$ for all $n$ and $\la\in\Sb_n$), we can define the corresponding \textit{up transition probabilities}. They depend on a choice of a coherent system. For $\la,\varkappa\in\Sb$, set
\begin{equation*}
  p^\uparrow(\la,\varkappa):=\left\{
  \begin{array}{ll}
    \displaystyle
    {M_{n+1}(\varkappa)}
    p^\downarrow(\varkappa,\la)/{M_n(\la)},&\qquad\mbox{if
    $\la\in\Sb_n$, 
    $\varkappa\in\Sb_{n+1}$ 
    and $\la\nearrow\varkappa$},\\
    0,&\qquad\mbox{otherwise}.
  \end{array}
  \right.
\end{equation*}
By (\ref{coherency}), the restriction of $p^\uparrow$ to $\Sb_{n}\times \Sb_{n+1}$ for all $n=0,1,\dots$ is a Markov transition kernel. We denote it by $p^\uparrow({n,n+1})=\{p^\uparrow({n,n+1})_{\la,\varkappa}\}_{\la\in\Sb_n,\,\varkappa\in\Sb_{n+1}}$ and call it the \textit{up transition kernel}. We have $M_n\circ p^{\uparrow}({n,n+1})=M_{n+1}$ (cf. (\ref{finite_level_invariance})).

Let us make a comment on the representation-theoretic meaning of the coherency relation (\ref{coherency}). The set of all coherent systems of measures on the Schur graph is a convex set. Its extreme points are identified with the points of the infinite-dimensional ordered simplex 
\begin{equation*}
  \Omega_+:=\Big\{ (\omega_1,\omega_2,\dots)\colon \omega_1\ge
  \omega_2\ge\dots\ge0,\
  \sum\nolimits_{k=1}^{\infty}\omega_k\le1\Big\}.
\end{equation*}
This is the so-called \textit{Martin boundary} of the Schur graph. It was first described by Nazarov \cite{Nazarov1992}. Another proof of this result can be obtained using the general methods of \cite{Kerov1998} together with the formulas of \cite{IvanovNewYork3517-3530} for dimensions of skew shifted Young diagrams.

Moreover, the following characterization of the coherent systems holds:
\begin{thm}[\cite{Nazarov1992}]\label{thm:Nazarov}
  There is a bijection between coherent systems of measures on the Schur graph $\Sb$ and Borel probability measures on the simplex $\Omega_+$.
\end{thm}

Let us explain how this bijection works. Consider embeddings $\Sb_n\to\Omega_+$, $\Sb_n\ni\la=(\la_1,\dots,\la_\ell)\mapsto \left( \frac{\la_1}n,\dots,\frac{\la_\ell}n,0,0,\dots \right)\in\Omega_+$. For a coherent system $\left\{ M_n \right\}$ on $\Sb$, the corresponding measure $P$ on $\Omega_+$ can be reconstructed as the weak limit (as $n\to+\infty$) of the push-forwards of the measures $M_n$ under these embeddings.

In the opposite direction, the points of $\Omega_+$ are in one-to-one correspondence with the indecomposable normalized projective characters of the infinite symmetric group, and any Borel probability measure $P$ on $\Omega_+$ can be viewed as a (possibly decomposable) projective character $\chi$ of $\mathfrak{S}_\infty$. This character $\chi$ can be restricted to the finite symmetric group $\mathfrak{S}_n\subset\mathfrak{S}_\infty$ (of any order $n$) and expressed as a linear combination of (suitably normalized) irreducible truly projective characters of $\mathfrak{S}_n$. These characters are parametrized by the set $\Sb_n$ \cite{Schur1911}, \cite{Hoffman1992}. The coefficients of the expansion of $\chi|_{\mathfrak{S}_n}$ are the numbers $\left\{ M_n(\la) \right\}_{\la\in\Sb_n}$, where $\left\{ M_n \right\}$ is the coherent system corresponding to the measure $P$ on $\Omega_+$ by Theorem \ref{thm:Nazarov}. The coherency condition (\ref{coherency}) for the measures $\left\{ M_n \right\}$ arises naturally in this context because the restrictions of the character $\chi$ to symmetric groups $\mathfrak{S}_n$ for different $n$ must be consistent with each other.

\subsection{Multiplicative measures}\label{subsection:multiplicative_measures}

There is a distinguished coherent system on the Schur graph, namely, the Plancherel measures $\left\{ \mathsf{Pl}_n \right\}_{n=0}^\infty$ (\ref{df:Pl_n}). This coherent system corresponds (in the sense of Theorem \ref{thm:Nazarov}) to the delta measure at the point $(0,0,\dots)\in\Omega_+$. Using the function $\gdim\la$ defined by (\ref{gdim}), one can write
\begin{equation*}
  \mathsf{Pl}_n(\la)=
  \frac{2^{n-\ell(\la)}}{n!}
  \left( \gdim\la \right)^2,\qquad
  n\in\N,\quad\la\in\Sb_n.
\end{equation*}
The Plancherel measures on strict partitions are analogues (in the theory of projective representations of symmetric groups) of the well-known Plancherel measures on ordinary partitions.

Borodin \cite{Borodin1997} has introduced a deformation $\mathsf{M}_{\al,n}$ (\ref{df:mnun}) of the measures $\mathsf{Pl}_n$ on $\Sb_n$ depending on one real parameter $\al>0$. Here let us recall the characterization of the measures $\left\{ \mathsf{M}_{\al,n} \right\}$ from \cite{Borodin1997}.
\begin{df}\label{df:multiplicative_measures}
  A system of probability measures $M_n$ on $\Sb_n$ is called \textit{multiplicative} if there exists a function $f\colon\left\{ (\i,\j)\colon \j\ge \i\ge1 \right\}\to\C$ such that 
  \begin{equation*}
    M_n(\la)=c_{n}
    \cdot\mathsf{Pl}_n(\la)\cdot
    \prod_{\square=(\i,\j)\in\la}
    f(\i,\j)
    \qquad\mbox{for all $n$ and all 
    $\la\in\Sb_n$}.
  \end{equation*}
  Here $c_n$, $n=0,1,\dots$, are normalizing constants. The product above is taken over all boxes $\square=(\i,\j)$ of the shifted Young diagram $\la$, where $\i$ and $\j$ are the row and column numbers of the box $\square$, respectively. (Note that for shifted Young diagrams we always have $\j\ge \i$.)
\end{df}
\begin{thm}[\cite{Borodin1997}]\label{thm:Borodin}
  Let $\{M_n\}$ be a nondegenerate coherent system of measures on the Schur graph. It is multiplicative iff the function $f$ has the form
  \begin{equation}\label{function_f(i,j)}
    f(\i,\j)=(\j-\i)(\j-\i+1)+\al
  \end{equation}
  for some parameter $\al\in(0,+\infty]$.
\end{thm}
If $f(\i,\j)$ is given by (\ref{function_f(i,j)}), then $c_n=\al(\al+2)\dots(\al+2n-2)$. The case $\al=+\infty$ is understood in the limit sense: $\lim\limits_{\al\to+\infty} \frac1{c_n}{\prod_{\square=(\i,\j)\in\la}f(\i,\j)}=1$ for all $n$. This case corresponds to the Plancherel measures $\left\{ \mathsf{Pl}_n \right\}$.

Recall that the number $(\j-\i)$ is called the \textit{content} of the box $\square=(\i,\j)$. For shifted Young diagrams all contents are nonnegative.

We denote by $\{\mathsf{M}_{\al,n}\}_{n=0}^\infty$ the multiplicative coherent system corresponding to the parameter $\al\in(0,+\infty)$. We see that $\mathsf{M}_{\al,n}$ converges to $\mathsf{Pl}_n$ as $\al\to+\infty$. The up transition kernel on $\Sb_{n}\times \Sb_{n+1}$ (\S\ref{subsection:coherent_systems}) for the coherent system $\left\{ \mathsf{M}_{\al,n} \right\}$ is denoted by $p^\uparrow_\al(n,n+1)$.

\begin{rmk}\label{rmk:degenerate_multiplicative}
  For certain negative values of $\al$ one can also define the measures $\mathsf{M}_{\al,n}$ by Definition \ref{df:multiplicative_measures} with $f$ given by (\ref{function_f(i,j)}). Namely, for $\al=-N(N+1)$ (where $N=1,2,\dots$) the measures $\mathsf{M}_{\al,n}$ are well-defined and nonnegative for $0\le n\le{N(N+1)}/2$. Moreover, $\mathsf{M}_{\al,n}(\la)>0$ iff $\la$ is inside the staircase-shaped shifted diagram $(N,N-1,\dots,1)$. We do not focus much on this case in the present paper.
\end{rmk} 

\subsection{Mixing of measures. Point configurations on the half-lattice}
\label{subsection:mixing_and_point_configurations}

For a set $\mathfrak{X}$, by $\mathop{\mathsf{Conf}}(\mathfrak{X})$ denote the space of all (locally finite) point configurations on $\mathfrak{X}$, and by $\mathop{\mathsf{Conf}_\mathrm{fin}}(\mathfrak{X})\subset\mathop{\mathsf{Conf}}(\mathfrak{X})$ denote the subset consisting only of finite configurations. A Borel probability measure (with respect to a certain natural topology) on $\mathop{\mathsf{Conf}}(\mathfrak{X})$ is called a random point process on $\mathfrak{X}$. If $\mathfrak{X}$ is discrete, then $\mathop{\mathsf{Conf}}(\mathfrak{X})\cong\left\{ 0,1 \right\}^{\mathfrak{X}}$, and we take the standard coordinatewise topology on $\left\{ 0,1 \right\}^{\mathfrak{X}}$ which turns it into a compact space. In more detail about random point processes, e.g., see \cite{Soshnikov2000}.

As explained in \S\ref{subsection:Point_processes}, we mix the measures $\mathsf{M}_{\al,n}$ (\ref{df:mnun}) using the negative binomial distribution $\pi_{\al,\xi}$ (\ref{df:negbinom}) on the set $\left\{ 0,1,\dots \right\}$ of indices $n$. As a result we get a probability measure $\mathsf{M}_{\al,\xi}$ (\ref{df:mnuxi}) on the set $\Sb$ of all strict partitions. Identifying strict partitions with point configurations in a natural way (\S\ref{subsection:Point_processes}), we see that the set $\Sb$ is the same as $\mathop{\mathsf{Conf}_\mathrm{fin}}(\N)$. Thus, $\mathsf{M}_{\al,\xi}$ can be viewed as a random point process on $\N$ supported by finite configurations. Under the Plancherel degeneration (\ref{Pl_degen}), the measures $\mathsf{M}_{\al,\xi}$ become the poissonized Plancherel measure $\mathsf{Pl}_\te$ (\ref{df:Pl_te}).

Let us now prove that the point processes $\mathsf{M}_{\al,\xi}$ and $\mathsf{Pl}_\te$ on $\N$ are determinantal. Observe that both these processes have a general structure described in the following definition:
\begin{df}
  Let $w$ be a nonnegative function on $\N$ such that
  \begin{equation}\label{ppfunc_finiteness_condition}
    \sum\nolimits_{x=1}^{\infty} w(x)<\infty.
  \end{equation}
  By $\mathbf{P}^{(w)}$ denote the point process on $\N$ that lives on finite configurations and assigns the probability
  \begin{equation}\label{ppfunc_definition}
    \mathbf{P}^{(w)}(\la):=\textit{const}\cdot
    \prod_{k=1}^{\ell}w(\la_k)
    \prod_{1\le k<j\le\ell}
    \left(\frac{\la_k-\la_j}{\la_k+\la_j}\right)^2
  \end{equation}
  to every configuration $\la=\left\{ \la_1,\dots,\la_\ell \right\}\subset \N$, where $\textit{const}$ is a normalizing constant.
\end{df}
The process $\mathsf{M}_{\al,\xi}$ has the form $\mathbf{P}^{(w)}$ with $w(x)=w_{\al,\xi}(x)$ given by (\ref{df:w_al,xi}), and for $\mathsf{Pl}_\te$ we have $w(x)=w_\te(x)=\frac{\te^2}{2(x!)^2}$, which is the Plancherel degeneration of $w_{\al,\xi}(x)$.

Let $\mathbf{L}^{(w)}$ be the following $\N\times\N$ matrix:
\begin{equation}\label{L_kernel_generic}
  \mathbf{L}^{(w)}(x,y):=
  \frac{2\sqrt{xy\cdot w(x)w(y)}}{x+y},\qquad x,y\in\N.
\end{equation}
The condition (\ref{ppfunc_finiteness_condition}) implies that the operator in $\ell^2(\N)$ corresponding to $\mathbf{L}^{(w)}$ is of trace class, and, in particular, the Fredholm determinant $\det(1+\mathbf{L}^{(w)})$ is well defined. 

\begin{lemma}\label{lemma:L-ensemble}
  \begin{enumerate}[\rm(1)]
    \item Let $\la=\left\{ \la_1,\dots,\la_\ell \right\}\subset\N$ be a point configuration. We have 
  \begin{equation*}
    \mathbf{P}^{(w)}(\la)=\frac{\det \mathbf{L}^{(w)}(\la)}{\det(1+\mathbf{L}^{(w)})}, 
  \end{equation*}
  where $\mathbf{L}^{(w)}(\la)$ denotes the submatrix $\left[ \mathbf{L}^{(w)}(\la_k,\la_j) \right]_{k,j=1}^{\ell}$ of $\mathbf{L}^{(w)}$.

  \item The point process $\mathbf{P}^{(w)}$ is determinantal with the correlation kernel $\mathbf{K}^{(w)}=\mathbf{L}^{(w)}(1+\mathbf{L}^{(w)})^{-1}$.
  \end{enumerate}
\end{lemma}
\begin{proof}
  The first claim directly follows from the Cauchy determinant identity \cite[Ch. I, \S4, Ex. 6]{Macdonald1995}.

  This means that the point process $\mathbf{P}^{(w)}$ is a so-called L-ensemble corresponding to the matrix $\mathbf{L}^{(w)}$ defined above (e.g., see \cite[Prop. 2.1]{Borodin2000a} or \cite[\S5]{Borodin2009}). This implies the second claim about the correlation kernel.
\end{proof}

Note that the normalizing constant in (\ref{ppfunc_definition}) is equal to $\frac1{\det(1+\mathbf{L}^{(w)})}$, so the condition (\ref{ppfunc_finiteness_condition}) is necessary for the point process $\mathbf{P}^{(w)}$ to be well defined.

\begin{rmk}
  The correlation kernel $\mathbf{K}^{(w)}$ of the process $\mathbf{P}^{(w)}$ is symmetric, because it has the form $\mathbf{K}^{(w)}=\mathbf{L}^{(w)}(1+\mathbf{L}^{(w)})^{-1}$, where $\mathbf{L}^{(w)}$ is symmetric. However, the operator of the form  $\mathbf{L}^{(w)}(1+\mathbf{L}^{(w)})^{-1}$ cannot be a projection operator in $\ell^2(\N)$. This aspect discriminates our processes from many other determinantal processes appearing in, e.g., random matrix models (see the references given in Introduction). 
  
  On the other hand, the static Pfaffian kernel in our model resembles the structure of a spectral projection operator, see Proposition \ref{prop:Pfk_proj}.
\end{rmk}

Lemma \ref{lemma:L-ensemble} implies, in particular, that our point processes $\mathsf{M}_{\al,\xi}$ and $\mathsf{Pl}_\te$ on $\N$ are determinantal. Denote their correlation kernels (given by Lemma \ref{lemma:L-ensemble}(2)) by $\mathbf{K}_{\al,\xi}$ and $\mathbf{K}_{\te}$, respectively. These kernels are symmetric. However, Lemma \ref{lemma:L-ensemble} does not give any suggestions on how to calculate them. Below we compute the correlation kernel $\mathbf{K}_{\al,\xi}$ using a fermionic Fock space technique. The kernel $\mathbf{K}_{\te}$ is obtained from $\mathbf{K}_{\al,\xi}$ via the Plancherel degeneration (\S\ref{subsection:Plancherel_degeneration}).



\section{Kerov's operators} 
\label{sec:kerov_s_operators}

\subsection{Definition, characterization and properties}\label{subsection:Kerov_operators_definition}

The main tool that we use in the present paper to compute the correlation functions of the point processes $\mathsf{M}_{\al,\xi}$ (and also of the associated dynamical models, see \S\ref{sec:markov_processes}--\S\ref{sec:dynamical_correlation_functions}) is a representation of the Lie algebra $\mathfrak{sl}(2,\C)$ in the pre-Hilbert space $\ell_{\mathrm{fin}}^2(\Sb)$ given by the so-called Kerov's operators. This approach was introduced by Okounkov \cite{Okounkov2001a} for the $z$-measures  on ordinary partitions.

By $\ell_{\mathrm{fin}}^2(\Sb)$ we denote the space of all finitely supported functions on $\Sb$ with the inner product 
\begin{equation*}
  (f,g):=\sum\nolimits_{\la\in\Sb}f(\la)g(\la).
\end{equation*}
This is a pre-Hilbert space whose Hilbert completion is the usual space $\ell^2(\Sb)$ of all functions on $\Sb$ which are square integrable with respect to the counting measure on $\Sb$. The standard orthonormal basis in $\ell^2(\Sb)$ is denoted by $\{ \un\la \}_{\la\in\Sb}$, that is,
\begin{equation}\label{basis_un_lambda}
  \un\la(\mu):=\left\{
  \begin{array}{ll}
    1,&\mbox{if $\mu=\la$};\\
    0,&\mbox{otherwise}.
  \end{array}
  \right.
\end{equation}

\begin{df}\label{df:KO}
  The Kerov's operators in $\ell_{\mathrm{fin}}^2(\Sb)$ depend on our parameter $\al>0$ and are defined as
  \begin{align}
    \label{KO}\nonumber
      &\Uf\un\la:=\displaystyle
      \sum\nolimits_{\varkappa\colon\varkappa\searrow\la}
      2^{-\delta(\j-\i)/2}
      \sqrt{(\j-\i)(\j-\i+1)+\al}
      \cdot\un\varkappa,&(\i,\j)=\varkappa/\la;\\&
      \Df\un\la:=\displaystyle
      \sum\nolimits_{\mu\colon\mu\nearrow\la}
      2^{-\delta(\j-\i)/2}
      \sqrt{(\j-\i)(\j-\i+1)+\al}\cdot\un\mu,&(\i,\j)
      =\la/\mu;\\&\nonumber
      \Hf\un\la:=\left( 2|\la|+\tfrac\al2 \right)\un\la.&
  \end{align}
  We denote a box by $(\i,\j)$ iff its row number is $\i$ and its column number is $\j$. 
\end{df}
The Kerov's operators are closely related to the measures $\mathsf{M}_{\al,n}$ (\ref{df:mnun}) on $\Sb_n$. Namely, it is clear that
\begin{equation*}
    (\Uf^n\uno\varnothing,\uno\la)=
    (\Df^n\uno\la,\uno\varnothing)=
    \gdim\la\cdot 2^{-\ell(\la)/2}\prod_{\square=(\i,\j)\in\la}
    \sqrt{(\j-\i)(\j-\i+1)+\al}
\end{equation*}
for all $n$ and $\la\in\Sb_n$, so
\begin{equation}\label{mnun_through_KO}
  \mathsf{M}_{\al,n}(\la)=Z_n^{-1}
  (\Uf^n\uno\varnothing,\uno\la)
  (\Df^n\uno\la,\uno\varnothing),
\end{equation}
where $Z_n$ is a normalizing constant. See also the end of this subsection for more connections between the Kerov's operators and the measures $\mathsf{M}_{\al,n}$.

The Kerov's operators (\ref{KO}) satisfy the following:

\begin{properties}\label{four_properties}
\begin{enumerate}[{\bf{}1.\/}]
\item
The map 
\begin{equation}\label{slf_representation}
  U:=\left[
  \begin{array}{cc}
    0&1\\
    0&0
  \end{array}
  \right]\to \Uf,\qquad
  D:=\left[
  \begin{array}{cc}
    0&0\\
    -1&0
  \end{array}
  \right]\to \Df,\qquad
  H:=\left[
  \begin{array}{cc}
    1&0\\
    0&-1
  \end{array}
  \right]\to \Hf
\end{equation}
defines a representation of the Lie algebra $\mathfrak{sl}(2,\C)$ in $\ell_{\mathrm{fin}}^2(\Sb)$. That is, the operators $\Uf$, $\Df$, and $\Hf$ satisfy the commutation relations
\begin{equation}\label{KO_commutation_relations}
  \left[ \Hf,\Uf \right]=
  2\Uf,\qquad \left[ \Hf,\Df \right]=-2\Df,\qquad
  \left[ \Df,\Uf \right]=\Hf.
\end{equation}

\item
The operators $\Uf$ and $\Df$ are adjoint to each other in the space $\ell_{\mathrm{fin}}^2(\Sb)$.

\item
For any $\la\in\Sb$, the vector $\Uf\un\la$ is a linear combination of vectors $\un\varkappa$, where $\varkappa\searrow\la$, and the coefficient of $\un\varkappa$ depends only on the box $\varkappa/\la$ (through its row and column numbers). Likewise, the vector $\Df\un\la$ is a linear combination of vectors $\un\mu$, where $\mu\nearrow\la$, and the coefficient of $\un\mu$ depends only on the box $\la/\mu$.

\item
Each basis vector $\un\la$, $\la\in\Sb$, is an eigenvector of the operator $\Hf$, and the eigenvalue of $\un\la$ depends only on $|\la|$.
\end{enumerate}
\end{properties}

The only property above that is not obvious is the first one:
\begin{lemma}\label{lemma:KO_commutation_relations}
  The Kerov's operators $\Uf$, $\Df$, and $\Hf$ (\ref{KO}) satisfy the commutation relations (\ref{KO_commutation_relations}).
\end{lemma}
\begin{proof}
  Denote 
  \begin{equation}\label{ufunc}
    q_\al(\square)=
    q_\al(\i,\j):=
    2^{-\delta(\j-\i)/2}\sqrt{(\j-\i)(\j-\i+1)+\al},
  \end{equation}
  where $\square=(\i,\j)$. The relation $\left[ \Hf,\Uf \right]=2\Uf$ is straightforward:
  \begin{equation*}
    \begin{array}{r}
      \displaystyle
      \left[ \Hf,\Uf \right]\un\la=
      \Hf\sum\nolimits_{\varkappa\searrow\la}
      q_\al(\varkappa/\la)\un\varkappa-
      \left( 2|\la|+\frac\al2 \right)
      \sum\nolimits_{\varkappa\searrow\la}
      q_\al(\varkappa/\la)\un\varkappa
      \qquad\qquad\qquad\\=
      2\left( |\la|+2-|\la| \right)\Uf\un\la=
      2\Uf\un\la,
    \end{array}
  \end{equation*}
  and the same for the relation $\left[ \Hf,\Df \right]=-2\Df$.

  It remains to prove that $\left[ \Df,\Uf \right]=\Hf$. The vector $\left[ \Df,\Uf \right]\un\la$ has the form
  \begin{equation}\label{KO_commutation_relations_proof}
    \sum\nolimits_{\varkappa\searrow\la}
    \sum\nolimits_{\rho\nearrow\varkappa}
    q_\al(\varkappa/\la)q_\al(\varkappa/\rho)
    \un\rho-
    \sum\nolimits_{\mu\nearrow\la}
    \sum\nolimits_{\rho\searrow\mu}
    q_\al(\la/\mu)q_\al(\rho/\mu)\un\rho.
  \end{equation}
  This is a linear combination of vectors $\un\rho$, where $\rho\in\Sb_n$ and either $\rho=\la$, or $\rho=\la+\square_1-\square_2$ for some boxes $\square_1\ne\square_2$. In the second case the coefficient by the vector $\un\rho$ with $\rho=\la+\square_1-\square_2$ is
  \begin{equation*}
    q_\al(\square_1)q_\al(\square_2)-
    q_\al(\square_2)q_\al(\square_1)=0.
  \end{equation*}
  Thus, in (\ref{KO_commutation_relations_proof}) it remains to consider only the terms with $\rho=\la$. Therefore, one must establish the combinatorial identity
  \begin{equation*}
    \sum\nolimits_{\varkappa\colon\varkappa\searrow\la}
    q_\al(\varkappa/\la)^2-
    \sum\nolimits_{\mu\colon\mu\nearrow\la}
    q_\al(\la/\mu)^2= 2|\la|+\tfrac\al2\qquad
    \mbox{for all $\la\in\Sb$}.
  \end{equation*}
  The proof of this identity (using Kerov's interlacing coordinates of shifted Young diagrams) is essentially contained in \S3.1 of the paper \cite{Petrov2009eng} (the arXiv version).
\end{proof}

In fact, the Kerov's operators (\ref{KO}) are completely characterized by the above four properties:

\begin{prop}\label{prop:KO_characterization}
  If three operators $\Uf$, $\Df$, and $\Hf$ in the space $\ell_{\mathrm{fin}}^2(\Sb)$ satisfy the four properties \ref{four_properties}, then they have the form (\ref{KO}) with some parameter $\al\in\C$.
\end{prop}

By agreement, for arbitrary complex $\al$, in the definition of $\Uf$ and $\Df$ in (\ref{KO}) we take the same branches of the square roots $\sqrt{\al+c(c+1)}$, $c=0,1,2,\dots$. In fact, it is the square of the function $q_\al(\cdot,\cdot)$ (\ref{ufunc}) that  really plays the role in the definition of the Kerov's operators.

\begin{proof}
  By properties 2 and 3, there exists a (complex-valued) function $q$ on the set of all boxes, that is, on the set $\left\{ (\i,\j)\colon \j\ge \i\ge1 \right\}$ (where $\i$ and $\j$ are the row and column numbers of the box, respectively), such that the operators $\Uf$ and $\Df$ have the form
  \begin{equation*}
    \Uf\un\la=\sum\nolimits_{\varkappa\searrow\la}
    q(\varkappa/\la)\un\varkappa,\qquad
    \Df\un\la=\sum\nolimits_{\mu\nearrow\la}
    q(\la/\mu)\un\mu.
  \end{equation*}
  
  By property 4, for all $n=0,1,\dots$ and all $\la\in\Sb_n$ we have $\Hf\un\la=h_n\un\la$ for some (complex) numbers $h_n$. Using the commutation relation $\left[ \Hf,\Uf \right]=2\Uf$ (property 1), it is easy to see that $h_{n+1}=h_n+2$ for all $n=0,1,\dots$. Set $\al:=2h_0$ (this is some complex parameter). Thus, we have $h_n=2n+\frac\al2$.

  The function $q(\cdot,\cdot)^2$ can be found by applying the commutation relation $\left[ \Df,\Uf \right]=\Hf$ (property 1) to various vectors $\un\la$, $\la\in\Sb$. First, applying this relation to $\un\varnothing$ and $\un\square$ (the latter vector corresponds to the one-box shifted diagram), we get
  \begin{equation*}
    q(1,1)^2=\tfrac\al2\qquad\mbox{and}\qquad
    q(1,2)^2=2+\al.
  \end{equation*}

  Next, apply $\left[ \Df,\Uf \right]=\Hf$ to $\un{(n)}$ for $n\ge2$:
  \begin{equation*}
    q(1,n+1)^2-q(1,n)^2+q(2,2)^2=2n+\tfrac\al2.
  \end{equation*}
  Solving this recurrence and taking into account the initial value $q(1,2)^2$, we get
  \begin{equation*}
    q(1,n)^2
    =-(n-2)q(2,2)^2+n\left( n-1+\tfrac\al2 \right),
    \qquad n=2,3,\dots.
  \end{equation*}

  To find $q(2,2)^2$, we apply $\left[ \Df,\Uf \right]=\Hf$ to $\un{(2,1)}$:
  \begin{equation*}
    q(1,3)^2-q(2,2)^2=
    6+\tfrac\al2\qquad\Rightarrow\qquad
    q(2,2)^2=\tfrac\al2,
  \end{equation*}
  and so 
  \begin{equation*}
    q(1,n)^2=n(n-1)+\al,\qquad n=2,3,\dots.
  \end{equation*}

  To find $q(n,n)^2$ for $n\ge3$, use the vector $\un\la$ with $\la=(n,n-1,\dots,1)$:
  \begin{equation*}
    q(1,n+1)^2-q(n,n)^2=
    2\cdot\frac{n(n+1)}2+\frac\al2
    \qquad\Rightarrow\qquad
    q(n,n)^2=\frac\al2. 
  \end{equation*}

  Finally, to find $q(\i,\j)^2$ for arbitrary $\j>\i>1$ (these are the remaining unknown values of $q(\cdot,\cdot)^2$), we apply the relation $\left[ \Df,\Uf \right]=\Hf$ to the vector $\un\la$ with  $\la=(\j,\j-1,\dots,\j-\i+1)$:
  \begin{equation*}
    q(1,\j+1)^2+q(\i+1,\i+1)^2-q(\i,\j)^2=
    2\cdot\frac{\i(2\j-\i+1)}2+\frac\al2.
  \end{equation*}
  We thus have $q(\i,\j)^2=(\j-\i)(\j-\i+1)+\al$ for all $\j>\i>1$.

  Putting all together, we see that the function $q$ is identical to the function $q_\al$ defined by (\ref{ufunc}) (but note the remark after the formulation of the present proposition). The fact that the commutation relation $\left[ \Df,\Uf \right]\un\la=\Hf\un\la$ (with the above choice of $q=q_\al$) holds for all shifted Young diagrams $\la\in\Sb$ follows from Lemma \ref{lemma:KO_commutation_relations}. This concludes the proof.
\end{proof}

\begin{rmk}\label{rmk:Kerov_gauge}
  One can also prove a statement analogous to Proposition \ref{prop:KO_characterization} without property \ref{four_properties}.2. The operator $\Hf$ is still defined uniquely up to a parameter $\al\in\C$. The two other operators are equal to $\Uf$ and $\Df$ (\ref{KO}) up to a ``gauge transformation'' that is written in terms of matrix elements in the basis $\left\{ \un\la \right\}_{\la\in\Sb}$ as:
  \begin{equation*}
    (\Uf\un\la,\un\varkappa)\mapsto
    f(\varkappa/\la)\cdot
    (\Uf\un\la,\un\varkappa),\qquad
    (\Df\uno\la,\uno\mu)\mapsto
    {f(\la/\mu)}^{-1}\cdot(\Df\uno\la,\uno\mu),
  \end{equation*}
  where $f$ is some nonzero function on the set of boxes. One possible choice of such operators (which are not adjoint to each other) is (\ref{KO_twisted}) below.
\end{rmk}
\begin{rmk}
  A statement parallel to Proposition \ref{prop:KO_characterization} can be proved for the Young graph whose vertices are ordinary partitions. As a result we will get operators similar to those considered in \cite{Okounkov2001a}. This allows one to give a purely combinatorial characterization of the $z$-measures on the Young graph. Another characterization of the $z$-measures similar to Theorem \ref{thm:Borodin} (of \cite{Borodin1997}) is given in \cite{Rozhkovskaya1997multiplicative}.
\end{rmk}

Let us give some remarks on how deep is the connection between the measures $\mathsf{M}_{\al,n}$ (\ref{df:mnun}) and the Kerov's operators (\ref{KO}). This discussion is also applicable to the $z$-measures on the Young graph.

First, using commutation relations (\ref{KO_commutation_relations}) for the Kerov's operators, one can compute the normalizing constants $Z_n$ in (\ref{mnun_through_KO}) that are defined as
\begin{equation*}
  Z_n=\sum\nolimits_{\la\in\Sb_n}
  (\Uf^n\uno\varnothing,\uno\la)
  (\Df^n\uno\la,\uno\varnothing).
\end{equation*}
In the above sum the parameter $\al$ is hidden in the definition of the operators $\Uf$ and $\Df$ (\ref{KO}), and one can assume $\al$ to be an arbitrary complex number. Write
\begin{equation*}
  \begin{array}{l}\displaystyle
    Z_n=\sum\nolimits_{\la\in\Sb_n}
    (\Uf^n\uno\varnothing,\uno\la)
    (\Df^n\uno\la,\uno\varnothing)
    =\Big( \Df^n\sum\nolimits_{\la\in\Sb_n}
    (\Uf^n\uno\varnothing,\uno\la)
    \cdot\uno\la,\uno\varnothing\Big)
    =
    (\Df^n\Uf^n\un\varnothing,\un\varnothing).
  \end{array}
\end{equation*}
By the commutation relations (\ref{KO_commutation_relations}),
\begin{equation*}
  \Df\Uf^n=\Uf^n\Df+\sum\nolimits_{k=0}^{n-1}\Uf^{n-k-1}\Hf\Uf^k.
\end{equation*}
Using the fact that $\Df\un\varnothing=0$, we get
\begin{equation*}
  \begin{array}{l}\displaystyle
    Z_n=\sum_{k=0}^{n-1}
    (\Df^{n-1}\Uf^{n-k-1}\Hf\Uf^k\un\varnothing,\un\varnothing)
    =Z_{n-1}
    \sum_{k=0}^{n-1}\left(2k+\frac\al2\right)=
    n\left( n-1+\frac\al2 \right)Z_{n-1}.
  \end{array}
\end{equation*}
Taking into account the initial value $Z_0=(\Uf^0\Df^0\un\varnothing,\un\varnothing)=1$, we see that $Z_n=n!(\al/2)_n$.

Thus, the (complex-valued) measures $\mathsf{M}_{\al,n}$ are well-defined by (\ref{mnun_through_KO}) for all $\al\in\C\setminus\left\{ 0,-2,-4,\dots \right\}$ because for such $\al$ the normalizing constants $Z_n$ are nonzero for all $n$. Moreover, under this assumption the measures $\mathsf{M}_{\al,n}$ are nondegenerate in the sense that $\mathsf{M}_{\al,n}(\la)\ne0$ for all $n$ and all $\la\in\Sb_n$. Many formulas in the present paper hold in a purely algebraic sense for $\al\in\C\setminus\left\{ 0,-2,-4,\dots \right\}$.

Now let us present an alternative proof of the coherency condition (\ref{coherency}) of the measures $\left\{ \mathsf{M}_{\al,n} \right\}$. Another proof can be found in \cite{Borodin1997}. Here we also assume that $\al\in\C\setminus\left\{ 0,-2,-4,\dots \right\}$. Consider the following operators which are slightly different from
$\Uf$ and $\Df$:
\begin{equation}\label{KO_twisted}
  \begin{array}{lcll}\displaystyle
    \hat \Uf\un\la&:=&\displaystyle
    \sum\nolimits_{\varkappa\colon\varkappa\searrow\la}
    2^{-\delta(\j-\i)}
    \big((\j-\i)(\j-\i+1)+\al\big)
    \cdot\un\varkappa,&\quad(\i,\j)=
    \varkappa/\la;\\\rule{0pt}{16pt}
    \hat \Df\un\la&:=&\displaystyle
    \sum\nolimits_{\mu\colon\mu\nearrow\la}\un\mu.
  \end{array}
\end{equation}
Clearly, $[\hat\Df,\hat \Uf]=\Hf$ (see also Remark \ref{rmk:Kerov_gauge}), and $(\hat\Df\uno\la,\uno\varnothing)=\gdim\la$. Moreover,
\begin{equation}\label{mnun_through_KO_hat}
  \mathsf{M}_{\al,n}(\la)=\frac1{Z_n}(\hat\Uf^n\uno\varnothing,\uno\la)
  (\hat\Df^n\uno\la,\uno\varnothing)
  \qquad\mbox{for all $n$ and $\la\in\Sb_n$}
\end{equation}
(here $Z_n=n!(\al/2)_n$ is the same as in (\ref{mnun_through_KO})). Fix $n=1,2,\dots$ and $\mu\in\Sb_{n-1}$. Write
\begin{equation*}
  \begin{array}{l}
    \displaystyle
    \sum_{\la\colon\la\searrow\mu}
    \frac1{\gdim\la}{(\hat \Uf^n\uno\varnothing,\uno\la)
    (\hat \Df^n\uno\la,\uno\varnothing)}
    =
    \sum_{\la\colon\la\searrow\mu}
    (\hat \Uf^n\uno\varnothing,\uno\la)=
    (\hat \Uf^n\unt\varnothing,\hat \Df^*\unt\mu)=
    (\hat \Df\hat \Uf^n\unt\varnothing,\unt\mu)
    \\\rule{0pt}{14pt}
    \displaystyle\qquad
    =
    \sum_{k=0}^{n-1}
    (\hat \Uf^{n-k-1}\Hf\hat \Uf^k\unt\varnothing,\unt\mu)
    =
    \frac{Z_n}{Z_{n-1}}
    (\hat \Uf^{n-1}\unt\varnothing,\unt\mu)
    \\\rule{0pt}{18pt}
    \displaystyle\qquad=
    \frac{Z_n}{Z_{n-1}}\cdot\frac1
    {\gdim\mu}{(\hat \Uf^{n-1}\unt\varnothing,\unt\mu)
    (\hat\Df^{n-1}\unt\mu,\unt\varnothing)}.
  \end{array}
\end{equation*}
The identity that we have obtained is clearly equivalent to the coherency condition (\ref{coherency}) for the measures $\left\{ \mathsf{M}_{\al,n} \right\}$ written in the form (\ref{mnun_through_KO_hat}).
  
\begin{rmk}
  The Kerov's operators (\ref{KO}) with certain minor modifications fall into the framework of the paper by Fulman \cite{Fulman2007}. Namely, consider the operators $U_n\colon\C\Sb_n\to\C\Sb_{n+1}$ and $D_n\colon \C\Sb_{n}\to\C\Sb_{n-1}$ defined by $U_n\un\la:=\frac1{n+\al/2}\Uf\un\la$ and $D_n\un\la:=\Df\un\la$ (where $\la\in\Sb_n$). Then the operators $U_n$ and $D_n$ satisfy the commutation relations in the form of \cite[(1.1)]{Fulman2007}: $D_{n+1}U_n=a_nU_{n-1}D_n+b_nI_n$, where $I_n\colon \C\Sb_n\to\C\Sb_n$ is the identity operator and $a_n=1-\frac1{n+\al/2}$, $b_n=1+\frac{n}{n+\al/2}$. Another similar operators were earlier considered by Stanley \cite{stanley1988differential} and Fomin \cite{Fomin1994duality}.
\end{rmk}

\subsection{Kerov's operators and averages with respect
to our point processes}
\label{subsection:expectation_formula}

The probability assigned to a strict partition $\la$ by the measure $\mathsf{M}_{\al,\xi}$ (\ref{df:mnuxi}) (which is a mixture of the measures $\mathsf{M}_{\al,n}$) can be written for small enough $\xi$ as follows:
\begin{equation*}
  \mathsf{M}_{\al,\xi}(\la)=(1-\xi)^{\al/2}
  ( e^{\sqrt\xi \Uf}\un\varnothing,\un\la )
  ( e^{\sqrt\xi \Df}\un\la,\un\varnothing ).
\end{equation*}
Here $e^{\sqrt\xi \Df}\un\la$ is clearly an element of $\ell_{\mathrm{fin}}^2(\Sb)$. The fact that the vector $e^{\sqrt\xi\Uf}\un\la$ belongs to $\ell^2(\Sb)$ (for small enough $\xi$) requires a justification (see the proof of Proposition \ref{prop:integrability_ell2_fin}), because the operator $\Uf$ in $\ell^2(\Sb)$ is unbounded. This makes the above formula for $\mathsf{M}_{\al,\xi}(\la)$ not very convenient for taking averages with respect to the measure $\mathsf{M}_{\al,\xi}$.\footnote{Static correlation functions are readily expressed as averages with respect to $\mathsf{M}_{\al,\xi}$ (see (\ref{corr_f_averages}) below), so we need good tools for computing such averages.} In this subsection we overcome this difficulty and give a convenient way of writing expectations with respect to $\mathsf{M}_{\al,\xi}$. Our approach here is similar to that of Olshanski \cite{Olshanski-fockone} and is also based on the ideas of \cite{Okounkov2001a}.

Recall that the Kerov's operators $\Uf$, $\Df$, and $\Hf$ (\ref{KO}) define (via the map (\ref{slf_representation})) a representation of the complex Lie algebra $\mathfrak{sl}(2,\C)$ in the (complex) pre-Hilbert space $\ell_{\mathrm{fin}}^2(\Sb)$. Consider the real form $\mathfrak{su}(1,1)\subset \mathfrak{sl}(2,\C)$ spanned by the matrices $U-D$, $i(U+D)$, and $iH$ (here $i=\sqrt{-1}$). The corresponding operators $\Uf-\Df$, $i(\Uf+\Df)$, and $i\Hf$ act skew-symmetrically in $\ell_{\mathrm{fin}}^2(\Sb)$. Now we prove that the representation of the Lie algebra $\mathfrak{su}(1,1)$ can be lifted to a representation of a corresponding Lie group:

\begin{prop}\label{prop:integrability_ell2_fin}
  All vectors of the space $\ell_{\mathrm{fin}}^2(\Sb)$ are analytic for the described above action of the Lie algebra $\mathfrak{su}(1,1)$ in $\ell_{\mathrm{fin}}^2(\Sb)$. Consequently, this action of $\mathfrak{su}(1,1)$ gives rise to a unitary representation of the universal covering group $SU(1,1)^\sim$ in the Hilbert space $\ell^2(\Sb)$.
\end{prop}
\begin{proof}
  Recall \cite{nelson1959analytic} that a vector $h$ is analytic for an operator $A$ if the power series 
  \begin{equation*}
    \sum_{n=0}^{\infty}\frac{\|A^nh\|}{n!}s^n
  \end{equation*}
  in $s$ has a positive radius of convergence.

  We can use Lemma 9.1 in \cite{nelson1959analytic} that guarantees the existence of the desired unitary representation of $SU(1,1)^\sim$ in $\ell^2(\Sb)$ if we first prove that for some constant $s_0>0$ we have
  \begin{equation}\label{Nelson_estimate}
    \|A_{i_1}\dots A_{i_n}h\|\le \frac{n!}{s_0^n}
  \end{equation}
  for any $h\in\ell_{\mathrm{fin}}^2(\Sb)$, all sufficiently large $n$ (the bound on $n$ depends on $h$), and any indices $i_1,\dots,i_n$ taking values $1,2,3$, where $A_1=\Uf-\Df$, $A_2=i(\Uf+\Df)$, and $A_3=i\Hf$. Note that this in fact implies that any vector in $\ell_{\mathrm{fin}}^2(\Sb)$ is analytic for the action of $\mathfrak{su}(1,1)$.

  It suffices to prove the estimate (\ref{Nelson_estimate}) for $\hat A_1:=\Uf$, $\hat A_2:=\Df$, and $\hat A_3:=\Hf$, this can only affect the value of the constant $s_0$. Moreover, we can consider only the cases when $h=\un\varkappa$ for an arbitrary $\varkappa\in\Sb$. Because all the matrix elements of the operators $\Uf$, $\Df$, and $\Hf$ are nonnegative in the standard basis $\left\{ \un\la \right\}_{\la\in\Sb}$, we have
  \begin{equation*}
    \|\hat A_{i_1}\dots \hat A_{i_n}\un\varkappa\|\le 
    \|(\Uf+\Df+\Hf)^n\un\varkappa\|.
  \end{equation*}
  The desired estimate would follow if we show that the power series expansion of $\exp\left( s(\Uf+\Df+\Hf) \right)\un\varkappa$ converges for some small enough $s>0$. For matrices in $SL(2,\C)$ (see (\ref{slf_representation})) we have
  \begin{equation*}
    \exp({s(U+D+H)})=\exp\left( \frac{s}{1-s}U \right)
    \exp\left( \log\left( \frac1{1-s} \right)H \right)
    \exp\left( \frac{s}{1-s}D \right).
  \end{equation*}
  Thus, the power series expansion of $\exp\left( s(\Uf+\Df+\Hf) \right)\un\varkappa$ is the same as that of
  \begin{equation*}
    \exp\left( \frac{s}{1-s}\Uf \right)
    \exp\left( \log\left( \frac1{1-s} \right)\Hf \right)
    \exp\left( \frac{s}{1-s}\Df \right)\un\varkappa.
  \end{equation*}
  Since the operator $\Df$ is locally nilpotent and the  operator $\Hf$ acts on each $\un\la$ as multiplication by $(2|\la|+\al/2)$, to obtain the desired estimate (\ref{Nelson_estimate}) it remains to show that the series 
  \begin{equation*}
    \sum\nolimits_{n=0}^{\infty}\|\Uf^n\un\mu\|\frac{s^n}{n!}
  \end{equation*}
  converges for all $\mu\in\Sb$ for sufficiently small $s>0$ (the bound on $s$ must not depend on $\mu$). Let us fix $\mu$ with $|\mu|=k$. We can write by definition of $\Uf$:
  \begin{equation*}
    \begin{array}{l}
      \displaystyle
      \|\Uf^n\un\mu\|^2=\sum\nolimits_{\la\in\Sb_{k+n}}
      (\Uf^n\unt\mu,\unt\la)^2=
      \sum\nolimits_{\la\in\Sb_{k+n}}\gdim(\mu,\la)^2
      \prod\nolimits_{\square\in\la/\mu}q_\al(\square)^2,
    \end{array}
  \end{equation*}
  where $q_\al$ is defined by (\ref{ufunc}). Here the product is taken over all boxes of the skew shifted diagram $\la/\mu$ (see the end of \S\ref{subsection:Schur_graph}). Since $\gdim(\mu,\la)\le\gdim\la$, we can estimate
  \begin{equation*}
    \begin{array}{l}
      \displaystyle
      \|\Uf^n\un\mu\|^2\le
      \Big(\prod\nolimits_{\square\in\mu}
      q_\al(\square)^{-2}\Big)\cdot
      \sum\nolimits_{\la\in\Sb_{k+n}}
      (\gdim\la)^2
      \prod\nolimits_{\square\in\la}q_\al(\square)^2
      \\\rule{0pt}{16pt}
      \displaystyle\qquad\qquad=
      \Big(\prod\nolimits_{\square\in\mu}
      q_\al(\square)^{-2}\Big)\cdot
      \sum\nolimits_{\la\in\Sb_{n+k}}
      (\Uf^n\uno\varnothing,\uno\la)^2=
      Z_{n+k}\cdot
      \Big(\prod\nolimits_{\square\in\mu}
      q_\al(\square)^{-2}\Big).
    \end{array}
  \end{equation*}
  The factor $\prod_{\square\in\mu}q_\al(\square)^{-2}$ is just a constant depending on $\mu$, and the normalizing constants $Z_n=n!(\al/2)_n$ were computed in the previous subsection. Putting all together, we get
  \begin{equation}\label{U_series}
    \sum_{n=0}^{\infty}\frac{s^n}{n!}\|\Uf^n\un\mu\|
    \le
    \Big(\prod_{\square\in\mu}
    q_\al(\square)^{-2}\Big)^{\frac12}\cdot
    \sum_{n=0}^{\infty}
    \frac{s^n}{n!}\sqrt{(n+k)!(\al/2)_{n+k}}.
  \end{equation}
  Using \cite[1.18.(5)]{Erdelyi1953}, we see that
  \begin{equation*}
    \frac{\sqrt{(n+k)!(\al/2)_{n+k}}}{n!}\sim
    \sqrt{\frac{n^{2k+\al/2-1}}{\Gamma(\al/2)}},
  \end{equation*}
  so the series (\ref{U_series}) converges for small enough $s>0$. This concludes the proof of the proposition.
\end{proof} 

To formulate the central statement of this section, we need some preparation. By $G_\xi$ denote the matrix
\begin{equation}\label{Gxi_df}
  G_\xi:=\left[
  \begin{array}{cc}
    \frac{1}{\sqrt{1-\xi}}&
    \frac{\sqrt\xi}{\sqrt{1-\xi}}\\\rule{0pt}{18pt}
    \frac{\sqrt\xi}{\sqrt{1-\xi}}&
    \frac{1}{\sqrt{1-\xi}}
  \end{array}
  \right]=
  \left( \frac{1+\sqrt\xi}{1-\sqrt\xi} \right)^{\frac{U-D}2}
  \in SU(1,1),\qquad 0\le \xi<1.
\end{equation}
Clearly, $(G_\xi)_{0\le\xi<1}$ is a continuous curve in $SU(1,1)$ starting at the unity. By $(\widetilde {{G}}_\xi)_{0\le\xi<1}$ denote the lifting of this curve to $SU(1,1)^\sim$, again starting at the unity. The unitary operators in $\ell^2(\Sb)$ corresponding (by Proposition \ref{prop:integrability_ell2_fin}) to $\widetilde {{G}}_\xi$ are denoted by $\widetilde {\mathsf{G}}_\xi$.

The next thing we need is the weighted $\ell^2$ space $\ell^2(\Sb,\mathsf{M}_{\al,\xi})$ --- the space of functions on $\Sb$ that are square summable with the weight $\mathsf{M}_{\al,\xi}$. This is a Hilbert space with the inner product
\begin{equation*}
  (f,g)_{\mathsf{M}_{\al,\xi}}:=
  \sum\nolimits_{\la\in\Sb}f(\la)g(\la)\mathsf{M}_{\al,\xi}(\la).
\end{equation*}
There is an isometry map $I_{\al,\xi}$ from $\ell^2(\Sb,\mathsf{M}_{\al,\xi})$ to $\ell^2(\Sb)$:
\begin{equation}\label{istry}
  I_{\al,\xi}:=\mbox{multiplication of $f\in \ell^2(\Sb,\mathsf{M}_{\al,\xi})$ by 
  the function $\la\mapsto 
  \sqrt{\mathsf{M}_{\al,\xi}(\la)}$}.
\end{equation}
The standard orthonormal basis $\left\{ \un\la \right\}_{\la\in\Sb}$ (\ref{basis_un_lambda}) of the space $\ell^2(\Sb)$ corresponds to the orthonormal basis $\big\{ ({\mathsf{M}_{\al,\xi}(\la)})^{-\frac12}\un\la \big\}_{\la\in\Sb}$ of $\ell^2(\Sb,\mathsf{M}_{\al,\xi})$. To any operator $A$ in $\ell^2(\Sb,\mathsf{M}_{\al,\xi})$ corresponds the operator $I_{\al,\xi} A I_{\al,\xi}^{-1}$ acting in $\ell^2(\Sb)$.

Now we can formulate and prove the main statement of this section:
\begin{prop}\label{prop:expectation_formula}
  Let $A$ be a bounded operator in $\ell^2(\Sb,\mathsf{M}_{\al,\xi})$. Then
  \begin{equation}\label{expectation_formula_prop_formula}
    (A\mathbf{1},\mathbf{1})_{\mathsf{M}_{\al,\xi}}=
    \big(\widetilde {\mathsf{G}}_\xi^{-1} (I_{\al,\xi} AI_{\al,\xi} ^{-1})\widetilde {\mathsf{G}}_\xi 
    \un\varnothing,\un\varnothing\big).
  \end{equation}
  Here $\mathbf{1}\in\ell^2(\Sb,\mathsf{M}_{\al,\xi})$ is the constant identity function. On the left the inner product is in $\ell^2(\Sb,\mathsf{M}_{\al,\xi})$, while on the right it is taken in $\ell^2(\Sb)$.
\end{prop}
\begin{proof}
  Let us first show that 
  \begin{equation}\label{wgxi_varnothing_identity}
    \widetilde {\mathsf{G}}_\xi\un\varnothing=\sum\nolimits_{\la\in\Sb}
    \left( \mathsf{M}_{\al,\xi}(\la) \right)^{\frac12}\un\la.
  \end{equation}
  In the matrix group $SL(2,\C)$ we have
  \begin{equation*}
    G_\xi=\exp\left( \sqrt\xi U \right)
    \exp\left( \frac12\log(1-\xi)H \right)
    \exp\left( -\sqrt\xi D \right).
  \end{equation*}
  The vector $\un\varnothing\in\ell_{\mathrm{fin}}^2(\Sb)$ is analytic for the action of $\mathfrak{su}(1,1)$ (Proposition \ref{prop:integrability_ell2_fin}), so on this vector the representation of $SU(1,1)^\sim$ can be extended to a representation of the local complexification of the group $SU(1,1)^\sim$ (see, e.g., the beginning of \S7 in \cite{nelson1959analytic}). This means that for small enough $\xi$ (when $\widetilde {{G}}_\xi$ is close to the unity of the group $SU(1,1)^\sim$) we have
  \begin{equation*}
    \widetilde {\mathsf{G}}_\xi\un\varnothing=\exp\left( \sqrt\xi \Uf \right)
    \exp\left( \frac12\log(1-\xi)\Hf \right)
    \exp\left( -\sqrt\xi \Df \right)\un\varnothing.
  \end{equation*}
  The operator $e^{-\sqrt\xi \Df}$ preserves $\un\varnothing$, and thus 
  \begin{equation*}
    \widetilde {\mathsf{G}}_\xi\un\varnothing=
    (1-\xi)^{\al/4}\sum\nolimits_{\la\in\Sb}
    \tfrac{\xi^{|\la|}}{|\la|!}\gdim\la\cdot\left(
    \prod\nolimits_{\square\in\la}q_\al(\la)\right)\un\la=
    \sum\nolimits_{\la\in\Sb}\left( \mathsf{M}_{\al,\xi}(\la) \right)^{\frac12}\un\la.
  \end{equation*}
  We have established (\ref{wgxi_varnothing_identity}) for small $\xi$. The left-hand side of (\ref{wgxi_varnothing_identity}) is analytic in $\xi$\footnote{Throughout the paper, when speaking about analytic functions in $\xi$, we assume that $\xi$ lies in the unit open disc $\left\{ z\in\C\colon |z|< 1 \right\}$.} because $\un\varnothing$ is an analytic vector for the operator $\widetilde {\mathsf{G}}_\xi$ by Proposition \ref{prop:integrability_ell2_fin}. The right-hand side of (\ref{wgxi_varnothing_identity}) is also analytic in $\xi$ by definition of $\mathsf{M}_{\al,\xi}$, see \S\ref{subsection:Point_processes}. Thus, (\ref{wgxi_varnothing_identity}) holds for all $\xi\in(0,1)$.

  It follows that $I_{\al,\xi}^{-1}\widetilde {\mathsf{G}}_\xi\un\varnothing= \mathbf{1}\in\ell^2(\Sb,\mathsf{M}_{\al,\xi})$, see (\ref{istry}). Therefore,
  \begin{equation*}
    (\widetilde {\mathsf{G}}_\xi^{-1}I_{\al,\xi} 
    AI_{\al,\xi}^{-1}\widetilde {\mathsf{G}}_\xi\un\varnothing,\un\varnothing)=
    (\widetilde {\mathsf{G}}_\xi^{-1}I_{\al,\xi} (A\mathbf1),\un\varnothing)
    =
    (I_{\al,\xi}\left( A\mathbf 1 \right),\widetilde {\mathsf{G}}_\xi\un\varnothing),
  \end{equation*}
  because the operator $\widetilde {\mathsf{G}}_\xi$ is unitary and has real matrix elements. We have
  \begin{equation*}
    \begin{array}{l}\displaystyle
      (I_{\al,\xi}\left( A\mathbf 1 \right),
      \widetilde {\mathsf{G}}_\xi\un\varnothing)=
      \bigg(
      I_{\al,\xi}\left( A\mathbf 1 \right),
      \sum_{\la\in\Sb}
      \left( \mathsf{M}_{\al,\xi}(\la) \right)^{\frac12}\un\la
      \bigg)
      =
      \sum_{\la\in\Sb}
      \left(I_{\al,\xi}\left( A\mathbf 1 \right),\un\la  \right)
      \left( \mathsf{M}_{\al,\xi}(\la) \right)^{\frac12}
      \\\displaystyle\qquad
      =
      \sum\nolimits_{\la\in\Sb}
      \left(I_{\al,\xi}\left( A\mathbf 1 \right),
      I_{\al,\xi}(\un\la)  \right)=
      \left(A\mathbf1,
      \sum\nolimits_{\la\in\Sb}
      \un\la\right)_{\mathsf{M}_{\al,\xi}}=
      (A\mathbf1,\mathbf1)_{\mathsf{M}_{\al,\xi}}.
    \end{array}
  \end{equation*}
  This concludes the proof.
\end{proof}
\begin{rmk}\label{rmk:expectation_formula_simple}
  The left-hand side of (\ref{expectation_formula_prop_formula}) can be regarded as an expectation with respect to the measure $\mathsf{M}_{\al,\xi}$ of the function $(A\mathbf{1})(\cdot)$ on $\Sb$. In the special case when the operator $A$ is diagonal, say, $A=A_f$ is the multiplication by a (bounded) function $f(\cdot)$ on $\Sb$, (\ref{expectation_formula_prop_formula}) is rewritten as the following formula for an expectation:
  \begin{equation}\label{expectation_formula_simple}
    \mathbb{E}_{\al,\xi}f:=
    \sum\nolimits_{\la\in\Sb}f(\la)\mathsf{M}_{\al,\xi}(\la)=
    \big(\widetilde {\mathsf{G}}_\xi^{-1} A_f\widetilde {\mathsf{G}}_\xi 
    \un\varnothing,\un\varnothing\big).
  \end{equation}
  This case is used in the computation of the static correlation functions, and for the dynamical correlation functions we need to use the more general statement of Proposition \ref{prop:expectation_formula}.
\end{rmk}



\section{Fermionic Fock space} 
\label{sec:fermionic_fock_space}

In this section we realize the Hilbert space $\ell^2(\Sb)$ as   a fermionic Fock space over $\ell^2(\N)$, and also define a representation of a Clifford algebra in this Fock space. This Clifford algebra is an infinite-dimensional analogue of a Clifford algebra over an odd-dimensional space (similar Clifford algebras and their Fock representations were considered in, e.g., \cite{Date1982transformation}, \cite{Matsumoto2005}, \cite{Vuletic2007shifted}). Note that in the case of the $z$-measures \cite{Okounkov2001a} one should work with an analogue of a Clifford algebra over an even-dimensional space. This difference, in particular, leads to the fact that in our case \textit{a priori} the use of this algebra provides us only with a Pfaffian formula for the correlation functions of the point processes $\mathsf{M}_{\al,\xi}$ (the static case). The proof that $\mathsf{M}_{\al,\xi}$ is actually a determinantal process requires additional considerations (see \S\ref{subsection:mixing_and_point_configurations} and Theorem \ref{thm:knuxi}) which in fact do not work in the case of dynamical correlation functions.

\subsection{Wick's theorem}\label{subsection:Clifford}

We begin with the definition of a certain Clifford algebra over the Hilbert space $V:=\ell^2(\Z)$. Denote the standard orthonormal basis of the space $V$ by $\left\{ v_x \right\}_{x\in\Z}$. Define a symmetric bilinear form $\left\langle{\cdot,\cdot} \right\rangle$ on $V$ by
\begin{equation*}
  \left\langle{v_x,v_y}\right\rangle:=\left\{
  \begin{array}{ll}
    1,&\qquad\mbox{if $x=-y\ne0$};\\
    2,&\qquad\mbox{if $x=y=0$};\\
    0,&\qquad\mbox{otherwise}.
  \end{array}
  \right.
\end{equation*}

Let $V^+$ and $V^-$ be the spans of $\left\{ v_x \right\}_{x\in\Z_{>0}}$ and $\left\{ v_x \right\}_{x\in\Z_{<0}}$, respectively, and let $V^0$ denote the space $\C v_0$. Note that the spaces $V^+$ and $V^-$ are maximal isotropic subspaces for the form $\left\langle{\cdot,\cdot}\right\rangle$, and 
\begin{equation*}
  V=V^-\oplus V^0\oplus V^+.
\end{equation*}

By $\Cl(V)$ denote the {Clifford algebra} over the quadratic space $(V,\left\langle{\cdot,\cdot}\right\rangle)$, that is, $\Cl(V)$ is the quotient of the tensor algebra $\bigoplus_{n=0}^{\infty}V^{\otimes n}$ of the space $V$ by the two-sided ideal generated by the elements
\begin{equation*}
  \left\{ v\otimes v'+v'\otimes v-\left\langle{v,v'}\right\rangle
  \colon v,v'\in V \right\}.
\end{equation*}
The tensor product of $v$ and $v'$ in $\Cl(V)$ is denoted simply by $vv'$. Thus,
\begin{equation}\label{Clifford_tensor}
  vv'+v'v=\left\langle{v,v'}\right\rangle
  \qquad\mbox{for all $v,v'\in V$}.
\end{equation}

Now let us prove a version of Wick's theorem that allows to write certain functionals on $\Cl(V)$ as Pfaffians. (In \S\ref{subsection:creation_annihilation} below we define a functional on $\Cl(V)$ called the vacuum average to which this version of Wick's theorem is applicable.)

\begin{thm}\label{thm:Wick}
  Let ${\mathbf{F}}$ be a linear functional on $\Cl(V)$ such that ${\mathbf{F}}(1)=1$ and for any $p,q,r\in\Z_{\ge0}$, $f_1^+,\dots,f_p^+\in V^+$, and $f_1^-,\dots,f_q^-\in V^-$, we have
  \begin{equation}\label{clf=0}
	  {\mathbf{F}}(f_{1}^+\dots f_p^+ v_0^r f_{1}^-\dots f_q^{-})=0
  \end{equation}
  if at least one of the numbers $p,q$ is nonzero.
  
  Then for any $n\ge1$ and any $2n$ elements $f_1,\dots,f_{2n}\in V$ we have
  \begin{equation*}
    {\mathbf{F}}(f_1\dots f_{2n})=
    \Pf({\mathbf{F}}\llbracket f_1,\dots,f_{2n}\rrbracket),
  \end{equation*}
  where ${\mathbf{F}}\llbracket f_1,\dots,f_{2n}\rrbracket$ is the skew-symmetric $2n\times 2n$ matrix in which the $kj$-th entry above the main diagonal is ${\mathbf{F}}(f_kf_j)$, $1\le k<j\le 2n$.
\end{thm}

\begin{proof}
  {\bf{}Step 1.\/}
  Consider decompositions 
  \begin{equation*}
    f_j=f_j^-+f_j^0+f_j^+,\qquad j=1,\dots,2n,
  \end{equation*}
  where $f_j^\pm\in V^\pm$ and $f^0_j\in V^0=\C v_0$. Thus,
  \begin{equation*}
    {\mathbf{F}}(f_1\dots f_{2n})=\sum\nolimits_{s_1,\dots,s_{2n}}
    {\mathbf{F}}(f_1^{s_1}\dots f_{2n}^{s_{2n}}), 
  \end{equation*}
  where each $s_j$ is a sign, $s_j\in \left\{ -,0,+ \right\}$, and the sum is taken over all $3^{2n}$ possible sequences of signs.

  {\bf{}Step 2.\/}
  Fix any particular sequence of signs $(s_1,\dots,s_{2n})$. Consider first the case when all of the $s_j$'s are nonzero. We aim to prove that
  \begin{equation}\label{wick_proof1}
    {\mathbf{F}}(f_1^{s_1}\dots f_{2n}^{s_{2n}})=
    \Pf({\mathbf{F}}\llbracket
    f_1^{s_1},\dots, f_{2n}^{s_{2n}}
    \rrbracket),
  \end{equation}
  where ${\mathbf{F}}\llbracket f_1^{s_1},\dots, f_{2n}^{s_{2n}}\rrbracket$ is the $2n\times 2n$ skew-symmetric matrix in which the $kj$th entry above the main diagonal is ${\mathbf{F}}(f_{k}^{s_k}f_j^{s_j})$. 

  First, note that if in the sequence $(s_1,\dots,s_{2n})$ all the ``$+$'' signs are on the left and all the ``$-$'' signs are on the right,\footnote{Including the case when there are only ``$+$'' or only ``$-$'' signs.} then by (\ref{clf=0}) we get (\ref{wick_proof1}), because in the Pfaffian in the right-hand side of (\ref{wick_proof1}) each entry is zero. 

  Next, observe that (\ref{wick_proof1}) is equivalent to 
  \begin{equation}\label{wick_proof2}
    {\mathbf{F}}(f_1^{s_1}\dots f_{2n}^{s_{2n}})=
    \sum\nolimits_{k=1}^{2n-1}(-1)^{k+1}
    {\mathbf{F}}(f_1^{s_1}\dots \widehat{f_{k}^{s_{k}}}
    \dots f_{2n-1}^{s_{2n-1}})
    {\mathbf{F}}(f_{k}^{s_k}f_{2n}^{s_{2n}}),
  \end{equation}
  this is just the standard Pfaffian expansion (here $\widehat{f_{k}^{s_{k}}}$ means the absence of ${f_{k}^{s_{k}}}$). It can be readily verified that the right-hand side and the left-hand side of (\ref{wick_proof2}) vary in the same way under the interchange  $f_{r}^{s_r}\leftrightarrow f_{r+1}^{s_{r+1}}$ for any $r=1,\dots,2n-1$. This implies that (\ref{wick_proof2}) holds because one can always move the ``$+$'' signs to the left and the ``$-$'' signs to the right. This argument is similar to the proof of Lemma 2.3 in \cite{Vuletic2007shifted}.

  {\bf{}Step 3.\/} Now assume that among the sequence of signs $(s_1,\dots,s_{2n})$ there can be zeroes. It is not hard to see that both sides of (\ref{wick_proof1}) vanish unless the number of zeroes is even. Let the positions of zeroes be $j_1<\dots<j_{2k}$. Thus, moving all $f_{j_1}^{0},\dots, f_{j_{2k}}^{0}$ to the left, we have
  \begin{align}\label{wick_proof3}
    {\mathbf{F}}(f_1^{s_1} \dots f_{2n}^{s_{2n}})
    =
    (-1)^{\sum\limits_{m=1}^{2k}(j_m-m)}
    {\mathbf{F}}(f_{j_1}^0\dots f_{j_{2k}}^{0})
    {\mathbf{F}}(f_{1}^{s_1}\dots\widehat{f_{j_1}^{0}}\dots
    \widehat{f_{j_{2k}}^{0}}\dots f_{2n}^{s_{2n}}).
  \end{align}
  By (\ref{wick_proof1}), the factor ${\mathbf{F}}(f_{1}^{s_1}\dots\widehat{f_{j_1}^{0}}\dots \widehat{f_{j_{2k}}^{0}}\dots f_{2n}^{s_{2n}})$ is written as the corresponding Pfaffian of order $(2n-2k)$. Assume that $f_{j_m}^0=c_mv_0$ (where $m=1,\dots,2k$), then
  \begin{equation*}
    {\mathbf{F}}(f_{j_1}^0\dots f_{j_{2k}}^{0})=c_1\dots c_{2k}.
  \end{equation*}
  Since for any $f\in V^+\oplus V^-$ we have (using (\ref{clf=0})) ${\mathbf{F}}(v_0f)={\mathbf{F}}(fv_0)=0$, the right-hand side of (\ref{wick_proof3}) can be interpreted as the Pfaffian of the block $2n\times 2n$ matrix with blocks formed by rows and columns with numbers $j_1,\dots,j_{2k}$ and $\left\{ 1,\dots,2n \right\}\setminus\left\{ j_1,\dots,j_{2k} \right\}$, respectively. This skew-symmetric $2n\times 2n$ matrix is exactly ${\mathbf{F}}\llbracket f_1^{s_1},\dots, f_{2n}^{s_{2n}}\rrbracket$ for our sequence $(s_1,\dots,s_{2n})$.

  This implies that (\ref{wick_proof1}) holds for any choice of signs $(s_1,\dots,s_{2n})$, $s_j\in\left\{ -,0,+ \right\}$.

  {\bf{}Step 4.\/}
  Let us now deduce the claim of the theorem from (\ref{wick_proof1}). We must prove that
  \begin{equation*}
    \sum\nolimits_{s_1,\dots,s_{2n}}
    \Pf({\mathbf{F}}\llbracket f_1^{s_1},\dots, 
    f_{2n}^{s_{2n}}\rrbracket)=
    \Pf({\mathbf{F}}\llbracket f_1,\dots, f_{2n}\rrbracket).
  \end{equation*}
  This is done by induction on $n$. The base is $n=1$:
  \begin{equation*}
    {\mathbf{F}}(f_1^-f_2^+)+{\mathbf{F}}(f_1^0f_2^0)=
    {\mathbf{F}}(f_1f_2)
  \end{equation*}
  (all other combinations of signs in the left-hand side give zero contribution). The induction step is readily verified using the Pfaffian expansion (\ref{wick_proof2}). This concludes the proof of the theorem.
\end{proof}

\subsection{Fermionic Fock space}\label{subsection:Fock_space}

Consider the space $\ell^2(\N)$ with the standard orthonormal basis $\left\{ \varepsilon_k \right\}_{k\in\N}$. The exterior algebra $\wedge\ell^2(\N)$ is the vector space with the basis
\begin{equation} \label{fock_basis}
  \left\{ \mathsf{vac} \right\}
  \cup\{\varepsilon_{i_1}\wedge\dots\wedge \varepsilon_{i_\ell}\colon 
  \infty> i_1>\ldots>i_\ell\ge1,\ \ell=1,2,\dots\},
\end{equation}
where $\mathsf{vac}\equiv1$ is called the \textit{vacuum vector}. Define an inner product $(\cdot,\cdot)$ in the exterior algebra $\wedge\ell^2(\N)$ with respect to which the basis (\ref{fock_basis}) is orthonormal. This inner product turns $\wedge\ell^2(\N)$ into a pre-Hilbert space. Its Hilbert completion is called the (\textit{fermionic}) \textit{Fock space} and is denoted by $\mathop{\mathsf{Fock}}(\N)$. The space $(\wedge\ell^2(\N),(\cdot,\cdot))$ consisting of finite linear combinations of the basis vectors (\ref{fock_basis}) is denoted by $\mathop{\mathsf{Fock}_{\mathrm{fin}}}(\N)$.

Clearly, the map
\begin{equation*}
  \un\la\mapsto\varepsilon_{\la_1}
  \wedge\dots\wedge\varepsilon_{\la_{\ell(\la)}},
  \qquad\la\in\Sb
\end{equation*}
(in particular, $\un\varnothing$ maps to $\mathsf{vac}$) defines an isometry between the pre-Hilbert spaces $\ell_{\mathrm{fin}}^2(\Sb)$ and $\mathop{\mathsf{Fock}_{\mathrm{fin}}}(\N)$, and also between their Hilbert completions $\ell^2(\Sb)$ and $\mathop{\mathsf{Fock}}(\N)$. Below we identify $\ell^2(\Sb)$ and $\mathop{\mathsf{Fock}}(\N)$, and by $\un\la$ we mean the vector $\varepsilon_{\la_1}\wedge\dots\wedge\varepsilon_{\la_{\ell(\la)}}$.

In the next subsection we describe the structure of $\mathop{\mathsf{Fock}}(\N)$ in more detail.

\subsection{Creation and annihilation operators. Vacuum average}\label{subsection:creation_annihilation}

Let $\phi_k$, $k=1,2,\dots$, be the \textit{creation operators} in $\mathop{\mathsf{Fock}}(\N)$, that is,
\begin{equation*}
  \phi_k\un\la:=\varepsilon_k\wedge\un\la,\qquad \la\in\Sb.
\end{equation*}
Let $\phi^*_k$, $k=1,2,\dots$, be the operators that are adjoint to $\phi_k$ with respect to the inner product in $\mathop{\mathsf{Fock}}(\N)$. They are called the \textit{annihilation operators} and act as follows:
\begin{equation*}
  \phi^*_k\un\la=\sum\nolimits_{j=1}^{\ell(\la)}
  (-1)^{j+1}\delta_{k,\la_j}\cdot
  \varepsilon_{\la_1}\wedge\dots\wedge
  \widehat{\varepsilon_{\la_j}}\wedge\dots\wedge 
  \varepsilon_{\la_{\ell(\la)}}.
\end{equation*}
We also need the operator $\phi_0=\phi_0^*$ acting as
\begin{equation*}
  \phi_0\un\la:=(-1)^{\ell(\la)}\un\la.
\end{equation*}

To simplify certain formulas below, we organize the operators $\phi_k$, $\phi_0$ and $\phi_k^*$ into a single family:
\begin{equation*}
  {\boldsymbol\phi}_m:=\left\{
  \begin{array}{ll}
    \phi_m,&\qquad\mbox{if $m\ge0$};\\
    (-1)^m\phi_{-m}^*,&\qquad\mbox{otherwise},
  \end{array}
  \right.\qquad \text{where $m\in\Z$}.
\end{equation*}
It can be readily checked that the operators ${\boldsymbol\phi}_m$ satisfy the following anti-com\-mu\-ta\-tion relations:
\begin{equation}\label{anti_commutation_pw}
  {\boldsymbol\phi}_k{\boldsymbol\phi}_l+
  {\boldsymbol\phi}_l{\boldsymbol\phi}_k=\left\{
  \begin{array}{ll}
    2,&\qquad\mbox{if $k=l=0$};\\
    (-1)^l\delta_{k,-l},&\qquad\mbox{otherwise}.
  \end{array}
  \right.
\end{equation}

In agreement to these definitions, let $\left\{ \boldsymbol v_x \right\}_{x\in\Z}$ be another orthonormal basis in the space $V=\ell^2(\Z)$ defined as
\begin{equation}\label{basis_vw}
  \boldsymbol v_x:=\left\{
  \begin{array}{ll}
    v_x,&\qquad\mbox{if $x\ge0$};\\
    (-1)^{x}v_{x},&\qquad\mbox{if $x<0$},
  \end{array}
  \right.\qquad \text{where $x\in\Z$}.
\end{equation}
In other words, $\boldsymbol v_x=(-1)^{x\wedge 0}v_x$. In the Clifford algebra $\Cl(V)$ we have
\begin{equation}\label{Clifford_algebra_with_basis_wv}
  \boldsymbol v_x\boldsymbol v_y+\boldsymbol v_y\boldsymbol v_x=
  \left\langle{\boldsymbol v_x,\boldsymbol v_y}\right\rangle=\left\{
  \begin{array}{ll}
    2,&\qquad x=y=0;\\
    (-1)^{x}\delta_{x,-y},&\qquad\mbox{otherwise}.
  \end{array}
  \right.
\end{equation}
\begin{df}\label{df:T}
  Let $\mathcal{T}$ be a representation of 
  the Clifford algebra $\Cl(V)$
  in $\mathop{\mathsf{Fock}}(\N)$ 
  defined on $V$ by
  \begin{equation*}
    \mathcal{T}(\boldsymbol v_x):={\boldsymbol\phi}_x,\qquad x\in\Z,
  \end{equation*}
  and extended to the whole $\Cl(V)$ by (\ref{Clifford_tensor}) and by linearity. The fact that $\mathcal{T}$ is indeed a representation follows from (\ref{anti_commutation_pw}) and (\ref{Clifford_algebra_with_basis_wv}).
\end{df}
\begin{df}\label{df:vacuum_average}
  The representation $\mathcal{T}$ allows to consider the following functional on the Clifford algebra $\Cl(V)$:
  \begin{equation*}
    {\mathbf{F}}_\mathsf{vac}(w):=
    \left( \mathcal{T}(w)\mathsf{vac},\mathsf{vac} \right),\qquad w\in\Cl(V)
  \end{equation*}
  called the \textit{vacuum average}. Here the inner product on the right is taken in $\mathop{\mathsf{Fock}}(\N)$.
\end{df} 
It can be readily verified that the functional ${\mathbf{F}}_\mathsf{vac}$ on $\Cl(V)$ satisfies the hypotheses of Wick's Theorem \ref{thm:Wick}.

\subsection{The representation $R$}\label{subsection:representation_R}

The space $\ell_{\mathrm{fin}}^2(\Sb)$ is isometric to $\mathop{\mathsf{Fock}_{\mathrm{fin}}}(\N)$, and thus the Kerov's operators $\Uf$, $\Df$, and $\Hf$ (\ref{KO}) in $\ell_{\mathrm{fin}}^2(\Sb)$ give rise to certain operators in $\mathop{\mathsf{Fock}_{\mathrm{fin}}}(\N)$. We obtain a representation of the Lie algebra $\mathfrak{sl}(2,\C)$ in $\mathop{\mathsf{Fock}}(\N)$, denote this representation by $R$.

It can be readily verified that the action of the operators $R(U)$, $R(D)$, and $R(H)$ in $\mathop{\mathsf{Fock}_{\mathrm{fin}}}(\N)$ (this subspace of $\mathop{\mathsf{Fock}}(\N)$ is invariant for the representation $R$ of $\mathfrak{sl}(2,\C)$) can be expressed in terms of the creation and annihilation operators as follows:
\begin{align}\nonumber
    R(U)&=\displaystyle
    \sum\nolimits_{k=0}^{\infty}
    2^{-\delta(k)/2}(-1)^{k}\sqrt{k(k+1)+\al}
    \cdot{\boldsymbol\phi}_{k+1}
    {\boldsymbol\phi}_{-k},\\\label{repres_R}
    R(D)&=\displaystyle
    \sum\nolimits_{k=0}^{\infty}
    2^{-\delta(k)/2}(-1)^{k+1}\sqrt{k(k+1)+\al}
    \cdot{\boldsymbol\phi}_{k}
    {\boldsymbol\phi}_{-k-1},\\\nonumber
    R(H)&=\displaystyle 
    \tfrac\al2+2\sum\nolimits_{k=1}^{\infty}(-1)^kk
    {\boldsymbol\phi}_k{\boldsymbol\phi}_{-k}.
\end{align}

Proposition \ref{prop:integrability_ell2_fin} can be reformulated for the representation $R$. Namely, the representation $R$ of $\mathfrak{sl}(2,\C)$ restricted to the real form $\mathfrak{su}(1,1)\subset \mathfrak{sl}(2,\C)$ gives rise to a unitary representation of the universal covering group $SU(1,1)^\sim$ in the Hilbert space $\mathop{\mathsf{Fock}}(\N)$. Denote this representation also by $R$.

Under the identification of $\ell^2(\Sb)$ with $\mathop{\mathsf{Fock}}(\N)$, we say that the map $I_{\al,\xi}$ (\ref{istry}) is an isometry between $\ell^2(\Sb,\mathsf{M}_{\al,\xi})$ and $\mathop{\mathsf{Fock}}(\N)$. By Proposition \ref{prop:expectation_formula}, for any bounded operator $A$ in $\ell^2(\Sb,\mathsf{M}_{\al,\xi})$ we have
\begin{equation}\label{expectation_formula_fock_hard}
  (A\mathbf{1},\mathbf{1})_{\mathsf{M}_{\al,\xi}}=
  \big(
  R(\widetilde {{G}}_\xi)^{-1}(I_{\al,\xi} AI_{\al,\xi}^{-1})R(\widetilde {{G}}_\xi)\mathsf{vac},\mathsf{vac}
  \big).
\end{equation}
Here $\widetilde {{G}}_\xi\in SU(1,1)^\sim$, $0\le \xi<1$ is defined in \S\ref{subsection:expectation_formula}, and $\mathbf{1}\in\ell^2(\Sb,\mathsf{M}_{\al,\xi})$ is the constant identity function. The inner products on the left and on the right are taken in the spaces $\ell^2(\Sb,\mathsf{M}_{\al,\xi})$ and $\mathop{\mathsf{Fock}}(\N)$, respectively.

Formula (\ref{expectation_formula_simple}) for the expectation of a bounded function $f(\cdot)$ on $\Sb$ with respect to the measure $\mathsf{M}_{\al,\xi}$ is rewritten as
\begin{equation}\label{expectation_formula_fock}
  \mathbb{E}_{\al,\xi}f=
  \big(
  R(\widetilde {{G}}_\xi)^{-1} A_f R(\widetilde {{G}}_\xi)\mathsf{vac},\mathsf{vac}
  \big),
\end{equation}
where $A_f$ is the operator of multiplication by $f$.

As we will see below, for averages expressing the correlation functions, the right-hand side of (\ref{expectation_formula_fock_hard}) (and (\ref{expectation_formula_fock})) can be written as a vacuum average. That is, the operator $R(\widetilde {{G}}_\xi)^{-1}(I_{\al,\xi} AI_{\al,\xi}^{-1})R(\widetilde {{G}}_\xi)$ (respectively, $R(\widetilde {{G}}_\xi)^{-1} A_f R(\widetilde {{G}}_\xi)$) has the form $\mathcal{T}(w)$ for a certain $w\in\Cl(V)$.



\section{Z-measures and an orthonormal basis in $\ell^2(\Z)$} 
\label{sec:z_measures_on_ordinary_partitions_and_an_orthonormal_basis_in_ell_2_z_}

In this section we examine functions on the lattice which are used in our expressions for correlation kernels (both static and dynamical). They form an orthonormal basis in the Hilbert space $\ell^2(\Z)$ and are eigenfunctions of a certain second order difference operator on the lattice. These functions arise as a particular case of the functions used to describe correlation kernels in the model of the $z$-measures on ordinary partitions, and we begin this section by recalling some of the results of the papers \cite{borodin2006meixner}, \cite{Borodin2006} which we will use below.

\subsection{Results about the z-measures on ordinary partitions} \label{subsection:z-measures}

For an ordinary (i.e., not necessary strict) partition $\si=(\si_1,\dots,\si_{\ell(\si)})$, let $\dim \si$ denote the number of standard Young tableaux of shape $\si$ (we identify partitions with ordinary Young diagrams as usual, e.g., see \cite[Ch. I, \S1]{Macdonald1995}), and $|\si|$ be the number of boxes in the Young diagram $\si$. 

Consider the following 3-parameter family of measures on the set of all ordinary partitions:
\begin{equation}\label{M_zz'xi}
  M_{z,z',\xi}(\si)=(1-\xi)^{zz'}\xi^{|\si|}
  (z)_{\si}(z')_{\si}
  \left(\frac{\dim\si}
  {|\si|!}\right)^2,
\end{equation}
where $(a)_\si:=\prod_{i=1}^{\ell(\si)}(a)_{\si_i}$ is a generalization of the Pochhammer symbol. Here the parameter $\xi\in(0,1)$ is the same as our parameter $\xi$ (e.g., in \S\ref{subsection:Point_processes}), and the parameters $z,z'$ are in one of the following two families (we call such parameters \emph{admissible}):
\begin{enumerate}[$\bullet$]
  \label{param_zz'_assumptions}
  \item (\emph{principal series}) The numbers $z,z'$ are not real and are conjugate to each other.
  \item (\emph{complementary series}) Both $z,z'$ are real and are contained in the same open interval of the form $(m,m+1)$, where $m\in\Z$.
\end{enumerate} 

To any ordinary partition $\si=(\si_1,\dots,\si_{\ell(\si)},0,0,\dots)$ is associated an infinite point configuration (sometimes called the \emph{Maya diagram}) on the lattice $\Z'=\Z+\frac12$:
\begin{equation}\label{und_X}
  \si\mapsto 
  \underline X(\si):=
  \{\si_i-i+\tfrac12\}_{i=1}^\infty
  \subset\Z'.
\end{equation}
One can see that the correspondence $\si\mapsto\underline X(\si)$ is a bijection between ordinary partitions and those (infinite) configurations $\underline X\subset \Z'$ for which the symmetric difference $\underline X\mathop{\triangle} \Z'_-$ is a finite subset containing equally many points in $\Z'_+$ and $\Z'_-$ (Here $\Z'_+$ and $\Z'_-$ denote the sets of all positive resp. negative half-integers.)

Using the above identification of ordinary partitions with point configurations on the lattice $\Z'$, it is possible to speak about the correlation functions of the measures $M_{z,z',\xi}$ (\ref{M_zz'xi}) in the same way as in (\ref{df:static_correlation_functions}). The resulting random point processes are determinantal with a correlation kernel $\underline K_{z,z',\xi}(\sh x,\sh y)$ (where $\sh x,\sh y\in\Z'$)
which is called the \emph{discrete hypergeometric kernel} \cite{Borodin2000a}, \cite{borodin2006meixner}.

\begin{rmk}
  Whenever speaking about points in the shifted lattice $\Z'=\Z+\frac12$, we denote them by $\sh x,\sh y,\dots$, because we want to reserve the letters $x,y,\dots$ for the non-shifted integers: $x,y,\ldots\in\Z$.
\end{rmk}

There are explicit formulas for the discrete hypergeometric kernel $\underline K_{z,z',\xi}$ which we will use. We proceed to describe them, but first we need to recall certain functions defined in \cite[(2.1)]{borodin2006meixner}:
\begin{align}\nonumber&
  \psi_{\sh a}(\sh x;z,z',\xi):=
  \left(
  \frac{\Gamma(\sh x+z+\frac12)
  \Gamma(\sh x+z'+\frac12)}
  {\Gamma(-\sh a+z+\frac12)
  \Gamma(-\sh a+z'+\frac12)}
  \right)^{\frac12}
  \xi^{\frac12(\sh x+\sh a)}
  (1-\xi)^{\frac12(z+z')-\sh a}\times
  \\
  \label{psi_a}
  &\qquad \times
  \frac{{}_2F_1(-z+\sh a+\frac12,
  -z'+\sh a+\frac12; \sh x+\sh a+1;
  \frac{\xi}{\xi-1})}
  {\Gamma(\sh x+\sh a+1)},\qquad \sh x\in\Z'.
\end{align}
Here $z,z',\xi$ are the parameters of the $z$-measures, and the index $\sh a$ of the functions runs over the lattice $\Z'$. As usual, ${}_2F_1$ is the Gauss hypergeometric function, ${}_2F_1(A,B;C;w):= \sum_{n=0}^\infty\frac{(A)_n(B)_n}{(C)_nn!}w^n$. As is explained in \cite[\S2]{borodin2006meixner}, the expression (\ref{psi_a}) makes sense for all $\sh a,\sh x\in\Z'$ due to the assumptions on the parameters $z,z'$ (see p. \pageref{param_zz'_assumptions}) and the fact that $\xi\in(0,1)$. Moreover, the functions $\psi_{\sh a}(\sh x;z,z',\xi)$ are real-valued. Let us summarize their properties for future use:
\begin{prop}
  [{\cite[\S2]{borodin2006meixner}}]
  \label{prop:psi_zz'xi_properties}
  \begin{enumerate}[\bf{}1)]
    \item The functions $\psi_{\sh a}(\sh x;z,z',\xi)$, as the index $\sh a$ ranges over $\Z'$, form an orthonormal basis in the Hilbert space $\ell^2(\Z')$.
    \item Consider the following second order difference operator $D(z,z',\xi)$ in $\ell^2(\Z')$ (acting on functions $f(\sh x)$, where $\sh x$ ranges over $\Z'$):
    \begin{align*}\nonumber
      D(z,z',\xi)f(\sh x)=
      \sqrt{\xi(z+\sh x+\tfrac12)
      (z'+\sh x+\tfrac12)}&f(\sh x+1)
      \\
      +
      \sqrt{\xi(z+\sh x-\tfrac12)
      (z'+\sh x-\tfrac12)}f(\sh x-1)&
      -\big(\sh x+\xi(z+z'+\sh x)\big)f(\sh x).
    \end{align*}
    The operator $D(z,z',\xi)$ is symmetric. The functions $\psi_{\sh a}$ are eigenfunctions of this operator:
    \begin{equation*}
      D(z,z',\xi)\psi_{\sh a}(\sh x;z,z',\xi)
      =
      \sh a(1-\xi)
      \psi_{\sh a}(\sh x;z,z',\xi),\qquad 
      \sh a,\sh x\in\Z'.
    \end{equation*}
    \item The functions $\psi_{\sh a}$ satisfy the following symmetry relations:
    \begin{align}
      \psi_{\sh a}(\sh x;z,z',\xi)&=
      \psi_{\sh x}(\sh a;-z,-z',\xi);
      \label{psi_zz'xi_symm1}
      \\
      \psi_{\sh a}(\sh x;z,z',\xi)&=
      (-1)^{\sh x+\sh a}
      \psi_{-\sh a}(-\sh x;-z,-z',\xi),
      \qquad \sh a,\sh x\in\Z'.
    \end{align}
    \item The functions $\psi_{\sh a}$
    satisfy the following three-term relation ($\sh a\in\Z'$):
    \begin{align}&
      (1-\xi)\sh x\psi_{\sh a}=
      \sqrt{\xi(z-\sh a+\tfrac12)
      (z'-\sh a+\tfrac12)}\psi_{\sh a-1}
      \nonumber
      \\&\qquad
      \label{psi_zz'xi_three-term}
      +
      \sqrt{\xi(z-\sh a-\tfrac12)
      (z'-\sh a-\tfrac12)}\psi_{\sh a+1}
      +
      (-\sh a+\xi(z+z'-\sh a))\psi_{\sh a}.
    \end{align}
  \end{enumerate}
\end{prop}

\begin{thm}[\cite{Borodin2000a}, \cite{borodin2006meixner}]\label{thm:K_zz'xi}
  Under the correspondence $\si\mapsto \underline X(\si)$ (\ref{und_X}), the $z$-measures become a determinantal point process on $\Z'$ with the correlation kernel given by
  \begin{equation}\label{K_zz'xi_formula}
    \underline K_{z,z',\xi}
    (\sh x, \sh y)=
    \sum\nolimits_{\sh a\in\Z'_+}
    \psi_{\sh a}(\sh x;z,z',\xi)
    \psi_{\sh a}(\sh y;z,z',\xi),\qquad
    \sh x,\sh y\in\Z'.
  \end{equation}
\end{thm}

  From Proposition \ref{prop:psi_zz'xi_properties}, one readily sees that the discrete hypergeometric kernel $\underline K_{z,z',\xi}$ (viewed as an operator in $\ell^2(\Z')$) is an \emph{orthogonal spectral projection} operator corresponding to the positive part of the spectrum of the difference operator $D(z,z',\xi)$. We will discuss the extended discrete hypergeometric kernel $\underline K_{z,z',\xi}(s,\sh x;t,\sh y)$ (which serves as a correlation kernel in a determinantal dynamical model associated to the $z$-measures) below in \S\ref{subsection:Pfk_dyn_Kzz} while describing how our dynamical Pfaffian kernel is related to it.

\subsection{An orthonormal basis $\{\boldsymbol\varphi_m\}$ in the Hilbert space $\ell^2(\Z)$} 
\label{subsection:an_orthonormal_basis_phw}

For the study of our model, we need the following family of functions:
\begin{align}
  \nonumber
  &
  \boldsymbol\varphi_m(x;\al,\xi):=
  \left(
  \frac{\Gamma(\frac12+\nu(\al)+x)
  \Gamma(\frac12-\nu(\al)+x)}
  {\Gamma(\frac12+\nu(\al)-m)
  \Gamma(\frac12-\nu(\al)-m)}
  \right)^{\frac12}
  \xi^{\tfrac12(x+m)}
  (1-\xi)^{-{m}}
  \times\\
  &\qquad\qquad\qquad\times
  \frac{
  {}_2F_1
  (\frac12+\nu(\al)+m,
  \frac12-\nu(\al)+m;
  x+m+1;\frac\xi{\xi-1})
  }{\Gamma(x+m+1)},
  \label{phw_al,xi}
\end{align}
where $\nu(\al)$ is given in Definition \ref{df:nu(al)}. Here the argument $x$ and the index $m$ range over the lattice $\Z$. Because $\al>0$, we have
$\Gamma(\tfrac12+\nu(\al)+k)\Gamma(\tfrac12-\nu(\al)+k)>0$ for any $k\in\Z$. Thus, the expression in (\ref{phw_al,xi}) which is taken to the power $\frac12$ is positive. Note also that while the hypergeometric function ${}_2F_1(A,B;C;w)$ is not defined if $C$ is a negative integer, the ratio $\frac{{}_2F_1(A,B;C;w)}{\Gamma(C)}$ (occurring in (\ref{phw_al,xi})) is well-defined for all $C\in\C$. Thus, we see that the functions $\boldsymbol\varphi_m(x;\al,\xi)$ are well-defined.

It can be readily verified that $\boldsymbol\varphi_m$'s arise as a particular case of the functions $\psi_{\sh a}$: described in \S\ref{subsection:z-measures} above:
\begin{equation}\label{phw_psi}
  \boldsymbol\varphi_m(x;\al,\xi)=
  \psi_{m+\frac12+d}
  (x-\tfrac12-d;
  \nu(\al)+\tfrac12+d,
  -\nu(\al)+\tfrac12+d;\xi)
\end{equation}
for \emph{any} $d\in\Z$. For $x,m,d\in\Z$, the numbers $m+\frac12+d$ and $x-\frac12-d$ belong to $\Z'$, as it should be.
Observe that the parameters
\begin{equation*}
  z=z(\al):=\nu(\al)+\tfrac12+d,\qquad
  z'=z'(\al):=-\nu(\al)+\tfrac12+d
\end{equation*}
for any $d\in\Z$ are admissible (i.e., of principal or complementary series, see p. \pageref{param_zz'_assumptions}). By Definition \ref{df:nu(al)}, for $0<\al\le\frac14$ these parameters belong to the complementary series, and for $\al>\frac14$ they are of principal series.
\begin{rmk}
  
  The fact that (\ref{phw_psi}) holds for any $d$ is a reflection of a certain translation invariance property of the $z$-measures, see \cite[\S10, 11]{Borodin1998}.
\end{rmk}

From Proposition \ref{prop:psi_zz'xi_properties} one can readily deduce the corresponding properties of our functions $\boldsymbol\varphi_m$:

\begin{prop}
  \label{prop:phw_properties}
  \begin{enumerate}[\bf{}1)]
    \item The functions $\boldsymbol\varphi_m(x;\al,\xi)$, as the index $m$ ranges over $\Z$, form an orthonormal basis in the Hilbert space $\ell^2(\Z)$:
    \begin{equation*}
      \sum\nolimits_{x\in\Z}
      \boldsymbol\varphi_m(x;\al,\xi)
      \boldsymbol\varphi_l(x;\al,\xi)=
      \delta_{ml},\qquad m,l\in\Z.
    \end{equation*}
    \item The functions $\boldsymbol\varphi_m$ are eigenfunctions of the following second order difference operator in $\ell^2(\Z)$ (acting on functions $f(x)$, where $x$ ranges over $\Z$):
    \begin{align*}
      \mathfrak{D}_{\al,\xi}f(x)=
      \sqrt{\xi(\al+x(x+1))}&f(x+1)
      \\\nonumber
      +
      \sqrt{\xi(\al+x(x-1))}f(x-1)&-
      x(1+\xi)f(x).
    \end{align*}
    This operator is symmetric in $\ell^2(\Z)$.
    We have
    \begin{equation*}
      \mathfrak{D}_{\al,\xi}
      \boldsymbol\varphi_m(x;\al,\xi)
      =
      m(1-\xi)
      \boldsymbol\varphi_m(x;\al,\xi),\qquad m,x\in\Z.
    \end{equation*}

    \item The functions $\boldsymbol\varphi_{m}$ satisfy the following symmetry relations:
      \begin{align}
        \boldsymbol\varphi_m(x;\al,\xi)&=
        \boldsymbol\varphi_x(m;\al,\xi);
        \label{phw_symm1}
        \\  
        \label{phw_symm2}      
        \boldsymbol\varphi_m(x;\al,\xi)&=
        (-1)^{x+m}
        \boldsymbol\varphi_{-m}(-x;\al,\xi),
        \qquad x,m\in\Z.
      \end{align}
      
      \item The functions $\boldsymbol\varphi_{m}$ satisfy the three-term relation:
      \begin{align}
        (1-\xi)x\boldsymbol\varphi_{m}=
        \sqrt{\xi(m(m+1)+\al)}
        &\boldsymbol\varphi_{m-1}
        \label{phw_three-term}
        \\
        \nonumber
        +
        \sqrt{\xi(m(m-1)+\al)}\boldsymbol\varphi_{m+1}
        &
        -m(1+\xi)\boldsymbol\varphi_m,
        \qquad m \in\Z.
      \end{align}
  \end{enumerate}
\end{prop}

Note that the property (\ref{phw_symm1}) here means that the functions $\boldsymbol\varphi_m(x; \al,\xi)$ are self-dual (in contrast to the more general functions $\psi_{\sh a}(\sh x;z,z',\xi)$, cf. (\ref{psi_zz'xi_symm1})).

\begin{proof}
  Every property is a straightforward consequence of (\ref{phw_psi}) and the corresponding claim of Proposition \ref{prop:psi_zz'xi_properties}.
\end{proof}

\subsection{``Twisting''} 
\label{subsection:_twisting_}

To simplify certain formulas in the paper (in particular, (\ref{Pfk_intro}) above), we will also need certain versions of our functions $\boldsymbol\varphi_m(x;\al,\xi)$ which differ from the original ones by multiplying by $(-1)^{x\wedge0}$:
\begin{equation}\label{wphw_al,xi}
  \widetilde{\boldsymbol\varphi}_m(x;\al,\xi):=
  (-1)^{x\wedge0}\boldsymbol\varphi_m(x;\al,\xi),\qquad
  x,m\in\Z.
\end{equation}
These functions also form an orthonormal basis in $\ell^2(\Z)$. They are eigenfunctions of a difference operator $\widetilde{\mathfrak{D}}_{\al,\xi}$ in $\ell^2(\Z)$ which is conjugate to $\mathfrak{D}_{\al,\xi}$:
\begin{align}\label{D_al,xi_twisted}
  (\widetilde{\mathfrak{D}}_{\al,\xi}f)(x):=
  (-1)^{\mathbbm{1}_{x<0}}
  \sqrt{\xi(\al+x(x+1))}&
  f(x+1)\\\nonumber
  +
  (-1)^{\mathbbm{1}_{x\le 0}}
  \sqrt{\xi(\al+x(x-1))}
  f(x-1)
  &
  -x(1+\xi)f(x),\qquad x\in\Z
\end{align}
(here $\mathbbm{1}$ means the indicator),
\begin{equation*}
  \widetilde{\mathfrak{D}}_{\al,\xi}
  \widetilde{\boldsymbol\varphi}_m(x;\al,\xi)=
  m(1-\xi)\widetilde{\boldsymbol\varphi}_m(x;\al,\xi),
  \qquad
  m\in\Z.
\end{equation*}
The functions $\{\widetilde{\boldsymbol\varphi}_m\}_{m\in\Z}$ also satisfy certain symmetry relations similar to (\ref{phw_symm1})--(\ref{phw_symm2}). Moreover, they clearly satisfy \textit{the same} three-term relations (\ref{phw_three-term}) as the non-twisted functions $\boldsymbol\varphi_m(x;\al,\xi)$.

\subsection{Matrix elements of $\mathfrak{sl}(2,\C)$-modules}
\label{subsection:matrix_elements_phw}

Here we interpret the functions $\{\boldsymbol\varphi_m\}$ (\ref{phw_al,xi}) introduced in this section through certain matrix elements of irreducible unitary representations of the Lie group $PSU(1,1)=SU(1,1)/\{\pm I\}$ ($I$ is the identity matrix) in the Hilbert space $\ell^2(\Z)$. 

\begin{rmk}
  The more general functions $\psi_{\sh a}$ (\ref{psi_a}) first appeared in the works of Vilenkin and Klimyk \cite{Vilenkin-Klimyk-DAN_UKR_1988}, \cite{Vilenkin-Klimyk-ITOGI1995-en} as matrix elements of unitary representations of the universal covering group $SU(1,1)^\sim$. In a context similar to ours they were obtained by Okounkov \cite{Okounkov2001a} in a computation of the discrete hypergeometric kernel $\underline K_{z,z',\xi}$ (\ref{K_zz'xi_formula}) using the fermionic Fock space.
\end{rmk}

Let $S$ be the representation of the Lie algebra $\mathfrak{sl}(2,\C)$ (spanned by the operators $U,D$, and $H$  (\ref{slf_representation})) in the Hilbert space $\ell^2(\Z)$ with the canonical orthonormal basis $\{\un k\}_{k\in\Z}$ (that is, $\un k(x)=\delta_{k,x}$) defined as follows:
\begin{align}\nonumber
  S(U)\un{k}&=\sqrt{k(k+1)+\al}\cdot\un{k+1};\\
  \label{Sch_representation}
  S(D)\un{k}&=\sqrt{k(k-1)+\al}\cdot\un{k-1};\\
  S(H)\un k &= 2k\cdot \un k.
  \nonumber
\end{align}
This representation depends on our parameter $\al>0$. For it one can prove an analogue
of Proposition \ref{prop:integrability_ell2_fin}:

\begin{prop}\label{prop:integrability_ellf_Z}
  All vectors of the space $\ell_{\mathrm{fin}}^2(\Z)$ (consisting of finite linear combinations of the basis vectors $\{\un k\}$) are analytic for the action $S$ of $\mathfrak{sl}(2,\C)$ (\ref{Sch_representation}). The representation $S$ of the Lie algebra $\mathfrak{su}(1,1)\subset \mathfrak{sl}(2,\C)$ in $\ell_{\mathrm{fin}}^2(\Z)$ lifts to a unitary representation of the Lie group $PSU(1,1)$ in the Hilbert space $\ell^2(\Z)$. 
\end{prop}

\begin{proof}
  The operators $S(U)-S(D)$, $i\big(S(U)+S(D)\big)$, and $iS(H)$ act skew-sym\-met\-ri\-cally in the (complex) pre-Hilbert space $\ell_{\mathrm{fin}}^2(\Z)$. It is known (e.g., see \cite{Pukanszky_SL2_1964} or \cite[Ch. VI, \S6]{Lang1985sl2}) that for any $\al>0$ the above representation $S$ of $\mathfrak{su}(1,1)$ in $\ell_{\mathrm{fin}}^2(\Z)$ is irreducible (this is an irreducible Harish--Chandra module) and lifts to a unitary representation of the Lie group $SU(1,1)$ in $\ell^2(\Z)$. Moreover, since $S(H)\un k=2k\cdot\un k$, this is in fact a representation of the group $PSU(1,1)$. The claim about analytic vectors follows from, e.g., \cite[Ch. X, \S3, Thm. 7]{Lang1985sl2}.
\end{proof}

If $0<\al\le\frac14$, the above irreducible representation of $PSU(1,1)$ in $\ell^2(\Z)$ belongs to complementary series, and for $\al>\frac14$ it is of principal series, e.g., see \cite{Pukanszky_SL2_1964} (cf. the series of the parameters $(z,z')$ in (\ref{phw_psi})). Denote this representation of $PSU(1,1)$ again by $S$. For notational reasons (e.g., see Proposition \ref{prop:phw_matr_el} below), also by $S$ let us denote the corresponding representations of $SU(1,1)$ and $SU(1,1)^\sim$ in $\ell^2(\Z)$ that are obtained from the representation of $PSU(1,1)$ by a trivial lifting procedure.

Now let us compute the matrix elements of the operator $S(G_\xi)^{-1}$ (where $G_\xi\in SU(1,1)$ is defined in (\ref{Gxi_df})) in the basis $\{\un k\}_{k\in\Z}$. These matrix elements will be used below in formulas for our correlation kernels.

\begin{prop}\label{prop:phw_matr_el}
  For all $x,k\in\Z$ we have
  \begin{equation}\label{matrix_el_phw}
    \left(S(G_\xi)^{-1}\un x,\un k\right)_{\ell^2(\Z)}=
    \boldsymbol\varphi_{-k}(x;\al,\xi).
  \end{equation}
\end{prop}

\begin{proof}
  Fix $x,k\in\Z$. By Proposition \ref{prop:integrability_ellf_Z}, the function $\xi\mapsto\left(S(G_\xi)^{-1}\un x,\un k\right)_{\ell^2(\Z)}$ is analytic. The right-hand side of (\ref{matrix_el_phw}) is also analytic in $\xi$, see (\ref{phw_al,xi}). Thus, it suffices to prove (\ref{matrix_el_phw}) for small $\xi$. Also by Proposition \ref{prop:integrability_ellf_Z}, on $\un x\in \ell_{\mathrm{fin}}^2(\Z)$ the representation $S$ can be extended to a representation of the local complexification of $PSU(1,1)$. This means that for small $\xi$ (when $G_\xi$ is close to the unity of the group) we can write:
  \begin{equation*}
    S(G_\xi)^{-1} \un x=
    \exp\left( -\sqrt\xi S(U) \right)
    \exp\left( \frac{\sqrt\xi}{1-\xi}S(D) \right)
    \exp\left( \frac12 \log(1-\xi) S(H) \right) \un x
  \end{equation*}
  (this follows from the corresponding identity for matrices in $SL(2,\C)$, see also the proof of Proposition \ref{prop:expectation_formula}).

  Denote $c_y:=\sqrt{y(y+1)+\al}$, so that $S(U)\un y=c_y\cdot \un {y+1}$ and $S(D)\un y=c_{y-1}\cdot\un{y-1}$. Note that $c_y^2=y(y+1)+\al=(y+\nu(\al)+\tfrac12)(y-\nu(\al)+\tfrac12)$.  
Also set $a:=-\sqrt\xi$ and $b:=\sqrt\xi/(1-\xi)$. We have
  \begin{align*}&
      \left(S(G_\xi)^{-1}\un x,\un k\right)_{\ell^2(\Z)}=
      (1-\xi)^{x}
      \left( e^{a S(U)}e^{b S(D)}\un x,
      \un k \right)_{\ell^2(\Z)}
      \\&
      \qquad\qquad
      =
      (1-\xi)^{x}\sum_{r=0}^{\infty}\sum_{l=0}^{\infty}
      \frac{a^rb^l}{r!l!}
      c_{x-l+r-1}\dots c_{x-l}c_{x-l}\dots c_{x-1}
      (\un{x-l+r},\un k)_{\ell^2(\Z)}.
  \end{align*}
  Clearly, $(\un{x-l+r},\un{k})_{\ell^2(\Z)}=\delta_{x-l+r,k}$. There are two cases: $x\ge k$, and $x\le k$. For $x\ge k$ we perform the above summation over $r\ge0$ and set $l=r+x-k$. For $x\le k$ we sum over $l$ and set $r=l+k-x$. After direct calculations we obtain (we omit the argument in $\nu(\al)$):
  \begin{align*}
    \left(S(G_\xi)^{-1}\un x,\un k\right)_{\ell^2(\Z)}&=
    (1-\xi)^x {b^{x-k}}
    \left(\frac{\Gamma(x+\nu+\frac12)\Gamma(x-\nu+\frac12)}
    {\Gamma(k+\nu+\frac12)\Gamma(k-\nu+\frac12)}\right)^{\frac12}
    \times\\&\qquad\times 
    \frac{{}_2F_1
    (\tfrac12-k-\nu,\tfrac12-k+\nu;x-k+1;ab)}{(x-k)!},\qquad
    \mbox{if $x\ge k$};\\
    \left(S(G_\xi)^{-1}\un x,\un k\right)_{\ell^2(\Z)}&=
    (1-\xi)^x {a^{k-x}}
    \left(\frac{\Gamma(k+\nu+\frac12)\Gamma(k-\nu+\frac12)}
    {\Gamma(x+\nu+\frac12)\Gamma(x-\nu+\frac12)}\right)^{\frac12}
    \times\\&\qquad\times 
    \frac{{}_2F_1
    (\tfrac12-x-\nu,\tfrac12-x+\nu;k-x+1;ab)}{(k-x)!},\qquad
    \mbox{if $x\le k$}.
  \end{align*}
  It is known that the expression $\frac{{}_2F_1(A,B;C;w)}{\Gamma(C)}$ is well-defined for all $C\in\C$, and by \cite[2.1.(3)]{Erdelyi1953} we see that
  \begin{equation*}
    \frac{{}_2F_1(A,B;n+1;w)}{\Gamma(n+1)}
    =\frac{{}_2F_1(A-n,B-n;-n+1;w)}
    {(A-n)_{n}(B-n)_{n}\Gamma(-n+1)}w^{-n},\qquad n=1,2,\dots.
  \end{equation*}
  Let us apply this identity in the case $x\le k$ above:
  \begin{align*}
    &
    \frac{{}_2F_1
    (\tfrac12-x-\nu,\tfrac12-x+\nu;k-x+1;ab)}{(k-x)!}
    \\&\qquad 
    =\frac{{}_2F_1
    (\tfrac12-k-\nu,\tfrac12-k+\nu;x-k+1;ab)}{\Gamma(x-k+1)}\cdot
    \frac{\Gamma(\tfrac12+x-\nu)\Gamma(\tfrac12+x+\nu)}
    {\Gamma(\tfrac12+k-\nu)\Gamma(\tfrac12+k+\nu)}
    (ab)^{x-k}
  \end{align*}
  (we have also used the fact that $(-1)^m\Gamma({\textstyle\frac12}+m\pm\nu)=\frac{\Gamma(\frac12+\nu) \Gamma(\frac12-\nu)}{\Gamma(\frac12\mp\nu-m)}$ for $m\in\Z$, see (\ref{pochhammer})). Thus, we get the desired result (\ref{matrix_el_phw}) for small $\xi$, and hence for all $\xi\in(0,1)$ by analyticity. This concludes the proof.
\end{proof}

\begin{rmk}\label{rmk:D_sl2}
  The operator $\mathfrak{D}_{\al,\xi}$ acting in $\ell^2(\Z)$, can be written through the operators of the representation $S$ of $\mathfrak{sl}(2,\C)$ as follows:
  \begin{equation*}
    (1-\xi)^{-1}\mathfrak{D}_{\al,\xi}=
    \frac{\sqrt\xi}{1-\xi}(S(U)+S(D))
    -\frac12\frac{1+\xi}{1-\xi}S(H)=-S(H_\xi),
  \end{equation*}
  where 
  \begin{equation*}
    H_\xi:=\frac12G_\xi HG_\xi^{-1}\in\mathfrak{sl}(2,\C).
  \end{equation*}
  Indeed, this is verified by a simple matrix computation (see (\ref{Gxi_df})):
  \begin{equation*}
    \begin{array}{r}
      \displaystyle
      \frac12G_\xi HG_\xi^{-1}=
      \frac1{2(1-\xi)}
      \left[
      \begin{matrix}
	1&\sqrt\xi\\
	\sqrt\xi&1
      \end{matrix}
      \right]
      \left[
      \begin{matrix}
	1&0\\
	0&-1
      \end{matrix}
      \right]
      \left[
      \begin{matrix}
	1&-\sqrt\xi\\
	-\sqrt\xi&1
      \end{matrix}
      \right]
      =\left[
      \begin{array}{cc}
	\frac12\frac{1+\xi}{1-\xi}&-\frac{\sqrt\xi}{1-\xi}\\
	\rule{0pt}{14pt}
	\frac{\sqrt\xi}{1-\xi}&-\frac12\frac{1+\xi}{1-\xi}
      \end{array}
      \right].
    \end{array}
  \end{equation*}
  Operators corresponding to the matrix $H_\xi$ under other representations of $\mathfrak{sl}(2,\C)$ appear in our model twice more, see \S\ref{subsection:pre_generator}.
\end{rmk}

\subsection{Connection with Meixner and Krawtchouk polynomials}

The $z$-me\-asu\-res $M_{z,z',\xi}$ (\S\ref{subsection:z-measures}) for $\xi\in(0,1)$ and $(z,z')$ of principal or complementary series (see p. \pageref{param_zz'_assumptions}) are supported by the set of all ordinary partitions. As is known (e.g., see \cite{borodin2006meixner}), the $z$-measures admit two degenerate series of parameters: 
\begin{enumerate}[$\bullet$]
  \item (\emph{first degenerate series}) $\xi\in(0,1)$, and one of the numbers $z$ and $z'$ (say, $z$) is a nonzero integer while $z'$ has the same sign and, moreover, $|z'|>|z|-1$. 
  
  Here if $z=N=1,2,\dots$, then the measure $M_{z,z',\xi}(\si)$ vanishes unless $\ell(\si)\le N$. Likewise, if $z=-N$, $M_{z,z',\xi}(\si)=0$ if $\ell(\si')=\si_1$ exceeds $N$ ($\si'$ denotes the transposed Young diagram).
  
  \item (\emph{second degenerate series}) $\xi<0$, and $z=N$ and $z'=-N'$, where $N$ and $N'$ are positive integers. 
  
  In this case, the measure $M_{z,z',\xi}$ is supported by the (finite) set of all ordinary Young diagrams which are contained in the rectangle $N\times N'$ (that is, $\ell(\si)\le N$ and $\ell(\si')\le N'$).
\end{enumerate}

As is explained in the paper \cite{borodin2006meixner}, in the first degenerate series the functions $\psi_{\sh a}(\sh x;z,z',\xi)$ are expressed through the classical Meixner orthogonal polynomials (about their definition, e.g., see \cite[\S1.9]{Koekoek1996}). In the second degenerate series these functions are related to the Krawtchouk orthogonal polynomials \cite[\S1.10]{Koekoek1996}.

For our measures $\mathsf{M}_{\al,\xi}$ on strict partitions there exists only one degenerate series of parameters: $\al=-N(N+1)$ for some $N=1,2,\dots$, and $\xi<0$ (Remark \ref{rmk:degenerate_multiplicative}). In this case, the measure $\mathsf{M}_{\al,\xi}$ is supported by the set of all shifted Young diagrams which are contained inside the staircase shifted shape $(N,N-1,\dots,1)$. This case corresponds to the second degenerate series of the $z$-measures, and our functions $\boldsymbol\varphi_m$ are expressed through the Krawtchouk orthogonal polynomials. 

The measures $\mathsf{M}_{\al,\xi}$ in this case are interpreted as random point processes on the finite lattice $\{1,\dots,N\}$, and one could also define a suitable dynamics for these measures as is done below in \S\ref{sec:markov_processes}. The results of the present paper about the structure of the static and dynamical correlation kernels also hold for the degenerate model, and the Krawtchouk polynomials enter formulas for these correlation kernels.



\section{Static correlation functions} 
\label{sec:static_correlation_functions}

In this section we obtain a Pfaffian formula for the correlation functions of the point process $\mathsf{M}_{\al,\xi}$, and discuss the resulting Pfaffian kernel.

\subsection{Pfaffian formula}
\label{subsection:static_Pfaffian_formula}

Recall that by $\Z_{\ne0}$ we denote the set of all nonzero integers. For $x_1,\dots,x_n\in\N$ we put, by definition,
\begin{equation}\label{x_-k_convention}
  x_{-k}:=-x_k,\qquad k=1,\dots,n.
\end{equation}
We use this convention in the formulation of the next theorem. Let the function $\boldsymbol\Phi_{\al,\xi}$ on $\Z_{\ne0}\times \Z_{\ne0}$ be defined by (see \S\ref{sec:fermionic_fock_space} for definitions of objects below)
\begin{equation}\label{static_Pfaffian_kernel_1}
  \boldsymbol\Phi_{\al,\xi}(x,y):=(-1)^{x\wedge 0+y\wedge 0}
  \left(R(\widetilde {{G}}_\xi)^{-1}{\boldsymbol\phi}_x{\boldsymbol\phi}_y R(\widetilde {{G}}_\xi)\mathsf{vac},\mathsf{vac}\right),
\end{equation}
where the inner product is taken in $\mathop{\mathsf{Fock}}(\N)$. (For now $\boldsymbol\Phi_{\al,\xi}(x,y)$ is defined for $x,y\in\Z_{\ne0}$, but in \S\ref{subsection:static_Pfaffian_kernel} we extend the definition of $\boldsymbol\Phi_{\al,\xi}(x,y)$ to zero values of $x,y$ in a natural way. See also Remark \ref{rmk:intro_xy=0}.1.) In this subsection we prove the following:

\begin{thm}\label{thm:static_Pfaffian_formula}
  The correlation functions $\rho^{(n)}_{\al,\xi}$ (\ref{df:static_correlation_functions})
  of the measures $\mathsf{M}_{\al,\xi}
  $ (\ref{df:mnuxi})
  are given by the following Pfaffian formula:
  \begin{equation}\label{static_Pfaffian_formula_thm_formula}
    \rho^{(n)}_{\al,\xi}(x_1,\dots,x_n)=
    \Pf( \hat{\boldsymbol\Phi}_{\al,\xi}\llbracket X\rrbracket ),
  \end{equation}
  where $X=\left\{ x_1,\dots,x_n \right\}\subset \N$ (here $x_j$'s are distinct), and $\hat{\boldsymbol\Phi}_{\al,\xi}\llbracket X\rrbracket$ is the skew-symmetric $2n\times2n$ matrix with rows and columns indexed by the numbers $1,2,\dots,n,-n,\dots,-2,-1$, and the $kj$-th entry in $\hat{\boldsymbol\Phi}_{\al,\xi}\llbracket X \rrbracket$ above the main diagonal is $\boldsymbol\Phi_{\al,\xi}(x_k,x_j)$, where $k,j=1,\dots,n,-n,\dots,-1$.\footnote{Theorem \ref{thm:static_Pfaffian_formula} is the same as Proposition 2 in \cite{Petrov2010}, the only difference is that in \cite{Petrov2010} the factor $(-1)^{\sum_{k=1}^{n}x_k}$ is put in front of the Pfaffian, and thus in the definition of the Pfaffian kernel in \cite{Petrov2010} there is no factor of the form $(-1)^{x\wedge0+y\wedge0}$.}
\end{thm}

Below in (\ref{Pfk_static_xy_all}) we write $\boldsymbol\Phi_{\al,\xi}(x,y)$ in terms of the Gauss hypergeometric function. For this reason, we call $\boldsymbol\Phi_{\al,\xi}$ the \textit{Pfaffian hypergeometric-type kernel}. Another form of a $2n\times 2n$ skew-symmetric matrix (constructed using the kernel $\boldsymbol\Phi_{\al,\xi}(x,y)$) which can be put in the right-hand side of (\ref{static_Pfaffian_formula_thm_formula}) is discussed below in \S\ref{subsection:skew_symm}. The above form $\hat{\boldsymbol\Phi}_{\al,\xi}\llbracket X\rrbracket$ is most useful when rewriting the Pfaffian in (\ref{static_Pfaffian_formula_thm_formula}) as a determinant, see Theorem \ref{thm:knuxi} and Proposition \ref{prop:A1_Determ_reduction} from Appendix. 

The rest of this subsection is devoted to proving Theorem \ref{thm:static_Pfaffian_formula}. Consider the following operators in $\ell^2(\Sb,\mathsf{M}_{\al,\xi})$:
\begin{equation*}
  \Delta_x\un\la:=\left\{
  \begin{array}{ll}
    \un\la,&\mbox{if $x\in\la
    $};\\
    0,&\mbox{otherwise},
  \end{array}
  \right.\qquad x\in\N.
\end{equation*}
Fix a finite subset $X=\left\{ x_1,\dots,x_n \right\}\subset \N$ and set $\Delta_{\llbracket X\rrbracket}:=\Delta_{x_1}\dots\Delta_{x_n}$. This is a diagonal operator of multiplication by a function which is the indicator of the event $\{\la\colon\la\supseteq X\}$. We view $\Delta_{\llbracket X\rrbracket}$ as an operator acting in $\ell^2(\Sb,\mathsf{M}_{\al,\xi})$. Since this operator is diagonal, it does not change under the isometry $I_{\al,\xi}\colon\ell^2(\Sb,\mathsf{M}_{\al,\xi})\to\mathop{\mathsf{Fock}}(\N)$ (\ref{istry}). Thus, $\Delta_{\llbracket X\rrbracket}$ also acts in $\mathop{\mathsf{Fock}}(\N)$ (in the same way).

The correlation functions $\rho^{(n)}_{\al,\xi}$ (\ref{df:static_correlation_functions}) of the measures $\mathsf{M}_{\al,\xi}$ (\ref{df:mnuxi}) clearly have the form
\begin{equation}\label{corr_f_averages}
  \rho^{(n)}_{(\al,\xi)}(x_1,\dots,x_n)=\mathsf{M}_{\al,\xi}
  \left( \la\colon\la\supseteq 
  \left\{ x_1,\dots,x_n \right\} \right)=
  (\Delta_{\llbracket X\rrbracket}\mathbf{1},\mathbf{1})_{\mathsf{M}_{\al,\xi}},
\end{equation}
where $\mathbf1\in\ell^2(\Sb,\mathsf{M}_{\al,\xi})$ is the constant identity function. Using (\ref{expectation_formula_fock_hard}) (or (\ref{expectation_formula_fock})), we can rewrite the correlation functions as
\begin{equation}\label{rho_n_RDeltaR}
  \rho^{(n)}_{(\al,\xi)}(x_1,\dots,x_n)=\left( R(\widetilde {{G}}_\xi)^{-1}
  \Delta_{\llbracket X\rrbracket}
  R(\widetilde {{G}}_\xi)\mathsf{vac},\mathsf{vac}
  \right).
\end{equation}
In this formula the operator $\Delta_{\llbracket X\rrbracket}$ acts in $\mathop{\mathsf{Fock}}(\N)$. Clearly, $\Delta_{\llbracket X\rrbracket}$ is expressed through the creation and annihilation operators in the Fock space $\mathop{\mathsf{Fock}}(\N)$ as
\begin{equation*}
  \Delta_{\llbracket X\rrbracket}=
  \prod\nolimits_{k=1}^{n}\phi_{x_k}^{\phantom{*}}\phi_{x_k}^{{*}}.
\end{equation*}
It is more convenient for us to rewrite $\Delta_{\llbracket X\rrbracket}$ using the anti-commutation relations for $\phi_x$ and $\phi_x^*$ (see (\ref{anti_commutation_pw})) as follows:
\begin{equation}\label{hX_in_order}
  \Delta_{\llbracket X\rrbracket}=\phi_{x_1}^{\phantom{*}}\dots\phi_{x_n}^{\phantom{*}}\phi_{x_n}^*\dots\phi_{x_1}^*
\end{equation}
(after moving all the $\phi_k$'s to the left and $\phi_k^*$'s to the right there is no change of sign).

Our next step is to write (\ref{rho_n_RDeltaR}) with $\Delta_{\llbracket X\rrbracket}$ given by (\ref{hX_in_order}) as the vacuum average functional ${\mathbf{F}}_\mathsf{vac}$ applied to a certain element of the Clifford algebra $\Cl(V)$ (\S\ref{subsection:Clifford}).

Recall that in \S\ref{subsection:creation_annihilation} we have defined the representation $\mathcal{T}$ of $\Cl(V)$ in $\mathop{\mathsf{Fock}}(\N)$ such that $\mathcal{T}(\boldsymbol v_x)={\boldsymbol\phi}_x$, $x\in\Z$, where $\left\{ \boldsymbol v_x \right\}_{x\in\Z}$ is the basis of $V=\ell^2(\Z)$ defined by (\ref{basis_vw}). Using the anti-commutation relations (\ref{anti_commutation_pw}), one can readily compute the commutators between the operators $\mathcal{T}(\boldsymbol v_x)$ and the operators of the representation $R$ (\ref{repres_R}):
\begin{equation*}
  \begin{array}{ccll}
    \left[ R(U),\mathcal{T}(\boldsymbol v_x) \right]&=& 
    2^{(\delta(x)-\delta(x+1))/2}\sqrt{x(x+1)+\al}\cdot 
    \mathcal{T}(\boldsymbol v_{x+1});\\
    \rule{0pt}{12pt}
    \left[ R(D),\mathcal{T}(\boldsymbol v_x) \right]&=& 
    2^{(\delta(x)-\delta(x-1))/2}\sqrt{x(x-1)+\al}\cdot 
    \mathcal{T}(\boldsymbol v_{x-1});\\
    \rule{0pt}{12pt}
    \left[ R(H),\mathcal{T}(\boldsymbol v_x) \right]&=& 2x\cdot \mathcal{T}(\boldsymbol v_x),&\quad 
    x\in\Z.
  \end{array}
\end{equation*}
These formulas motivate the following definition:
\begin{df}
  Let $\check{S}$ be the representation of the Lie algebra $\mathfrak{sl}(2,\C)$ in the (pre-Hilbert) space $V_{\mathrm{fin}}$ (consisting of of all finite linear combinations of the basis vectors $\left\{ \boldsymbol v_x \right\}_{x\in\Z}$) defined as:
  \begin{equation}\label{representation_S}
    \begin{array}{lcll}
      \check{S}(U)\boldsymbol v_x&:=&
      2^{(\delta(x)-\delta(x+1))/2}\sqrt{x(x+1)+\al}\cdot 
      \boldsymbol v_{x+1};\\
      \check{S}(D)\boldsymbol v_x&:=&
      2^{(\delta(x)-\delta(x-1))/2}\sqrt{x(x-1)+\al}\cdot 
      \boldsymbol v_{x-1};\\
      \check{S}(H)\boldsymbol v_x&:=&
      2x\cdot \boldsymbol v_x,&\quad x\in\Z.
    \end{array}
  \end{equation}
\end{df}
The representation $\check{S}$ is chosen in such a way that for all matrices $M\in\mathfrak{sl}(2,\C)$ and vectors $v\in V_{\mathrm{fin}}$ we have
\begin{equation}\label{sl2_adjoint_R_S}
  \left[ R(M),\mathcal{T}(v) \right]=\mathcal{T}( \check{S}(M)v )
\end{equation}
(the equality of operators in $\mathop{\mathsf{Fock}_{\mathrm{fin}}}(\N)$). This follows from definitions of $R$, $\mathcal{T}$, and $\check{S}$.

Comparing (\ref{representation_S}) and (\ref{Sch_representation}), we see that the representation $\check{S}$ is conjugate to the representation $S$ discussed in \S\ref{subsection:matrix_elements_phw} above. Namely, $\check{S}=Z^{-1}S Z$, where $Z$ is an operator in $V=\ell^2(\Z)$ defined by $Z\boldsymbol v_x:=2^{\delta(x)/2}\un x$, $x\in\Z$. This means that Proposition \ref{prop:integrability_ellf_Z} also holds for the representation $\check{S}$. In particular, $\check{S}$ lifts to a representation of the group $PSU(1,1)$ in the Hilbert space $V$.\footnote{Also by $\check{S}$ we denote the corresponding representations of $SU(1,1)$ and $SU(1,1)^\sim$ in $V$ that are obtained from the representation $\check{S}$ of $PSU(1,1)$ by a trivial lifting procedure.} Note that due to the conjugation by $Z$, the representation $\check{S}$ is not unitary (but we do not need this property).

The next proposition (due to Olshanski \cite{Olshanski-fockone}) is a ``group level'' version of the identity (\ref{sl2_adjoint_R_S}).

\begin{prop}\label{prop:group_level_identity}
  For all $g\in SU(1,1)^\sim$ and all $v\in V$ we have
  \begin{equation}\label{RTR=TS}
    R(g)\mathcal{T}(v)R(g)^{-1}=\mathcal{T}(\check{S}(g)v)
  \end{equation}
  (the equality of operators in $\mathop{\mathsf{Fock}}(\N)$).
\end{prop}

\begin{proof}
  {\bf{}Step 1.\/} 
  Since the representation $\mathcal{T}$ is norm preserving, it suffices to take $v\in V$ from the dense subspace $V_{\mathrm{fin}}$. Without loss of generality, we can assume that $v=\boldsymbol v_x$ for some $x\in \Z$.

  {\bf{}Step 2.\/}
  Rewrite the claim (\ref{RTR=TS}) as
  \begin{equation}\label{RTR=TS_proof.5}
    R(g)\mathcal{T}(\boldsymbol v_x)=\mathcal{T}(\check{S}(g)\boldsymbol v_x) R(g).
  \end{equation}
  This is an equality of operators in the Hilbert space $\mathop{\mathsf{Fock}}(\N)$. It is enough to show that these operators agree on $\mathop{\mathsf{Fock}_{\mathrm{fin}}}(\N)$, which is true if
  \begin{equation}\label{RTR=TS_proof1}
    R(g)\mathcal{T}(\boldsymbol v_x)\un\la=\mathcal{T}\big(\check{S}(g)\boldsymbol v_x\big) R(g)\un\la\quad
    \mbox{for all $g\in SU(1,1)^\sim
    $, $x\in\Z
    $, 
    and $\la\in\Sb
    $}.
  \end{equation}

  {\bf{}Step 3.\/}
  Now let us prove that both sides of (\ref{RTR=TS_proof1}) are analytic functions in $g\in SU(1,1)^\sim$ with values in $\mathop{\mathsf{Fock}}(\N)$:
  \begin{enumerate}[$\bullet$]
    \item (left-hand side)
      The vector $\mathcal{T}(\boldsymbol v_x)\un\la$ belongs to $\mathop{\mathsf{Fock}_{\mathrm{fin}}}(\N)$, and hence is analytic for the representation $R$, see Proposition \ref{prop:integrability_ell2_fin}. This means that the function $g\mapsto R(g)\mathcal{T}(\boldsymbol v_x)\un\la$ is analytic.
    \item (right-hand side)
      By Proposition \ref{prop:integrability_ellf_Z} (and remarks before the present proposition), the function $g\mapsto \check{S}(g)\boldsymbol v_x$ is an analytic function with values in the Hilbert space $V$. Since $\mathcal{T}$ is continuous in the norm topology, $\mathcal{T}\big(\check{S}(g)\boldsymbol v_x\big)$ is an analytic function with values in the Banach space $\mathrm{End}\big(\mathop{\mathsf{Fock}}(\N)\big)$ of bounded operators in the space $\mathop{\mathsf{Fock}}(\N)$. On the other hand, the function $R(g)\un\la$ is also analytic (with values in $\mathop{\mathsf{Fock}}(\N)$). Therefore, the function $g\mapsto\mathcal{T}\big(\check{S}(g)\boldsymbol v_x\big) R(g)\un\la$ is analytic, too.
  \end{enumerate}

  {\bf{}Step 4.\/} 
  Now it remains to compare the Taylor series expansions of both sides of (\ref{RTR=TS_proof.5}) at $g=e$, the unity element of $SU(1,1)^{\sim}$. That is, we need to establish that for any $M\in\mathfrak{sl}(2,\C)$ and any $x\in\Z$:
  \begin{equation}\label{RTR=TR_proof2}
    \sum_{k=0}^{\infty}\frac{R(M)^ks^k}{k!}\mathcal{T}(\boldsymbol v_x)=
    \left( \sum_{l=0}^{\infty}
    \frac{\mathcal{T}\big(\check{S}(M)^l\boldsymbol v_x\big)s^l}{l!}
    \right)
    \left( \sum_{r=0}^{\infty}\frac{R(M)^rs^r}{r!} \right).
  \end{equation}
  This should be understood as an equality of formal power series in $s$ with coefficients being operators in $\mathop{\mathsf{Fock}_{\mathrm{fin}}}(\N)$. Let us divide both sides by the last formal sum, $\sum_{r=0}^{\infty}\frac{R(M)^rs^r}{r!}$. After that it can be readily verified that the identity (\ref{RTR=TR_proof2}) of formal power series is a corollary of the ``Lie algebra level'' commutation identity (\ref{sl2_adjoint_R_S}).

  This last step concludes the proof of the proposition.
\end{proof}

Define (the second equality holds because $\check{S}$ is a representation of $PSU(1,1)$)
\begin{equation}\label{vxi}
  v_{x,\xi}:=\check{S}(\widetilde {{G}}_\xi)^{-1}v_x=
  \check{S}(G_\xi)^{-1}v_x\in V,\qquad x\in\Z.
\end{equation}
Putting this together with Proposition \ref{prop:group_level_identity}, we can rewrite the correlation functions (\ref{rho_n_RDeltaR}) as the vacuum average (see Definition \ref{df:vacuum_average}):
\begin{equation}\label{clf_vacuum_average}
  \rho^{(n)}_{\al,\xi}(x_1,\dots,x_n)=
  {\mathbf{F}}_\mathsf{vac}
  \left( v_{x_1,\xi}\dots v_{x_n,\xi}v_{-x_n,\xi}\dots v_{-x_1,\xi} \right).
\end{equation}
Observe that for $x,y\in\Z_{\ne0}
$ we have
\begin{equation*}
  {\mathbf{F}}_\mathsf{vac}(v_{x,\xi}v_{y,\xi})=
  (-1)^{x\wedge 0+y\wedge 0}
  \left( 
  R(\widetilde {{G}}_\xi)^{-1}{\boldsymbol\phi}_x{\boldsymbol\phi}_yR(\widetilde {{G}}_\xi)\mathsf{vac},\mathsf{vac} 
  \right)=
  \boldsymbol\Phi_{\al,\xi}(x,y),
\end{equation*}
as in (\ref{static_Pfaffian_kernel_1}). Therefore, formula (\ref{clf_vacuum_average}) together with Wick's Theorem \ref{thm:Wick} immediately implies our Theorem \ref{thm:static_Pfaffian_formula}.

\subsection{Static Pfaffian kernel}
\label{subsection:static_Pfaffian_kernel}

Let us express our static Pfaffian kernel $\boldsymbol\Phi_{\al,\xi}(x,y)$ through the functions $\boldsymbol\varphi_m$ defined by (\ref{phw_al,xi}). This kernel is defined for $x,y\in\Z_{\ne0}$ and has the form (see the previous subsection)
\begin{equation*}
  \boldsymbol\Phi_{\al,\xi}(x,y)=
  {\mathbf{F}}_\mathsf{vac}(v_{x,\xi}v_{y,\xi})=
  \sum\nolimits_{k,l\in\Z}
  (v_{x,\xi},v_k)_{\ell^2(\Z)}(v_{y,\xi},v_l)_{\ell^2(\Z)}
  {\mathbf{F}}_\mathsf{vac}(v_kv_l)
\end{equation*}
(where the vectors $v_{x,\xi},v_{y,\xi}$ are defined by (\ref{vxi})). By definitions of \S\ref{subsection:creation_annihilation}, we have
\begin{equation}\label{Fvac_delta}
  {\mathbf{F}}_\mathsf{vac}(v_kv_l)=\left\{
  \begin{array}{ll}
    1,&\qquad\mbox{if $l=-k\ge0$},\\
    0,&\qquad\mbox{otherwise}.
  \end{array}
  \right.
\end{equation}
Therefore,
\begin{equation*}
  \boldsymbol\Phi_{\al,\xi}(x,y)=
  \sum\nolimits_{m=0}^{\infty}(v_{x,\xi},v_{-m})_{\ell^2(\Z)}
  (v_{y,\xi},v_m)_{\ell^2(\Z)}.
\end{equation*}

\begin{prop}
  For any $r,k\in\Z$ we have
  \begin{equation*}
    (v_{r,\xi},v_k)_{\ell^2(\Z)}=(-1)^{r\wedge0+k\wedge0}
    2^{(\delta(r)-\delta(k))/2}\boldsymbol\varphi_{-k}(r;\al,\xi),
  \end{equation*}
  where the functions $\boldsymbol\varphi_m$ are defined in \S\ref{subsection:an_orthonormal_basis_phw}.
\end{prop}
\begin{proof}
  By (\ref{vxi}) and then by (\ref{basis_vw}), 
  \begin{equation*}
    (v_{r,\xi},v_k)_{\ell^2(\Z)}=
    (\check{S}(G_{\xi})^{-1}v_{r},v_k)_{\ell^2(\Z)}=
    (-1)^{r\wedge0+k\wedge0}
    (\check{S}(G_{\xi})^{-1}\boldsymbol v_{r},\boldsymbol v_k)_{\ell^2(\Z)}.
  \end{equation*}
  Using the fact that $\check{S}=Z^{-1}S Z$ (see the discussion before Proposition \ref{prop:group_level_identity}) and Proposition \ref{prop:phw_matr_el}, we conclude the proof.
\end{proof}

Therefore, since $\boldsymbol\Phi_{\al,\xi}(x,y)$ is defined for $x,y\in\Z_{\ne0}$, we have (in our derivation we also used (\ref{phw_symm2})):
\begin{align}\label{Pfk_static_xy_all}
  \boldsymbol\Phi_{\al,\xi}(x,y)=
  (-1)^{x\wedge0+y\vee0}
  \sum\nolimits_{m=0}^{\infty}
  2^{-\delta(m)}
  \boldsymbol\varphi_m(x;\al,\xi)
  \boldsymbol\varphi_{m}(-y;\al,\xi).
\end{align}
In the rest of the paper we agree that by this formula the kernel $\boldsymbol\Phi_{\al,\xi}(x,y)$ is defined for arbitrary $x,y\in\Z$ (see also Remark \ref{rmk:intro_xy=0}.1). This is needed to view $\boldsymbol\Phi_{\al,\xi}$ as an operator in $\ell^2(\Z)$. One also has (see \S\ref{subsection:_twisting_})
\begin{align}\label{Pfk_static_phw}
  \boldsymbol\Phi_{\al,\xi}(x,y)=
  \sum\nolimits_{m=0}^{\infty}
  2^{-\delta(m)}
  \widetilde{\boldsymbol\varphi}_m(x;\al,\xi)
  \widetilde{\boldsymbol\varphi}_{m}(-y;\al,\xi).
\end{align}

\subsection{Interpretation through spectral projections}
One can interpret the kernel $\boldsymbol\Phi_{\al,\xi}$ through orthogonal spectral projections related to the difference operator $\widetilde{\mathfrak{D}}_{\al,\xi}$ defined by (\ref{D_al,xi_twisted}). Namely, the projection onto the positive part of the spectrum of $\widetilde{\mathfrak{D}}_{\al,\xi}$ has the form (see \S\ref{subsection:_twisting_})
\begin{equation*}
  \mathop{\mathrm{Proj}}\nolimits_{>0}
  (\widetilde{\mathfrak{D}}_{\al,\xi})(x,y)=
  \sum\nolimits_{m=1}^{\infty}
  \widetilde{\boldsymbol\varphi}_m(x;\al,\xi)
  \widetilde{\boldsymbol\varphi}_{m}(y;\al,\xi).
\end{equation*}
We also need the projection onto the zero eigenspace, which is simply
\begin{equation*}
  \mathop{\mathrm{Proj}}\nolimits_{=0}
  (\widetilde{\mathfrak{D}}_{\al,\xi})(x,y)=
  \widetilde{\boldsymbol\varphi}_0(x;\al,\xi)
  \widetilde{\boldsymbol\varphi}_{0}(y;\al,\xi).
\end{equation*}

From (\ref{Pfk_static_phw}) we get the following interpretation of our kernel:
\begin{prop}\label{prop:Pfk_proj}
  Viewing the static Pfaffian kernel $\boldsymbol\Phi_{\al,\xi}$ as an operator in $\ell^2(\Z)$, we have 
  \begin{equation*}
    \boldsymbol\Phi_{\al,\xi}=
    \big(\mathop{\mathrm{Proj}}\nolimits_{>0}
    (\widetilde{\mathfrak{D}}_{\al,\xi})+ 
    \tfrac12\mathop{\mathrm{Proj}}\nolimits_{=0}
    (\widetilde{\mathfrak{D}}_{\al,\xi})\big)\mathrm{R},  
  \end{equation*}
  where $\mathrm{R}\colon\ell^2(\Z)\to\ell^2(\Z)$ is the operator corresponding to the reflection of the lattice $\Z$ with respect to $0$: $(\mathrm{R} f)(x):=f(-x)$, $f\in\ell^2(\Z)$.
\end{prop}

Since $\mathrm{R}^2$ is the identity operator, we see that the operator $\boldsymbol\Phi_{\al,\xi}\mathrm{R}$ is a rank one perturbation of the orthogonal spectral projection operator corresponding to the positive part of the spectrum of $\widetilde{\mathfrak{D}}_{\al,\xi}$.

\subsection{Expression through the discrete hypergeometric kernel}\label{subsection:Pfk_Kzz}
Recall that the discrete hypergeometric kernel $\underline K_{z,z',\xi}(\sh x,\sh y)$ (where $\sh x,\sh y\in\Z'=\Z+\frac12$) serves as a determinantal kernel for the $z$-measures on ordinary partitions (\S\ref{subsection:z-measures}). Under a suitable choice of parameters $z,z'$, the functions involved in the formula for $\underline K_{z,z',\xi}(\sh x,\sh y)$ turn into our functions $\boldsymbol\varphi_m$, see (\ref{phw_psi}). This means that one could express our Pfaffian kernel $\boldsymbol\Phi_{\al,\xi}$ through the discrete hypergeometric kernel:
\begin{prop}\label{prop:Pfk_Kzz}
  For all $x,y\in\Z$ we have
  \begin{align}\label{Pfk_Kzz}
    \boldsymbol\Phi_{\al,\xi}(x,y)&=
    \tfrac12{(-1)^{x\wedge0+y\vee0}}
    \Big[
    \underline
    K_{\nu(\al)-\frac12,-\nu(\al)-\frac12}
    (x+\tfrac12,-y+\tfrac12)
    \\&\qquad\qquad\qquad\qquad+
    \underline
    K_{\nu(\al)+\frac12,-\nu(\al)+\frac12}
    (x-\tfrac12,-y-\tfrac12)
    \Big],
    \nonumber
  \end{align}  
  where $\underline K$ is the discrete hypergeometric kernel described in \S\ref{subsection:z-measures}, and $\nu(\al)$ is given by Definition \ref{df:nu(al)}.
\end{prop}
\begin{proof}
  Using (\ref{K_zz'xi_formula}) and (\ref{phw_psi}) with $d=-1$ and $d=0$, we observe that for $x,y\in\Z$:
  \begin{align}
    \label{Kzz_sum_1}
    \sum\nolimits_{m=1}^\infty 
    \boldsymbol\varphi_m(x;\al,\xi)
    \boldsymbol\varphi_m(y;\al,\xi)&=
    \underline K_{\nu(\al)-\frac12,-\nu(\al)-\frac12,\xi}
    (x+\tfrac12,y+\tfrac12);\\
    \label{Kzz_sum_2}
    \sum\nolimits_{m=0}^\infty 
    \boldsymbol\varphi_m(x;\al,\xi)
    \boldsymbol\varphi_m(y;\al,\xi)&=
    \underline K_{\nu(\al)+\frac12,-\nu(\al)+\frac12,\xi}
    (x-\tfrac12,y-\tfrac12).
  \end{align}
  Taking half sum and using (\ref{Pfk_static_xy_all}), we conclude the proof.
\end{proof}

A time-dependent version of (\ref{Pfk_Kzz}) is (\ref{Pfkd_Kzz}), which is obtained and used in \S\ref{sec:dynamical_pfaffian_kernel}. A similar identity for the (static) determinantal kernels in (\ref{knuxi_Kzz}) below.

\subsection{Reduction formulas}
It is possible to rewrite the Pfaffian hypergeometric-type kernel $\boldsymbol\Phi_{\al,\xi}(x,y)$ in a closed form (without the sum):
\begin{prop}\label{prop:Pfk_integrable}
  For any $x,y\in\Z$ we have
  \begin{align*}
    \boldsymbol\Phi_{\al,\xi}(x,y)&=
    \frac{(-1)^{x\wedge0+y\wedge0}\sqrt{\al\xi}}{2(1-\xi)}
    \times\\&\qquad\times
    \frac{\boldsymbol\varphi_0(x)\big(
    \boldsymbol\varphi_{1}(y)-\boldsymbol\varphi_{-1}(y)
    \big)-\boldsymbol\varphi_0(y)\big(
    \boldsymbol\varphi_{1}(x)-\boldsymbol\varphi_{-1}(x)
    \big)}{x+y}.
  \end{align*}
  For $x=-y$ there is a singularity in the numerator (this is seen using (\ref{phw_symm2})) as well as in the denominator. In this case the value of $\boldsymbol\Phi_{\al,\xi}(x,y)$ is understood according to the L'Hos\-pital's rule using the analytic expression for $\boldsymbol\varphi_m$ (\ref{phw_al,xi}).
\end{prop}
\begin{proof}
  There are several ways of establishing this fact. One could use represen\-ta\-tion-theoretic arguments as in the proof of Theorem 3 in \cite{Okounkov2001a}. Another way is to argue directly using the three-term relations for the functions $\boldsymbol\varphi_m$ (\ref{phw_three-term}) to simplify the sum (\ref{Pfk_static_xy_all}) similarly to the standard derivation of the Christoffel–Darboux formula for orthogonal polynomials. 
  
  We use Proposition \ref{prop:Pfk_Kzz} together with the existing closed form expression for $\underline K_{z,z',\xi}$ \cite[Proposition 3.10]{borodin2006meixner}:\footnote{In fact, \cite[Proposition 3.10]{borodin2006meixner} itself is established using the three-term relations for the functions $\psi_{\sh a}$ (\ref{psi_zz'xi_three-term}).}
  \begin{equation*}
    \underline K_{z,z',\xi}(\sh x,\sh y)=
    \frac{\sqrt{zz'\xi}}
    {1-\xi}
    \frac{
    \psi_{-\frac12}(\sh x)
    \psi_{\frac12}(\sh y)
    -
    \psi_{\frac12}(\sh x)
    \psi_{-\frac12}(\sh y)
    }{\sh x-\sh y},\qquad 
    \sh x,\sh y\in\Z'.
  \end{equation*}
  (of course, the parameters of the functions $\psi$ above are $z,z',\xi$). We plug this formula into (\ref{Pfk_Kzz}), and then express each function $\psi_{\sh a}$ through $\boldsymbol\varphi_m$ using (\ref{phw_psi}) with $d=-1$ and $d=0$. Observe that for such $d$ we have $z(\al)z'(\al)=\al$. After that we apply (\ref{phw_symm2}) to simplify the resulting expression. This concludes the proof.
\end{proof}

\begin{corollary}[Reduction formulas for $\boldsymbol\Phi_{\al,\xi}$]
  \label{corollary:reduction_formulas}
  For all $x,y\in\Z$ we have:
  \begin{enumerate}[\rm(1)]
    \item $\boldsymbol\Phi_{\al,\xi}(x,-y)= \boldsymbol\Phi_{\al,\xi}(y,-x)$;
    \item $\boldsymbol\Phi_{\al,\xi}(x,-y)= -\boldsymbol\Phi_{\al,\xi}(-x,y)$ if $x\ne y$;
    \item $(x+y)\boldsymbol\Phi_{\al,\xi}(x,y)= (x-y)\boldsymbol\Phi_{\al,\xi}(x,-y)$ (note that $\boldsymbol\Phi_{\al,\xi}(x,x)=0$ for all $x\ne0$).
  \end{enumerate}
\end{corollary}
\begin{proof}
  Claim (1) is best seen from (\ref{Pfk_Kzz}), because the kernel $\underline K$ is symmetric. Claims (2) and (3) follow from Proposition \ref{prop:Pfk_integrable} and (\ref{phw_symm2}).
\end{proof}



\section{Static determinantal kernel} 
\label{sec:static_determinantal_kernel}

In this section we compute and discuss the determinantal correlation kernel $\mathbf{K}_{\al,\xi}$ of the point process $\mathsf{M}_{\al,\xi}$ on $\N$. We also consider the Plancherel degeneration (\ref{Pl_degen}) of the measures $\mathsf{M}_{\al,\xi}$ and of the kernel $\mathbf{K}_{\al,\xi}$. This section completes the proof of Theorem \ref{thm:A} from \S\ref{sec:model_and_results}.

\subsection{Hypergeometric-type kernel $\mathbf{K}_{\al,\xi}$}

\begin{thm}\label{thm:knuxi}
  For all $\al>0$ and $0<\xi<1$, the point process $\mathsf{M}_{\al,\xi}$ on $\N$ is determinantal. Its correlation kernel $\mathbf{K}_{\al,\xi}$ can be expressed in several ways (here $x,y\in\N$):
  \begin{enumerate}[\rm(1)]
    \item As an infinite sum
    \begin{equation}\label{knuxi-sum}
      \mathbf{K}_{\al,\xi}(x,y)=
      \frac{2\sqrt{xy}}{x+y}\sum_{m=0}^\infty
      2^{-\delta(m)}
      \boldsymbol\varphi_m(x;\al,\xi)
      \boldsymbol\varphi_m(y;\al,\xi)
    \end{equation}
     (the functions $\boldsymbol\varphi_m$ are defined in \S\ref{subsection:an_orthonormal_basis_phw}).
    \item In an integrable form
    \begin{equation}\label{knuxi-integrable}
      \mathbf{K}_{\al,\xi}(x,y)=
      \frac{\sqrt{\al\xi xy}}{1-\xi}\cdot
      \frac{P(x)Q(y)-Q(x)P(y)}{x^2-y^2},
    \end{equation}
    where $P(x):=\boldsymbol\varphi_0(x;\al,\xi)$ and $Q(x):= \boldsymbol\varphi_1(x;\al,\xi)- \boldsymbol\varphi_{-1}(x;\al,\xi)$.
    \item In terms of the discrete hypergeometric kernel of the $z$-measures (\S\ref{subsection:z-measures})
    \begin{align}
      \label{knuxi_Kzz}
      \widetilde{\mathbf{K}}_{\al,\xi}(x,y)&=
      \underline K_{\nu(\al)+\frac12,-\nu(\al)+\frac12,\xi}
      (x-\tfrac12,y-\tfrac12)\\&\qquad+
      (-1)^y \underline K_{\nu(\al)-\frac12,-\nu(\al)-\frac12,\xi}
      (x+\tfrac12,-y+\tfrac12),
      \nonumber
    \end{align}
    where we have denoted $\widetilde{\mathbf{K}}_{\al,\xi}(x,y):=
    \sqrt{\dfrac xy}\cdot\mathbf{K}_{\al,\xi}(x,y)$. 
    \item Viewed as an operator in $\ell^2(\N)$, $\widetilde{\mathbf{K}}_{\al,\xi}$ can be interpreted in terms of orthogonal spectral projections corresponding to the difference operator $\widetilde{\mathfrak{D}}_{\al,\xi}$ (\ref{D_al,xi_twisted}) as follows (we restrict the operator below to $\ell^2(\N)\subset\ell^2(\Z)$):
    \begin{equation*}
      \widetilde{\mathbf{K}}_{\al,\xi}=
      \big(\mathop{\mathrm{Proj}}\nolimits_{>0}(
      \widetilde{\mathfrak{D}}_{\al,\xi})+\tfrac12
      \mathop{\mathrm{Proj}}\nolimits_{=0}(
      \widetilde{\mathfrak{D}}_{\al,\xi})
      \big)\big(\mathbf{I}+\mathrm{R}\big),
    \end{equation*}
    where $\mathbf{I}$ is the identity operator and $\mathrm{R}$ is the reflection, see Proposition \ref{prop:Pfk_proj}.
  \end{enumerate}
\end{thm}
The expression $\widetilde{\mathbf{K}}_{\al,\xi}(x,y)$ is a so-called \emph{gauge transformation} of the original correlation kernel $\mathbf{K}_{\al,\xi}$, that is, $\widetilde{\mathbf{K}}_{\al,\xi}$ is related to $\mathbf{K}_{\al,\xi}$ by a conjugation by a diagonal matrix. This means that the $\N\times\N$ matrix $\widetilde{\mathbf{K}}_{\al,\xi}$ can also serve as a correlation kernel for the point process $\mathsf{M}_{\al,\xi}$.

\begin{proof}
  The fact that the process $\mathsf{M}_{\al,\xi}$ is determinantal is guaranteed by Lemma \ref{lemma:L-ensemble}. On the other hand, the reduction formulas for the Pfaffian kernel $\boldsymbol\Phi_{\al,\xi}$ (Corollary \ref{corollary:reduction_formulas}) allow us to apply Proposition \ref{prop:A1_Determ_reduction} from Appendix. This implies that 
  \begin{equation*}
    \rho^{(n)}_{\al,\xi}(X)=
    \Pf(\hat{\boldsymbol\Phi}_{\al,\xi}\llbracket X\rrbracket)=
    \det\left[ \mathbf{K}_{\al,\xi}(x_k,x_j) 
    \right]_{k,j=1}^{n},
  \end{equation*}
  where $X=\left\{ x_1,\dots,x_n \right\}\subset \N$ (with pairwise distinct $x_j$'s), $\hat{\boldsymbol\Phi}_{\al,\xi}\llbracket X\rrbracket$ is the skew-symmetric $2n\times 2n$ matrix introduced in Theorem \ref{thm:static_Pfaffian_formula}, and $\mathbf{K}_{\al,\xi}$ is related to $\boldsymbol\Phi_{\al,\xi}$ as
  \begin{equation*}
    \mathbf{K}_{\al,\xi}(x,y)=
    \frac{2\sqrt{xy}}{x+y}
    \boldsymbol\Phi_{\al,\xi}(x,-y),\qquad x,y\in\N.
  \end{equation*}

  This gives an argument (independently of Lemma \ref{lemma:L-ensemble}) that the process $\mathsf{M}_{\al,\xi}$ is determinantal. Moreover, this also provides us with explicit formulas for the kernel $\mathbf{K}_{\al,\xi}$. Namely, claims 1 and 2 of the present theorem directly follow from the expressions of $\boldsymbol\Phi_{\al,\xi}$ as a series (\ref{Pfk_static_xy_all}) and in a closed form (Proposition \ref{prop:Pfk_integrable}).
  
  To prove claims 3 and 4, observe that
  \begin{equation*}
    \mathbf{K}_{\al,\xi}(x,y)=
    \sqrt{\frac yx}
    \left[
    \frac{x-y}{x+y}+1
    \right]\boldsymbol\Phi_{\al,\xi}(x,-y)=
    \sqrt{\frac yx}
    \left[
    \boldsymbol\Phi_{\al,\xi}(x,y)+
    \boldsymbol\Phi_{\al,\xi}(x,-y)
    \right]
  \end{equation*}
  (the last equality is by Corollary \ref{corollary:reduction_formulas}.(3)), so
  \begin{equation*}
    \widetilde{\mathbf{K}}_{\al,\xi}(x,y)
    =
    \boldsymbol\Phi_{\al,\xi}(x,y)+
    \boldsymbol\Phi_{\al,\xi}(x,-y).
  \end{equation*}
  Now we see that claim 3 follows from (\ref{Pfk_static_xy_all}) and (\ref{Kzz_sum_1})--(\ref{Kzz_sum_2}), and claim 4 is due to Proposition \ref{prop:Pfk_proj}. This concludes the proof.
\end{proof}

\subsection{Comments to Theorem \ref{thm:knuxi}}
\label{subsection:comments_8.1}

\newcounter{comment82}
\setcounter{comment82}{1}

{\bf\arabic{comment82}\addtocounter{comment82}{1}.} Formulas (\ref{knuxi-sum}) and (\ref{knuxi-integrable}) for the correlation kernel $\mathbf{K}_{\al,\xi}$ are the same as the statements of Theorems 2.1 and 2.2 in \cite{Petrov2010}. This can be seen from the expression (\ref{phw_al,xi}) for the functions $\boldsymbol\varphi_m$.

{\bf\arabic{comment82}\addtocounter{comment82}{1}.} It is possible to obtain double contour integral expressions for the kernel $\mathbf{K}_{\al,\xi}(x,y)$ (given in \cite[Propositions 3 and 4]{Petrov2010}). They can be derived from (\ref{knuxi-sum}) in the same way as in the proof of \cite[Theorem 3.3]{borodin2006meixner}.

{\bf\arabic{comment82}\addtocounter{comment82}{1}.} 
The form (\ref{knuxi-integrable}) of the kernel $\mathbf{K}_{\al,\xi}$ is called \textit{integrable} because the operator (\ref{knuxi-integrable}) in $\ell^2(\N)$ can be viewed as a discrete analogue of an integrable operator (if we take $x^2$ and $y^2$ as variables). About integrable operators, e.g., see \cite{its1990differential}, \cite{deift1999integrable}. Discrete integrable operators are discussed in \cite{Borodin2000riemann} and \cite[\S6]{Borodin2000a}. This remark is also applicable to the kernel (\ref{kte_integrable}) below.

{\bf\arabic{comment82}\addtocounter{comment82}{1}.} Relation (\ref{knuxi_Kzz}) between the (determinantal) correlation kernels of  the measures $\mathsf{M}_{\al,\xi}$ on strict partitions and the $z$-measures on ordinary partitions, respectively, seems to be purely formal and have no consequences at the level of random point processes. About the behavior of (\ref{knuxi_Kzz}) in a scaling limit as $\xi\nearrow1$ see Remark \ref{rmk:knuxi_Kzz_whit_limit} below.

{\bf\arabic{comment82}\addtocounter{comment82}{1}.} Consider the $\N\times\N$ matrix $\mathbf{L}_{\al,\xi}$ which is defined by (\ref{L_kernel_generic}), where $w(x)=w_{\al,\xi}(x)$ is given by (\ref{df:w_al,xi}). Then one can show similarly to the proof of Theorem 3.3 in \cite{Borodin2000a} (and also using the identities from Appendix in that paper) that $\mathbf{K}_{\al,\xi}=\mathbf{L}_{\al,\xi}(1+\mathbf{L}_{\al,\xi})^{-1}$. That is, the kernel $\mathbf{K}_{\al,\xi}$ is precisely the one given by Lemma \ref{lemma:L-ensemble}.

\subsection{Plancherel degeneration}\label{subsection:Plancherel_degeneration}

Here we consider the Plancherel degeneration (\ref{Pl_degen}) of the hypergeometric-type kernel $\mathbf{K}_{\al,\xi}$ studied above in this section. Denote by $J_k$ the Bessel function (of the first kind) of order $k$ and argument $2\sqrt\te$:
\begin{equation}\label{BesselJ}
  J_k:=J_k(2\sqrt\te)=\sum_{r=0}^{\infty}
  \frac{(-1)^{r}\te^{r+\frac k2}}{r!\Gamma(r+k+1)},\qquad k\in\Z.
\end{equation}

\begin{thm}\label{thm:Plancherel_static}
  Under the Plancherel degeneration (\ref{Pl_degen}), the point processes $\mathsf{M}_{\al,\xi}$ on $\N$ converge to the poissonized Plancherel measure $\mathsf{Pl}_\te$. This is a determinantal point process on $\N$ supported by finite configurations. The correlation kernel $\mathbf{K}_{\te}(x,y)$ of $\mathsf{Pl}_\te$ can be expressed through the Bessel function in two ways: as a series
  \begin{equation}\label{kte_series}
    \mathbf{K}_{\te}(x,y)=
    \frac{2\sqrt{xy}}{x+y}\sum_{m=0}^{\infty}2^{-\delta(m)}
    J_{m+x}J_{m+y},
  \end{equation}
  and in an integrable form 
  \begin{equation}\label{kte_integrable}
    \mathbf{K}_{\te}(x,y)=\frac{2\sqrt{xy}}{x^2-y^2}
    \Big(\sqrt\te J_{x-1}J_{y}-
    \sqrt\te J_{y-1}J_x-{\textstyle\frac12}(x-y)J_xJ_y\Big).
  \end{equation}
\end{thm}

In principal, one can express the kernel $\mathbf{K}_{\te}$ through the discrete Bessel kernel of \cite{Borodin2000b} similarly to (\ref{knuxi_Kzz}) above, but we do not focus on this expression.

We will discuss three ways of proving Theorem \ref{thm:Plancherel_static}. The fact that formulas (\ref{kte_series}) and (\ref{kte_integrable}) are equivalent can be obtained as in \cite[Proposition 2.9]{Borodin2000b}.

\medskip
\par\noindent
\textit{Proof of Theorem \ref{thm:Plancherel_static}.\,I.}
Formulas (\ref{kte_series}) and (\ref{kte_integrable}) for $\mathbf{K}_{\te}$ can be obtained from the corresponding formulas (\ref{knuxi-sum}) and (\ref{knuxi-integrable}) for $\mathbf{K}_{\al,\xi}$ via the Plancherel degeneration (\ref{Pl_degen}). Namely, under the Plancherel degeneration we have $\boldsymbol\varphi_m(x;\al,\xi)\to J_{m+x}$. (This can be obtained by a termwise limit from the hypergeometric series for $\boldsymbol\varphi_m$, this series converges rapidly for fixed $x$ and $m$.) From this one can readily derive (\ref{kte_series}). To obtain (\ref{kte_integrable}) from the Plancherel degeneration of (\ref{knuxi-integrable}), one should also use the three-term relations for the Bessel functions (e.g., see \cite[7.2.(56)]{Erdelyi1953}):
\begin{equation}\label{Bessel_3term}
  J_{k+1}-\tfrac{k}{\sqrt\te}J_k+J_{k-1}=0,\qquad k\in\Z.
\end{equation}
This concludes the proof.
\qed

\medskip

The three-term relations for the Bessel functions (\ref{Bessel_3term}) are obtained via the Plancherel degeneration from Proposition \ref{prop:phw_properties}.2) (or, equivalently, from (\ref{phw_three-term}), by the self-duality of $\boldsymbol\varphi_m$'s). This agrees with the general approach described in \cite{Olshansk2008-difference} of studying limits of determinantal point processes via corresponding limits of self-adjoint operators which ``control'' the processes (in the sense that the correlation kernels of the processes are spectral projections corresponding to these operators).

\medskip
\par\noindent
\textit{Proof of Theorem \ref{thm:Plancherel_static}.\,II.}
Another way of proof is to observe that the point process $\mathsf{Pl}_\te$ on $\N$ is again an L-ensemble (see Lemma \ref{lemma:L-ensemble}). Denote by $\mathbf{L}_{\te}$ the corresponding $\N\times\N$ matrix which is given by (\ref{L_kernel_generic}) with $w(x)=w_\te(x)=\frac{\te^x}{2(x!)^2}$. To prove the theorem it suffices (by Lemma \ref{lemma:L-ensemble}) to show that the kernel $\mathbf{K}_{\te}$ has the form $\mathbf{K}_{\te}=\mathbf{L}_{\te}(1+\mathbf{L}_{\te})^{-1}$. This is equivalent to a certain identity for the Bessel functions which is readily verified using, e.g., Lemma 2.4 in \cite{Borodin2000b}. Equivalently, one may say that this identity is the Plancherel degeneration of the one in \S\ref{subsection:comments_8.1}.5.\qed

\medskip
\par\noindent
\textit{Proof of Theorem \ref{thm:Plancherel_static}.\,III.}
This way of proving the theorem uses the result of Matsumoto \cite[Thm. 3.1]{Matsumoto2005} that states that the correlation functions $\rho^{(n)}_\te(x_1,\dots,x_n)$ of the poissonized Plancherel measure $\mathsf{Pl}_\te$ have Pfaffian form similar to (\ref{static_Pfaffian_formula_thm_formula}) with the Pfaffian kernel $\boldsymbol\Phi_{\te}$ given by 
\begin{equation*}
  \boldsymbol\Phi_{\te}(x,y)=\frac12(-1)^{x\wedge 0+y\wedge 0}
  \left[ z^xw^{y} \right]\left\{
  e^{\sqrt\te(z-1/z)}
  e^{\sqrt\te(w-1/w)}\frac{z-w}{z+w}\right\},\qquad x,y\in\Z,
\end{equation*}
where $\left[ z^xw^{y} \right]\left\{ \dots \right\}$ denotes the coefficient by $z^xw^{y}$. The function $e^{\sqrt\te(z-1/z)}$ is the generating series for the Bessel functions $J_{k}(2\sqrt\te)$ \cite[7.2.(25)]{Erdelyi1953}, and thus
\begin{equation*}
  \boldsymbol\Phi_{\te}(x,y)=(-1)^{x\wedge 0+y\vee 0}
  \sum\nolimits_{m=0}^{\infty}
  2^{-\delta(m)}J_{m+x}J_{m-y}.
\end{equation*}
For the Pfaffian kernel $\boldsymbol\Phi_{\te}$ we have the same reduction formulas as in Corollary \ref{corollary:reduction_formulas}. They can be verified independently, or obtained from the reduction formulas for $\boldsymbol\Phi_{\al,\xi}$, because $\boldsymbol\Phi_{\te}$ is the Plancherel degeneration of $\boldsymbol\Phi_{\al,\xi}$. Thus, from Proposition \ref{prop:A1_Determ_reduction} it follows that the poissonized Plancherel measure is a determinantal point process with the correlation kernel given by $\mathbf{K}_{\te}(x,y)=\frac{2\sqrt{xy}}{x+y}\boldsymbol\Phi_{\te}(x,-y)$, which is exactly formula (\ref{kte_series}) for $\mathbf{K}_{\te}$. \qed

\medskip

Theorems \ref{thm:knuxi} and \ref{thm:Plancherel_static} constitute Theorem \ref{thm:A} from \S\ref{sec:model_and_results}.



\section{Markov processes} 
\label{sec:markov_processes}

In this section we explain in detail the construction of the dynamical model on strict partitions described in \S\ref{subsection:Dynamical_model}.

The construction of our Markov processes on the Schur graph is similar to Borodin-Olshanski's construction \cite{Borodin2006} of the Markov processes on the Young graph which preserve the $z$-measures. In contrast to \cite{Borodin2006}, we restrict our attention to the stationary (time homogeneous) case, that is, we assume that the parameter $\xi$ does not vary in time. The introduction of the non-stationary processes in \cite{Borodin2006} was motivated by the technique of handling the stationary case (in particular, by the method of the computation of the dynamical correlation functions). The technique that we use in the present paper does not require dealing with non-stationary processes.

\subsection{Defining Markov processes in terms of jump rates}\label{subsection:backward_equations}

Let us first recall some basic notions and facts concerning Markov processes. Let $E$ be a finite or countable space. Assume that we have a continuous time homogeneous Markov process on $E$ with the time parameter $t\in\R_{\ge0}$. By $({\mathbb{P}}(t))_{t\ge0}$ denote the \textit{transition probabilities} of this Markov process. That is, each ${\mathbb{P}}(t)$ is a $E\times E$ matrix, and ${\mathbb{P}}_{ab}(t)$ (where $a,b\in E$) is the probability that the process starting from the state $a$ will be at the state $b$ after time $t$. The matrices $\left( {\mathbb{P}}(t) \right)_{t\ge0}$ have the following properties:
\begin{enumerate}[(P1)]
  \item 
    ${\mathbb{P}}_{ab}(t)\ge0$ for all $t\ge0$ and ${\mathbb{P}}_{ab}(0)=\delta_{ab}$ for all $a,b\in E$;
  \item 
    $\sum\nolimits_{b\in E}{\mathbb{P}}_{ab}(t)=1$ for all $a\in E$;
  \item (\textit{Chapman-Kolmogorov equation})
    ${\mathbb{P}}(t+s)={\mathbb{P}}(t){\mathbb{P}}(s)$ for $t, s\ge0$, or, in matrix form, ${\mathbb{P}}_{ab}(t+s)=\sum_{c\in E}{\mathbb{P}}_{ac}(t){\mathbb{P}}_{cb}(s)$, where $a,b\in E$.
\end{enumerate}

Assume that there exists a $E\times E$ matrix ${\mathbbm{Q}}$ such that
\begin{equation}\label{df:transition_rates}
  {\mathbb{P}}_{ab}(t)=
  \delta_{ab}+{\mathbbm{Q}}_{ab}\cdot t+o(t),\qquad t\to0,\qquad a,b\in E.
\end{equation}
The elements of the matrix ${\mathbbm{Q}}$ are called the \textit{jump rates}. Note that (\ref{df:transition_rates}) implies that each ${\mathbb{P}}_{ab}(t)$ is continuous at $t=0$:
\begin{enumerate}[(P4)]
  \item 
    $\lim\limits_{t\downarrow0}{\mathbb{P}}_{ab}(t)=\delta_{ab}$, $a,b\in E$.
\end{enumerate}

A family of matrices satisfying (P1)--(P4) is a (\textit{continuous}) \textit{stochastic matrix semigroup}. One can say that ${\mathbbm{Q}}$ is the infinitesimal matrix of this semigroup, that is, ${\mathbbm{Q}}=\frac{d}{dt}{\mathbb{P}}(t)\big|_{t=0}$. From (\ref{df:transition_rates}) it is clear that
\begin{enumerate}[(Q1)]
  \item 
    ${\mathbbm{Q}}_{ab}\ge0$ for $a\ne b$ and ${\mathbbm{Q}}_{aa}\le 0$.
\end{enumerate}
We assume that the jump rates also have the property
\begin{enumerate}[(Q2)]
  \item 
    ${\mathbbm{Q}}_{aa}=-\sum\limits_{b\ne a}{\mathbbm{Q}}_{ab}$ for all $a\in E$.
\end{enumerate}

The property (Q2) implies (e.g., see \cite[Ch. 14.2]{karlin1981second}) that the jump rates ${\mathbbm{Q}}$ and the transition probabilities ${\mathbb{P}}(t)$ are related to each other via the system of \textit{Kolmogorov's backward equations}:
\begin{equation}\label{backward_equation}
  \frac{d {\mathbb{P}}_{ab}(t)}{dt}=
  \sum\nolimits_{c\in E}{\mathbbm{Q}}_{ac}{\mathbb{P}}_{cb}(t),\qquad a,b\in E,
\end{equation}
with the initial conditions
\begin{equation}\label{backward_equation_initial_condition}
  {\mathbb{P}}_{ab}(0)=\delta_{ab},\qquad a,b\in E.
\end{equation}
 
We would like to start with the jump rates ${\mathbbm{Q}}$ satisfying properties (Q1)--(Q2) and obtain a stochastic matrix semigroup $({\mathbb{P}}(t))_{t\ge0}$ by solving backward equations (\ref{backward_equation})--(\ref{backward_equation_initial_condition}). It is known that a solution in a wider class of \textit{substochastic matrix semigroups} (when the condition (P2) is replaced by $\sum_{b\in E}{\mathbb{P}}_{ab}(t)\le1$) always exists. Among all possible substochastic solutions there is a distinguished \textit{minimal} solution $(\bar {\mathbb{P}}(t))_{t\ge0}$, that is,
${\mathbb{P}}_{ab}(t)\ge \bar {\mathbb{P}}_{ab}(t)$ for $t\ge0$ and $a,b\in E$, where $({\mathbb{P}}(t))_{t\ge0}$ is any substochastic solution. A minimal solution can be constructed using an approximation method (e.g., see \cite[Ch. 14.3]{karlin1981second}). If the minimal solution is stochastic, then it is a unique solution of (\ref{backward_equation})--(\ref{backward_equation_initial_condition}) in the class of substochastic matrices. About solving Kolmogorov's backward equations see also \cite[Ch. III.2]{gikhman2004theoryII}.

If the system of backward equations (\ref{backward_equation})--(\ref{backward_equation_initial_condition}) has a unique solution $({\mathbb{P}}(t))_{t\ge0}$ (or, which is equivalent, the minimal solution of this system is stochastic), we say that the jump rates ${\mathbbm{Q}}$ define a continuous time homogeneous Markov process on $E$ (with transition probabilities ${\mathbb{P}}(t)$) that can start from any point and any probability distribution. A common sufficient condition for this is $\sup_{a\in E}|{\mathbbm{Q}}_{aa}|< +\infty$, which however does not hold in our case.

Let us recall another useful sufficient condition for the minimal solution of (\ref{backward_equation})--(\ref{backward_equation_initial_condition}) to be stochastic. We formulate it as in \cite[Prop. 4.3]{Borodin2006}, it also can be derived from the discussion of \cite[Ch. III.2]{gikhman2004theoryII}.

Let $X\subset E$ be a finite set and $a\in X$. By $\tau_{a,X}$ denote the time of the first exit from $X$ of the process starting at $a$. Though we do not know yet if the process itself is uniquely determined by its jump rates ${\mathbbm{Q}}$, the random variable $\tau_{a,X}$ can be constructed from ${\mathbbm{Q}}$ as follows. Contract all the states $b\in E\setminus X$ into one absorbing state $\tilde b$ with ${\mathbbm{Q}}_{\tilde b,c}=0$ for all $c\in X\cup\{ \tilde b \}$. On the finite set $X\cup\{ \tilde b \}$ the backward equations have a unique solution $(\tilde {\mathbb{P}}(t))_{t\ge0}$,\footnote{Indeed, because the state space is finite, one can readily see that the minimal solution is stochastic.} where $\tilde {\mathbb{P}}(t)$ are matrices with rows and columns indexed by the set $X\cup\{ \tilde b \}$. The distribution of the random variable $\tau_{a,X}$ then has the form
\begin{equation*}
  \mathsf{Prob}\left\{ \tau_{a,X}\le t \right\}
  =\tilde {\mathbb{P}}_{a,\tilde b}(t)
\end{equation*}
and is defined only in terms of the jump rates ${\mathbbm{Q}}$. 

\begin{prop}[{\cite[Prop. 4.3]{Borodin2006}}]\label{prop:Markov_processes_regularity}
  If for any $a\in E$, any $t\ge0$, and any $\epsilon>0$ there exists a finite set $X(\epsilon)\subset E$ such that
  \begin{equation*}
    \mathsf{Prob}\left\{ \tau_{a,X(\epsilon)}\le t \right\}\le \epsilon,
  \end{equation*}
  then the minimal solution of the system of Kolmogorov's backward equations (\ref{backward_equation})--(\ref{backward_equation_initial_condition}) is stochastic.
\end{prop}
In other words, the hypotheses of this proposition mean that the Markov process does not make infinitely many jumps in finite time.

\subsection{Birth and death processes}
\label{subsection:birth_and_death}

Here we discuss underlying birth and death processes on $\Z_{\ge0}$ involved in the construction of the Markov processes on strict partitions (see \S\ref{subsection:Dynamical_model}).

A general birth and death process on $E=\Z_{\ge0}$ is a continuous time homogeneous Markov process with jump rates 
$\{ {\mathbbm{q}}_{k,j} \}_{k,j\in\Z_{\ge0}}$ satisfying conditions (Q1)--(Q2) from the previous subsection, with an additional property that ${\mathbbm{q}}_{k,j}=0$ if $|k-j|>1$. This means that from any point $n\ge1$ of $\Z_{\ge0}$ the process can jump only to the neighbor points $n-1$ and $n+1$, and that from $0$ it can jump only to $1$.

The following necessary and sufficient condition is well-known and can be deduced, e.g., from \cite[Ch. III.2, Thm. 4]{gikhman2004theoryII}.
\begin{prop}
  \label{prop:birth_death_gen}
  The minimal solution of the system of Kolmogorov's backward equations (\ref{backward_equation})--(\ref{backward_equation_initial_condition}) for a birth and death process is stochastic iff 
  \begin{equation}\label{birth_death_uniqueness}
    \sum_{n=1}^{\infty}
    \left[ \frac1{{\mathbbm{q}}_{n,n+1}}+
    \frac{{\mathbbm{q}}_{n,n-1}}
    {{\mathbbm{q}}_{n-1,n}{\mathbbm{q}}_{n,n+1}}+\dots
    +\frac{{\mathbbm{q}}_{n,n-1}\dots {\mathbbm{q}}_{2,1}
    {\mathbbm{q}}_{1,0}}
    {{\mathbbm{q}}_{0,1}{\mathbbm{q}}_{1,2}\dots 
    {\mathbbm{q}}_{n-1,n}
    {\mathbbm{q}}_{n,n+1}}\right]
    =+\infty.
  \end{equation}
\end{prop}

Now let us turn to our concrete situation and define the birth and death process that we use in the present paper. Its jump rates depend on our parameters $\al>0$ and $0<\xi<1$ and are as follows:
\begin{align}\nonumber
    {\mathbbm{q}}_{n,n+1}&:= (1-\xi)^{-1}
    \xi(n+\al/2);&\\
    {\mathbbm{q}}_{n,n-1}&:= (1-\xi)^{-1}n;
    \label{birth_death_transition_rates}
    &\\
    \nonumber
    {\mathbbm{q}}_{n,n}&:= 
    -({\mathbbm{q}}_{n,n+1}+{\mathbbm{q}}_{n,n-1})
    =-(1-\xi)^{-1}\{\xi(n+\al/2)+n\},
\end{align}
where $n=0,1,2,\dots$. All other jump rates are zero. It is not hard to see that jump rates (\ref{birth_death_transition_rates}) satisfy condition (\ref{birth_death_uniqueness}). Thus, Proposition \ref{prop:birth_death_gen} together with the facts from \S\ref{subsection:backward_equations} implies that there exists a unique continuous time Markov process on $\Z_{\ge0}$ with jump rates (\ref{birth_death_transition_rates}) that can start from any point and any probability distribution. This process preserves the negative binomial distribution $\pi_{\al,\xi}$ (\ref{df:negbinom}) on $\Z_{\ge0}$ because $\pi_{\al,\xi}\circ {\mathbbm{q}}=0$, or, in matrix form,
\begin{equation}\label{birth_and_death_invariance}
  \sum\nolimits_{k=0}^{\infty}
  \pi_{\al,\xi}(k){\mathbbm{q}}_{k,j}=0
  \qquad\mbox{for all $j\in \Z_{\ge0}$}.
\end{equation}
Denote the equilibrium version of this process (that is, starting from the distribution $\pi_{\al,\xi}$) by $({\boldsymbol{n}}_{\al,\xi}(t))_{t\ge0}$. Since 
\begin{equation}\label{birth_and_death_reversible}
  \pi_{\al,\xi}(n){\mathbbm{q}}_{n,n+1}=\pi_{\al,\xi}(n+1){\mathbbm{q}}_{n+1,n}\qquad 
  \mbox{for all $n\in\Z_{\ge0}$},
\end{equation}
the process ${\boldsymbol{n}}_{\al,\xi}$ is reversible with respect to $\pi_{\al,\xi}$.

The transition probabilities of the process ${\boldsymbol{n}}_{\al,\xi}$ can be expressed through the Meixner orthogonal polynomials, see \cite[\S4.3]{Borodin2006}, and also \cite{KMG57BDClassif}, \cite{KMG58Linear} for a much more general formalism.

\begin{rmk}\label{rmk:birth_death_regularity}
  From, e.g., the discussion of \cite[\S4]{Borodin2006} it follows that the process ${\boldsymbol{n}}_{\al,\xi}$ satisfies the hypotheses of Proposition \ref{prop:Markov_processes_regularity}. This fact will be used in the next subsection in construction of Markov processes on strict partitions that ``extend'' the processes ${\boldsymbol{n}}_{\al,\xi}$.
\end{rmk}

\subsection{Markov processes on strict partitions}
\label{subsection:mpsp}

Here we define continuous time Markov processes on the set $\Sb$ of all strict partitions. They depend on our parameters $\al>0$ and $0<\xi<1$ and are defined in terms of jump rates (here $\la\in\Sb_n$, $n=0,1,\dots$):
\begin{align}
  \nonumber
    {\mathbbm{Q}}_{\la,\varkappa}&:=
    (1-\xi)^{-1}{\xi(n+\al/2)}
    p^{\uparrow}_\al({n,n+1})_{\la,\varkappa},
    \qquad\qquad\qquad
    \mbox{where $\varkappa\searrow\la$};\\\nonumber
    {\mathbbm{Q}}_{\la,\mu}
    &:=
    (1-\xi)^{-1}{n}p^{\downarrow}({n,n-1})_{\la,\mu},
    \qquad\qquad\qquad
    \mbox{where $\mu\nearrow\la$};\\
    \label{mpsp_transition_rates}
    {\mathbbm{Q}}_{\la,\la}&:= 
    -\sum\nolimits_{\varkappa\colon\varkappa\searrow\la}
    {\mathbbm{Q}}_{\la,\varkappa}
    -\sum\nolimits_{\mu\colon\mu\nearrow\la}
    {\mathbbm{Q}}_{\la,\mu}
    \\&\qquad\qquad=
    -(1-\xi)^{-1}\{\xi(\al/2+n)+n\}.
    \nonumber
\end{align}
All other jump rates are zero. Here $p^{\downarrow}({n,n-1})$ and $p^{\uparrow}_\al(n,n+1)$ are the down and up transition kernels, respectively. Recall that (see \S\ref{sec:schur_graph_and_multiplicative_measures}) if $\la\in\Sb_n$, $\mu\nearrow\la$, and $\varkappa\searrow\la$, we have
\begin{equation*}
  p^{\downarrow}({n,n-1})_{\la,\mu}=\frac{\gdim\mu}{\gdim\la},\qquad
  p^{\uparrow}_\al({n,n+1})_{\la,\varkappa}=
  \frac{\gdim\la}{\gdim\varkappa}\cdot
  \frac{\mathsf{M}_{\al,n+1}(\varkappa)}{\mathsf{M}_{\al,n}(\la)},
\end{equation*}
where $\gdim(\cdot)$ is given by (\ref{gdim}) and $\left\{ \mathsf{M}_{\al,n} \right\}$ is the multiplicative coherent system of measures on the Schur graph (\S\ref{subsection:multiplicative_measures}).

Under the projection $\Sb\to\Z_{\ge0}$, $\la\mapsto|\la|$, the $\Sb\times\Sb$ matrix ${\mathbbm{Q}}$ (\ref{mpsp_transition_rates}) turns into the $\Z_{\ge0}\times\Z_{\ge0}$ matrix ${\mathbbm{q}}$ of jump rates of the birth and death process ${\boldsymbol{n}}_{\al,\xi}$ from \S\ref{subsection:birth_and_death}. This means that the processes on $\Sb$ ``extend'' the birth and death processes ${\boldsymbol{n}}_{\al,\xi}$.

\begin{prop}\label{prop:mpsp_regularity}
  The minimal solution of the system of Kolmogorov's backward equations (\ref{backward_equation})--(\ref{backward_equation_initial_condition}) for the matrix ${\mathbbm{Q}}$ (\ref{mpsp_transition_rates}) is stochastic. 
\end{prop}
\begin{proof}
  Let $n, K$ be nonnegative integers, $n\le K$. Consider the random variable $\tau_{n,\left\{ 0,\dots,K-1 \right\}}$ for the  birth and death process ${\boldsymbol{n}}_{\al,\xi}$, that is, the time of the first exit from $\left\{ 0,1,\dots,K-1 \right\}$ of the process with jump rates ${\mathbbm{q}}$ (\ref{birth_death_transition_rates}) starting at $n$. 

  Let $\la\in\Sb_n$. Observe that the time of the first exit from $\Sb_0\cup\dots\cup\Sb_{K-1}$ of the process on strict partitions with jump rates ${\mathbbm{Q}}$ (\ref{mpsp_transition_rates}) starting at $\la$ has the same distribution as $\tau_{n,\left\{ 0,\dots,K-1 \right\}}$. Applying Remark \ref{rmk:birth_death_regularity} and Proposition \ref{prop:Markov_processes_regularity}, we conclude the proof. 
\end{proof}
  
Thus, the jump rates ${\mathbbm{Q}}$ (\ref{mpsp_transition_rates}) uniquely define a continuous time Markov process on $\Sb$ that can start from any point and any probability distribution.

Recall that the measures $\mathsf{M}_{\al,\xi}$ and $\left\{ \mathsf{M}_{\al,n} \right\}_{n\in\Z_{\ge0}}$ are related as 
\begin{equation}\label{mnuxi_mnun_relation}
  \mathsf{M}_{\al,\xi}(\la)
  =\pi_{\al,\xi}(n)\mathsf{M}_{\al,n}(\la),\qquad
  \la\in\Sb_n,
\end{equation}
where $\pi_{\al,\xi}$ is the negative binomial distribution (\ref{df:negbinom}). Moreover, the measure $\pi_{\al,\xi}$ is invariant for the birth and death process ${\boldsymbol{n}}_{\al,\xi}$ on $\Z_{\ge0}$ (see (\ref{birth_and_death_invariance})), and the measures $\mathsf{M}_{\al,n}$ are consistent with the up and down transition kernels (see \S\ref{subsection:coherent_systems}):
\begin{equation*}
  \mathsf{M}_{\al,n}\circ p^\uparrow_\al({n,n+1})
  =\mathsf{M}_{\al,n+1}\qquad\mbox{and}\qquad
  \mathsf{M}_{\al,n+1}\circ p^\downarrow({n+1,n})=\mathsf{M}_{\al,n}.
\end{equation*}
This implies that $\mathsf{M}_{\al,\xi}\circ {\mathbbm{Q}}=0$, and hence the process with jump rates (\ref{mpsp_transition_rates}) preserves the measure $\mathsf{M}_{\al,\xi}$ on $\Sb$. By $(\boldsymbol\la_{\al,\xi}(t))_{t\in[0,+\infty)}$ we denote the equilibrium version of this process.

The process $\boldsymbol\la_{\al,\xi}$ is reversible with respect to $\mathsf{M}_{\al,\xi}$ because $\mathsf{M}_{\al,\xi}(\la){\mathbbm{Q}}_{\la,\mu}=\mathsf{M}_{\al,\xi}(\mu){\mathbbm{Q}}_{\mu,\la}$ for all $\mu,\la\in\Sb$. Indeed, it suffices to consider $\mu\nearrow\la$. Let $|\la|=n$. Then
\begin{equation*}
  \begin{array}{l}
    \displaystyle
    \mathsf{M}_{\al,\xi}(\la){\mathbbm{Q}}_{\la,\mu}=
    \Big(\pi_{\al,\xi}(n)
    \frac{n}{1-\xi}\Big)\left(
    \mathsf{M}_{\al,n}(\la)
    p^{\downarrow}(n,n-1)_{\la,\mu}\right)
    \\\displaystyle\qquad
    =\rule{0pt}{18pt}
    \Big(\pi_{\al,\xi}(n-1)\frac{\xi(n-1+\frac\al2)}{1-\xi}\Big)
    \Big(\mathsf{M}_{\al,n-1}
    (\mu)p^{\uparrow}_\al({n-1,n})_{\mu,\la}\Big)=
    \mathsf{M}_{\al,\xi}(\mu){\mathbbm{Q}}_{\mu,\la}
  \end{array}
\end{equation*}
by (\ref{birth_and_death_reversible}), (\ref{mnuxi_mnun_relation}), and the definition of the up transition kernel $p^{\uparrow}_\al({n-1,n})$, see \S\ref{subsection:coherent_systems}.

\subsection{Pre-generator}\label{subsection:pre_generator}

Here we discuss the pre-generator of the Markov process $\boldsymbol\la_{\al,\xi}$. We now regard the $\Sb\times\Sb$ matrices $\left( {\mathbb{P}}_{\la,\mu}(t) \right)_{t\ge0}$ of transition probabilities of the process $\boldsymbol\la_{\al,\xi}$ as operators acting on functions on $\Sb$ (from the left):
\begin{equation*}
  ({\mathbb{P}}(t)f)(\la):=
  \sum\nolimits_{\mu\in\Sb}{\mathbb{P}}_{\la,\mu}(t)f(\mu).
\end{equation*}
The family $\left( {\mathbb{P}}(t) \right)_{t\ge0}$ is a Markov semigroup of self-adjoint contractive operators in the weighted space $\ell^2(\Sb,\mathsf{M}_{\al,\xi})$ (see \S\ref{subsection:expectation_formula} for the definition of $\ell^2(\Sb,\mathsf{M}_{\al,\xi})$).

The semigroup $\left( {\mathbb{P}}(t) \right)_{t\ge0}$ in $\ell^2(\Sb,\mathsf{M}_{\al,\xi})$ has a generator which is an unbounded operator. By ${\mathbbm{Q}}$ let us denote the restriction of this generator to $\ell_{\mathrm{fin}}^2(\Sb,\mathsf{M}_{\al,\xi})\subset\ell^2(\Sb,\mathsf{M}_{\al,\xi})$, the dense subspace of all finitely supported functions in $\ell^2(\Sb,\mathsf{M}_{\al,\xi})$. The operator ${\mathbbm{Q}}$ acts as
\begin{equation}\label{B_in_ell2_M}
  ({\mathbbm{Q}} f)(\la)=
  \sum\nolimits_{\mu\in\Sb}{\mathbbm{Q}}_{\la,\mu}f(\mu),
  \qquad f\in\ell_{\mathrm{fin}}^2(\Sb,\mathsf{M}_{\al,\xi}),
\end{equation}
where ${\mathbbm{Q}}_{\la,\mu}$ (\ref{mpsp_transition_rates}) are the jump rates of the process $\boldsymbol\la_{\al,\xi}$.

The operator ${\mathbbm{Q}}$ is symmetric with respect to the inner product $(\cdot,\cdot)_{\mathsf{M}_{\al,\xi}}$. Moreover, it is closable in $\ell^2(\Sb,\mathsf{M}_{\al,\xi})$, and its closure generates the semigroup $({\mathbb{P}}(t))_{t\ge0}$ (see Remark \ref{rmk:Qw_closable} below). That is, ${\mathbbm{Q}}$ is the \textit{pre-generator} of the process $\boldsymbol\la_{\al,\xi}$.

\begin{rmk}
  As a wider domain for the operator ${\mathbbm{Q}}$ (\ref{B_in_ell2_M}) one can take the space of all functions $f$ on $\Sb$ such that both $f$ and ${\mathbbm{Q}} f$ (defined by (\ref{B_in_ell2_M})) belong to $\ell^2(\Sb,\mathsf{M}_{\al,\xi})$. This space clearly includes finitely supported functions.
\end{rmk}

Using the isometry $I_{\al,\xi}\colon\ell^2(\Sb,\mathsf{M}_{\al,\xi})\to\ell^2(\Sb)$ (\ref{istry}), we get a symmetric operator ${\mathbb{B}}$ in $\ell_{\mathrm{fin}}^2(\Sb)$ and a Markov semigroup $({\mathbb{V}}(t))_{t\ge0}$ of self-adjoint contractive operators in $\ell^2(\Sb)$ corresponding to ${\mathbbm{Q}}$ and $({\mathbb{P}}(t))_{t\ge0}$, respectively.\footnote{We denote, e.g., the operators ${\mathbb{B}}$ and ${\mathbbm{Q}}$ by different symbols only to indicate in what spaces they act. Essentially, these operators are the same.} Let us compute the matrix elements of the operator ${\mathbb{B}}$ in the standard orthonormal basis $\left\{ \un\la \right\}_{\la\in\Sb}$.
\begin{prop}\label{prop:B_ufunc}
  We have
  \begin{align*}
    \nonumber
      {\mathbb{B}} \un\la&=
      \sum\nolimits_{\nu\in\Sb}
      ({\mathbb{B}}\un\la,\un\nu)\un\nu
      =
      -(1-\xi)^{-1}\{|\la|+\xi(|\la|+\frac\al2)\}\un\la
      \\&+\frac{\sqrt\xi}{1-\xi}
      \sum\nolimits_{\mu\colon\mu\nearrow\la}
      q_\al(\la/\mu)\un\mu+
      \frac{\sqrt\xi}{1-\xi}
      \sum\nolimits_{\varkappa\colon\varkappa\searrow\la}
      q_\al(\varkappa/\la)\un\varkappa.
  \end{align*}
  Here $q_\al$ is the function of a box defined by (\ref{ufunc}).
\end{prop}
\begin{proof}
  Fix $\la\in\Sb$, and for any $\nu\in\Sb$ one has
  \begin{align*}
      ({\mathbb{B}}\un\la,\un\nu)&=
      \big( 
      (\mathsf{M}_{\al,\xi}(\la))^{-\frac12}
      {\mathbbm{Q}}\un\la,
      (\mathsf{M}_{\al,\xi}(\nu))^{-\frac12}
      \un\nu
      \big)_{\mathsf{M}_{\al,\xi}}\\&=
      (\mathsf{M}_{\al,\xi}(\la)
      \mathsf{M}_{\al,\xi}(\nu))^{-\frac12}
      ({\mathbbm{Q}}\un\la,\un\nu)_{\mathsf{M}_{\al,\xi}}=
      {\mathsf{M}_{\al,\xi}(\nu)}^{\frac12}
      {\mathsf{M}_{\al,\xi}(\la)}^{-\frac12}
      {\mathbbm{Q}}_{\nu,\la}
  \end{align*}
  ($\mathbbm{Q}$ is given in (\ref{mpsp_transition_rates})), and Proposition follows from a direct computation.
\end{proof}

\begin{corollary}\label{corollary:Qr_sl2}
  The operator ${\mathbb{B}}$ in $\mathop{\mathsf{Fock}_{\mathrm{fin}}}(\N)$ has the form
  \begin{equation*}
    {\mathbb{B}}=-R(H_\xi)+{\tfrac\al4}\mathbf{I},
  \end{equation*}
  where $\mathbf{I}$ is the identity operator, the unitary representation $R$ of $\mathfrak{sl}(2,\C)$ in the Hilbert space $\mathop{\mathsf{Fock}_{\mathrm{fin}}}(\N)$ is defined in \S\ref{subsection:representation_R}, and $H_\xi$ is given in Remark \ref{rmk:D_sl2}.
\end{corollary}

\begin{proof}
  This is a straightforward consequence of Proposition \ref{prop:B_ufunc} (where, of course, we identify $\ell^2(\Sb)$ and $\mathop{\mathsf{Fock}}(\N)$) and the matrix computation in Remark \ref{rmk:D_sl2}.
\end{proof}

\begin{rmk}\label{rmk:Qw_closable}
  From the above corollary it follows that the operator ${\mathbb{B}}$ (with domain $\mathop{\mathsf{Fock}_{\mathrm{fin}}}(\N)$) is essentially self-adjoint because, by Proposition \ref{prop:integrability_ell2_fin}, all vectors of the space $\mathop{\mathsf{Fock}_{\mathrm{fin}}}(\N)$ are analytic for the operator $R(H_\xi)$. The same also holds for the operator $R(H)$ (corresponding to the case $\xi=0$). Moreover, the closure of ${\mathbb{B}}$ looks as $\overline {\mathbb{B}}=\frac\al4\mathbf{I}-\overline{R(H_\xi)}=\frac\al4\mathbf{I}-R(\widetilde {{G}}_\xi)\overline{R(H)}R(\widetilde {{G}}_\xi)^{-1}$, and this operator generates the semigroup $({\mathbb{V}}(t))_{t\ge0}$.

  These properties of ${\mathbb{B}}$ in fact imply (using the isometry $I_{\al,\xi}$ (\ref{istry})) that the operator ${\mathbbm{Q}}$ is closable in $\ell^2(\Sb,\mathsf{M}_{\al,\xi})$, and its closure generates the semigroup $({\mathbb{P}}(t))_{t\ge0}$ of the Markov process $\boldsymbol\la_{\al,\xi}$ on strict partitions.
\end{rmk}

\begin{rmk}\label{rmk:all_is_sl2}
  Apart from the situations described in Remarks \ref{rmk:D_sl2} and \ref{rmk:Qw_closable}, the matrix $H_\xi\in\mathfrak{sl}(2,\C)$ appears also in connection with the birth and death process $({\boldsymbol{n}}_{\al,\xi}(t))_{t\ge0}$ defined in \S\ref{subsection:birth_and_death}. 
  
  Namely, take the pre-generator of the process ${\boldsymbol{n}}_{\al,\xi}$ which acts in $\ell^2(\Z_{\ge0},\pi_{\al,\xi})$ (with jump rates given by (\ref{birth_death_transition_rates})). Then use the isometry between $\ell^2(\Z_{\ge0},\pi_{\al,\xi})$ and $\ell^2(\Z_{\ge0})$ which is similar to (\ref{istry}) to rewrite this pre-generator in the latter Hilbert space. We arrive at the following operator acting on $f\in\ell^2(\Z_{\ge0})$:
  \begin{align}&
    \nonumber
    f(x)\mapsto
    (1-\xi)^{-1}
    \sqrt{\xi(x+1)(x+\tfrac\al2)}f(x+1)
    +
    (1-\xi)^{-1}
    \sqrt{\xi x(x-1+\tfrac\al2)}f(x-1)\\
    &
    \qquad\qquad\qquad\qquad-
    (1-\xi)^{-1}\{\xi(x+\tfrac\al2)+x\}f(x).
    \label{Meixner_dyn_generator}
  \end{align}
  Now consider the lowest weight representation $T$ of the Lie algebra $\mathfrak{sl}(2,\C)$ in $\ell^2(\Z_{\ge0})$ under which the matrices $U,D$, and $H$ (see (\ref{slf_representation})) act in the standard basis $\{\un k\}_{k\in\Z_{\ge0}}$ as $T(U)\un k=\sqrt{(k+1)(k+\tfrac\al2)}\cdot \un{k+1}$, $T(D)\un k=\sqrt{k(k-1+\tfrac\al2)}\cdot \un{k-1}$, and $T(H)\un k=(2k+\tfrac\al2)\un k$. Then the above operator (\ref{Meixner_dyn_generator}) in $\ell^2(\Z_{\ge0})$ is just $\tfrac\al4\mathbf{I}-T(H_\xi)$, i.e., the same as in Corollary \ref{corollary:Qr_sl2}, but under a different representation of $\mathfrak{sl}(2,\C)$.
\end{rmk}



\section{Dynamical correlation functions} 
\label{sec:dynamical_correlation_functions}

In this section we prove a Pfaffian formula for the dynamical correlation functions $\rho^{(n)}_{\al,\xi}$ (\ref{df:dynamical_correlation_functions}) of the Markov processes $\boldsymbol\la_{\al,\xi}$ on strict partitions, thus  finishing the proof of Theorem \ref{thm:B} from \S\ref{sec:model_and_results}.

\subsection{Dynamical correlation functions and Markov semigroups}
\label{subsection:Dynamical_correlation_functions_and_Markov_semigroups}

Let us fix $n\ge1$ and pairwise distinct space-time points $(t_1,x_1),\dots, (t_n,x_n)\in\R_{\ge0}\times\N$. We assume that the time moments are ordered as $0\le t_1\le \dots\le t_n$. Recall the operators $\Delta_x$ (where $x\in\N$) in the Hilbert space $\ell^2(\Sb,\mathsf{M}_{\al,\xi})$ defined in \S\ref{subsection:static_Pfaffian_formula}.

\begin{lemma}\label{lemma:dynamical_correlation_weighted}
  The dynamical correlation functions of $\boldsymbol\la_{\al,\xi}$ have the form
  \begin{equation*}
    \rho^{(n)}_{\al,\xi}(t_1,x_1;\dots;t_n,x_n)=
    \big(   
    \Delta_{x_1}{\mathbb{P}}(t_2-t_1)
    \Delta_{x_2}\dots\Delta_{x_{n-1}}
    {\mathbb{P}}(t_n-t_{n-1})
    \Delta_{x_n}\mathbf1,\mathbf1\big)_{\mathsf{M}_{\al,\xi}},
  \end{equation*}
  where $({\mathbb{P}}(t))_{t\ge0}$ is the semigroup of the process $\boldsymbol\la_{\al,\xi}$ in the space $\ell^2(\Sb,\mathsf{M}_{\al,\xi})$ and $\mathbf1\in\ell^2(\Sb,\mathsf{M}_{\al,\xi})$ is the constant identity function.
\end{lemma}
\begin{proof}
  This is a simple consequence of the Markov property of the process $\boldsymbol\la_{\al,\xi}$. Indeed, let us assume (for simplicity) that $t_j$'s are distinct. The $n$-dimensional distribution of the process $\boldsymbol\la_{\al,\xi}$ at time moments $t_1<\dots<t_n$ is a probability measure on $\Sb\times\dots\times\Sb$ ($n$ copies) which assigns the probability
  \begin{equation}\label{n_dimensional_distribution}
    \mathsf{M}_{\al,\xi}(\la^{(1)})
    {\mathbb{P}}_{\la^{(1)},\la^{(2)}}(t_2-t_1)\dots
    {\mathbb{P}}_{\la^{(n-1)},\la^{(n)}}(t_n-t_{n-1})
  \end{equation}
  to every point $(\la^{(1)},\dots,\la^{(n)})$, $\la^{(i)}\in\Sb$. By definition, $\rho^{(n)}_{\al,\xi}(t_1,x_1;\dots;t_n,x_n)$ is exactly the mass of the set $\{ \la^{(1)}\ni x_1,\dots,\la^{(n)}\ni x_n \}$ under the measure (\ref{n_dimensional_distribution}). This proves the claim for distinct $t_j$'s. It can be readily verified that the claim also holds if some of $t_j$'s coincide. This concludes the proof.
\end{proof}

Let us consider the following operator in $\mathop{\mathsf{Fock}}(\N)$:
\begin{equation*}
  \Delta_{\llbracket T,X\rrbracket}:=
  \Delta_{x_1}{\mathbb{V}}(t_2-t_1)\Delta_{x_2}\dots
  \Delta_{x_{n-1}}{\mathbb{V}}(t_n-t_{n-1})\Delta_{x_n}.
\end{equation*}
Here $({\mathbb{V}}(t))_{t\ge0}$ is the semigroup in $\ell^2(\Sb)$ defined in \S\ref{subsection:pre_generator}, and we have identified $\ell^2(\Sb)$ with $\mathop{\mathsf{Fock}}(\N)$ as in \S\ref{subsection:Fock_space}. The operators $\Delta_{x}$, $x\in\N$, are now acting in $\mathop{\mathsf{Fock}}(\N)$.

\begin{prop}\label{prop:dynamical_correlation_Delta_TX}
  The correlation functions of $\boldsymbol\la_{\al,\xi}$ have the form
  \begin{equation}\label{dynamical_correlation_functions_Delta_TX}
    \rho^{(n)}_{\al,\xi}(t_1,x_1;\dots;t_n,x_n)=
    \left( R(\widetilde {{G}}_\xi)^{-1}\Delta_{\llbracket T,X\rrbracket} 
    R(\widetilde {{G}}_\xi)\mathsf{vac},\mathsf{vac} \right).
  \end{equation}
  Note that now the expectation is taken in $\mathop{\mathsf{Fock}}(\N)$.
\end{prop} 

\begin{proof}
  Since ${\mathbb{V}}(t)=I_{\al,\xi} {\mathbb{P}}(t)I_{\al,\xi}^{-1}$,  the claim is a direct consequence of Lemma \ref{lemma:dynamical_correlation_weighted} and formula (\ref{expectation_formula_fock_hard}) with $A=\Delta_{x_1}{\mathbb{P}}(t_2-t_1)\Delta_{x_2}\dots\Delta_{x_{n-1}}{\mathbb{P}}(t_n-t_{n-1})\Delta_{x_n}$.
\end{proof}

Note that in contrast to the static case (\ref{rho_n_RDeltaR}), the operator $\Delta_{\llbracket T,X\rrbracket}$ is not diagonal  (see also Remark \ref{rmk:expectation_formula_simple}). It is worth noting that formula (\ref{dynamical_correlation_functions_Delta_TX}) does not hold if  $t_j$'s are not ordered as $t_1\le\dots\le t_n$.

\subsection{Pre-generator and Kerov's operators}
\label{subsection:pre-generator_and_KO}

Our next aim is to extend the definition of the semigroup $({\mathbb{V}}(t))_{t\ge0}$ from real nonnegative values of $t$ to complex values of $t$ with $\Re t\ge0$. This will be needed in the next subsection for computation of the dynamical correlation functions. 

Observe that the matrix $iH_\xi$ (where $H_\xi$ is defined in Remark \ref{rmk:all_is_sl2} and here and below $i=\sqrt{-1}$) belongs to the real form $\mathfrak{su}(1,1)\subset\mathfrak{sl}(2,\C)$. Denote
\begin{equation*}
  W_\xi(\tau):=
  e^{-i\tau H_\xi}=
  G_\xi
  \left[
  \begin{array}{cc}
    e^{-i\tau /2}&0\\
    0&e^{i\tau/2}
  \end{array}
  \right] G_\xi^{-1}\in SU(1,1),\qquad \tau\in\R.
\end{equation*}
The family $\left\{ W_\xi(\tau) \right\}_{\tau\in\R}$ for any fixed $\xi\in[0,1)$ is a continuous curve in $SU(1,1)$ passing through the unity at $\tau=0$. By $\{ \widetilde W_\xi(\tau)\}_{\tau\in\R}$ denote the lifting of this curve to the universal covering group $SU(1,1)^\sim$.\footnote{Compare this with the definition of $\widetilde {{G}}_\xi$ in \S\ref{subsection:expectation_formula}.}

For real $\tau$ one can consider unitary operators 
\begin{equation*}
  R\big(\widetilde W_\xi(\tau)\big)=
  R(\widetilde {{G}}_\xi)R\big(\widetilde W_0(\tau)\big)R(\widetilde {{G}}_\xi)^{-1}
\end{equation*}
in the Fock space $\mathop{\mathsf{Fock}}(\N)$. Here the operator $R(\widetilde W_0(\tau))$ (corresponding to $\xi=0$)  acts on $\mathop{\mathsf{Fock}_{\mathrm{fin}}}(\N)$ as 
\begin{equation}\label{action_W_0}
  R(\widetilde W_0(\tau))\un\la=
  e^{-i\tau R(H)/2}\un\la=e^{-i\tau(|\la|+\frac\al4)}
  \un\la,
  \qquad\la\in\Sb,\qquad 
  \tau\in\R.
\end{equation}

Informally speaking, for $s\in\R_{\ge0}$, the operator ${\mathbb{V}}(s)$ means $e^{s{\mathbb{B}}}$, and for $\tau\in\R$, the operator $R(\widetilde W_\xi(\tau))e^{i\tau \frac\al4\mathbf I}$ means $e^{i \tau{\mathbb{B}}}$ (here ${\mathbb{B}}$ is the generator of the semigroup $({\mathbb{V}}(s))_{s\ge0}$, see \S\ref{subsection:pre_generator}). Thus, it is natural to give the following definition:
\begin{df}\label{df:sg}
  For $t=s+i\tau\in\C_+:=\left\{ w\in\C\colon\Re w\ge0 \right\}$ let ${\mathbb{V}}(t)$ be the following operator in $\mathop{\mathsf{Fock}}(\N)$:
  \begin{equation*}
    {\mathbb{V}}(t):={\mathbb{V}}(s)R(\widetilde 
    W_\xi(\tau))e^{i\tau\frac\al4\mathbf I}
  \end{equation*}
\end{df}
For real nonnegative $t$ the operator ${\mathbb{V}}(t)$ is self-adjoint and bounded, it was defined in \S\ref{subsection:pre_generator}. For purely imaginary $t$, the operator ${\mathbb{V}}(t)$ is unitary. Thus, the operators ${\mathbb{V}}(t)$ are bounded for all $t\in\C_+$. Moreover, ${\mathbb{V}}(t_1+t_2)={\mathbb{V}}(t_1){\mathbb{V}}(t_2)$ for any $t_1,t_2\in\C_+$, so $\left\{ {\mathbb{V}}(t) \right\}_{t\in\C_+}$ is a semigroup (with complex parameter) that can be viewed as an analytic continuation of the semigroup $\left\{ {\mathbb{V}}(s) \right\}_{s\in\R_{\ge0}}$. In particular, the operators ${\mathbb{V}}(t)$ commute with each other. Moreover, it is clear that the function $t\mapsto {\mathbb{V}}(t)h$ is bounded and continuous in $\C_+$ and holomorphic in the interior $\left\{ w\in\C_+\colon\Re w>0 \right\}$ of $\C_+$ for any vector $h\in\mathop{\mathsf{Fock}}(\N)$ which is analytic for the operator $\overline{\mathbb{B}}$.

\subsection{Pfaffian formula for dynamical correlation functions}
\label{subsection:dynamical_Pfaffian_formula}

\begin{thm}\label{thm:dynamical_Pfaffian}
  The dynamical correlation functions of the equilibrium Markov process $(\boldsymbol\la_{\al,\xi}(t))_{t\ge0}$ have the form
  \begin{equation}\label{dyn_corr_Pfaff_formula}
    \rho^{(n)}_{\al,\xi}(t_1,x_1;\dots;t_n,x_n)=
    \Pf(\boldsymbol\Phi_{\al,\xi}\llbracket T,X\rrbracket),
  \end{equation}
  where the function $\boldsymbol\Phi_{\al,\xi}(s,x;t,y)$ ($x,y\in\Z$, $s\le t$) is given by
  \begin{equation}\label{Pfk_dyn}
    \boldsymbol\Phi_{\al,\xi}(s,x;t,y):=(-1)^{x\wedge0+y\vee0}
    \sum_{m=0}^{\infty}
    2^{-\delta(m)}
    e^{-m(t-s)}
    \boldsymbol\varphi_m(x;\al,\xi)
    \boldsymbol\varphi_{m}(-y;\al,\xi)
  \end{equation}
  (see (\ref{phw_al,xi}) for definition of $\boldsymbol\varphi_m$).
  In (\ref{dyn_corr_Pfaff_formula}), $(t_1,x_1),\dots,(t_n,x_n)\in\R_{\ge0}\times\N$ are pairwise distinct space-time points such that $0\le t_1\le\dots\le t_n$, and $\boldsymbol\Phi_{\al,\xi}\llbracket T,X\rrbracket$ is the $2n\times 2n$ skew-symmetric matrix with rows and columns indexed by the numbers $1,-1,\dots,n,-n$, such that the $kj$-th entry in $\boldsymbol\Phi_{\al,\xi}\llbracket T,X\rrbracket$ above the main diagonal is $\boldsymbol\Phi_{\al,\xi}(t_{|k|},x_k;t_{|j|},x_j)$, where $k,j=1,-1,\dots,n,-n$ (thus, $|k|\le|j|$).\footnote{Here and below we use convention (\ref{x_-k_convention}).}
\end{thm}

The rest of this subsection is devoted to proving Theorem \ref{thm:dynamical_Pfaffian}.
\begin{lemma}[\cite{Olshanski-fockone}]\label{lemma:holomorphic_functions}
  Let $F(z)$ be a function on the right half-plane $\C_+$ which is bounded and continuous in $\C_+$ and is holomorphic in the interior of $\C_+$. Then $F$ is uniquely determined by its values on the imaginary axis $\left\{ w\in\C\colon \Re w=0 \right\}$.
\end{lemma} 
\begin{proof}
  Conformally transforming $\C_+$ to the unit disc $|\zeta|<1$, we get a function $G$ on the disc which is holomorphic in the interior of the disc and bounded and continuous up to the boundary (with possible exception of one point corresponding to $w=\infty\in\C_+$).

  For any fixed $\zeta_0$ with $|\zeta_0|<1$, the value $G(\zeta_0)$ is represented by Cauchy's integral over the circle $|\zeta|=r$, for $|\zeta_0|<r<1$. By our hypotheses, this Cauchy's integral has a limit as $r\to1$, which gives an expression of $G(\zeta_0)$ through the boundary values.
\end{proof} 

Let us fix pairwise distinct space-time points $(t_1,x_1),\dots, (t_n,x_n)\in\R_{\ge0}\times\N$ such that $0\le t_1\le \dots\le t_n$. For convenience, set $t_{kj}:=t_k-t_j$. Above we have expressed the dynamical correlation functions as (\ref{dynamical_correlation_functions_Delta_TX}), that is,
\begin{equation}\label{dynamical_correlation_functions_fock}
  \begin{array}{l}
    \displaystyle
    \rho^{(n)}_{\al,\xi}(t_1,x_1;\dots;t_n,x_n)
    \\\displaystyle\quad
    =
    \left( R(\widetilde {{G}}_\xi)^{-1}
    \Delta_{x_1}{\mathbb{V}}(t_{2,1})\Delta_{x_2}\dots
    \Delta_{x_{n-1}}
    {\mathbb{V}}(t_{n,n-1})\Delta_{x_n}
    R(\widetilde {{G}}_\xi)\mathsf{vac},\mathsf{vac} \right).
  \end{array}
\end{equation}
Denote the right-hand side of (\ref{dynamical_correlation_functions_fock}) by $\mathcal{F}(t_{2,1},\dots,t_{n,n-1};x_1,\dots,x_n)$. As a function in $n-1$ variables $t_{2,1},\dots,t_{n,n-1}$, $\mathcal{F}$ is initially defined for $t_{j,j-1}$ taking real nonnegative values. However, as explained in \S\ref{subsection:pre-generator_and_KO}, the definition of each operator ${\mathbb{V}}(t_{j,j-1})$ can be extended to $t_{j,j-1}\in\C_+$, so $\mathcal{F}$ is defined on $(\C_+)^{n-1}\subset \C^{n-1}$. Moreover, $\mathcal{F}$ is continuous and bounded in $(t_{2,1},\dots,t_{n,n-1})$ belonging to the closed domain $(\C_+)^{n-1}$ and is holomorphic in the interior of this domain. Therefore, by Lemma \ref{lemma:holomorphic_functions}, $\mathcal{F}(t_{2,1},\dots,t_{n,n-1};x_1,\dots,x_n)$ is uniquely determined by its values when all the variables $t_{j,j-1}$ are purely imaginary.

From now on in the computation we will assume that the variables $t_j=i\tau_j$ (where $\tau_j\in\R$, $j=1,\dots,n$) are purely imaginary. This implies that the differences $t_{k,j}=i(\tau_{k}-\tau_j)$ are also purely imaginary. For such $t_j$, each operator ${\mathbb{V}}(t_{j,j-1})$ is unitary, and, moreover,
\begin{equation}\label{invertibility}
  {\mathbb{V}}(t_{j,j-1})
  ={\mathbb{V}}(t_{j-1,1})^{-1}{\mathbb{V}}(t_{j,1}),\qquad j=1,\dots,n
\end{equation}
(here by agreement $t_{1,1}=0$, and ${\mathbb{V}}(0)=\mathbf{I}$, the identity operator).

For purely imaginary $t_j$, we want to rewrite $\mathcal{F}(t_{2,1},\dots,t_{n,n-1};x_1,\dots,x_n)$ as a certain Pfaffian. First, we need a notation. Recall that in \S\ref{subsection:static_Pfaffian_formula} we have defined a representation $\check{S}$ of $SU(1,1)$ in the Hilbert space $V=\ell^2(\Z)$ with the standard orthonormal basis $\left\{ v_x \right\}_{x\in\Z}$. Denote
\begin{equation}\label{vxi^t}
  v_{x,\xi}^{(t)}:=\check{S}( W_0(\tau) )\check{S}(G_\xi)^{-1}v_x\in V,\qquad
  x\in\Z,\qquad t=i\tau\in i\R.
\end{equation}
For $t=0$ this vector becomes $v_{x,\xi}$ defined by (\ref{vxi}).

\begin{lemma}\label{lemma:imaginary_Pfaffian}
  For $t_j=i\tau_j\in i\R$ and distinct $x_j\in\N$ ($j=1,\dots,n$) we have
  \begin{equation}
    \label{dynamical_Pfaffian_formula_purely_imaginary}
    \mathcal{F}(t_{2,1},\dots,t_{n,n-1};x_1,\dots,x_n)=
    \Pf( \mathcal{F}\llbracket T,X\rrbracket),
  \end{equation}
  where $\mathcal{F}\llbracket T,X\rrbracket$ is the $2n\times 2n$ skew-symmetric matrix with rows and columns indexed by the numbers $1,-1,\dots,n,-n$, such that the $kj$-th entry in $\mathcal{F}\llbracket T,X\rrbracket$ above the main diagonal is ${\mathbf{F}}_\mathsf{vac}\big(v_{x_k,\xi}^{(t_{|k|,1})}, v_{x_j,\xi}^{(t_{|j|,1})} \big)$, where $k,j=1,-1,\dots,n,-n$ (and thus $|k|\le|j|$). Here ${\mathbf{F}}_\mathsf{vac}$ is the vacuum average on the Clifford algebra $\Cl(V)$, see \S\ref{sec:fermionic_fock_space}.
\end{lemma}

\begin{proof}
  The operators $\Delta_x$ have the form
  \begin{equation*}
    \Delta_{x}=\mathcal{T}(v_x)\mathcal{T}(v_{-x})
    =\mathcal{T}({v_xv_{-x}}),\qquad x\in\N.
  \end{equation*}
  By (\ref{vxi^t}) and Proposition \ref{prop:group_level_identity}, we have
  \begin{equation*}
    R( \widetilde W_0(\tau) )
    R(\widetilde {{G}}_\xi)^{-1}
    \Delta_x
    R(\widetilde {{G}}_\xi)
    R( \widetilde W_0(\tau) )^{-1}=
    \mathcal{T}(v_{x,\xi}^{(t)}
    v_{-x,\xi}^{(t)}),\quad x\in\Z,\ t=i\tau\in i\R.
  \end{equation*}
  A straightforward computation using Definition \ref{df:sg} and (\ref{invertibility}) allows us to rewrite the operator in the right-hand side of (\ref{dynamical_correlation_functions_fock}) as
  \begin{align*}
      R(\widetilde {{G}}_\xi)^{-1}
      \Delta_{x_1}
      &
      {\mathbb{V}}(t_{2,1})\Delta_{x_2}\dots
      \Delta_{x_{n-1}}
      {\mathbb{V}}(t_{n,n-1})\Delta_{x_n}
      R(\widetilde {{G}}_\xi)
      \\&
      =
      \mathcal{T}(v_{x_1,\xi}^{(t_{1,1})}
      v_{-x_1,\xi}^{(t_{1,1})}
      \dots
      v_{x_n,\xi}^{(t_{n,1})}
      v_{-x_n,\xi}^{(t_{n,1})})
      R(\widetilde W_0(\tau_{n,1}))
      e^{i\tau_{n,1}\frac\al4\mathbf I}.
  \end{align*}
  Observe that from (\ref{action_W_0}) it follows that $R(\widetilde W_0(\tau_{n,1})) e^{i\tau_{n,1}\frac\al4\mathbf I}\mathsf{vac}=\mathsf{vac}$, so
  \begin{align*}
    \mathcal{F}(t_{2,1},\dots,t_{n,n-1};x_1,\dots,x_n)
    =
    \big( 
    \mathcal{T}(v_{x_1,\xi}^{(t_{1,1})}&
    v_{-x_1,\xi}^{(t_{1,1})}
    \dots
    v_{x_n,\xi}^{(t_{n,1})}
    v_{-x_n,\xi}^{(t_{n,1})})
    \mathsf{vac},\mathsf{vac} \big)
    \qquad
    \\&=
    {\mathbf{F}}_\mathsf{vac}
    \left( 
    v_{x_1,\xi}^{(t_{1,1})}
    v_{-x_1,\xi}^{(t_{1,1})}\dots
    v_{x_n,\xi}^{(t_{n,1})}
    v_{-x_n,\xi}^{(t_{n,1})}
    \right).
  \end{align*}
  An application of Wick's Theorem \ref{thm:Wick} concludes the proof.
\end{proof}

Now that we have established a Pfaffian formula for purely imaginary time variables $t_{j,j-1}$ ($j=2,\dots,n$), we want to extend it to the case when all $t_{j,j-1}$'s are real nonnegative. Let us look closer at the function ${\mathbf{F}}_\mathsf{vac}(v_{x,\xi}^{(s)}v_{y,\xi}^{(t)})$, where $s=i\si$ and $t=i\tau$ are purely imaginary. We have
\begin{equation*}
  v_{x,\xi}^{(s)}=
  \check{S}(W_0(\si))
  v_{x,\xi},\qquad 
  v_{y,\xi}^{(t)}=
  \check{S}(W_0(\tau))
  v_{y,\xi},
\end{equation*}
where $v_{x,\xi}$ and $v_{y,\xi}$ are defined by (\ref{vxi}). Therefore, we get
\begin{equation*}
  {\mathbf{F}}_\mathsf{vac}(v_{x,\xi}^{(s)}v_{y,\xi}^{(t)})=
  \sum_{k,l\in\Z}
  (v_{x,\xi},v_k)_{\ell^2(\Z)}(v_{y,\xi},v_l)_{\ell^2(\Z)}
  {\mathbf{F}}_\mathsf{vac}\left(
  \big(\check{S}(W_0(\si))
  v_k\big)\big(
  \check{S}(W_0(\tau))
  v_l\big)\right).
\end{equation*}
On the space $V_{\mathrm{fin}}\subset V=\ell^2(\Z)$ consisting of finite linear combinations of the basis vectors $\left\{ v_x \right\}_{x\in\Z}$, the operator $\check{S}(W_0(u))$ acts as $e^{-i u\check{S}(H)/2}$ (where $u\in\R$). From this fact and (\ref{Fvac_delta}), we see that
\begin{align}
  {\mathbf{F}}_\mathsf{vac}(v_{x,\xi}^{(s)}v_{y,\xi}^{(t)})&=
  \sum\nolimits_{m=0}^{\infty}e^{-m(t-s)}
  (v_{x,\xi},v_{-m})_{\ell^2(\Z)}
  (v_{y,\xi},v_m)_{\ell^2(\Z)}
  \nonumber
  \\&=
  (-1)^{x\wedge0+y\vee0}
  \sum\nolimits_{m=0}^{\infty}
  2^{-\delta(m)}
  e^{-m(t-s)}
  \boldsymbol\varphi_m(x)
  \boldsymbol\varphi_{m}(-y).
  \label{clf_dyn_single_sum}
\end{align}
Note that here $s$ and $t$ are still purely imaginary. However, one can view the right-hand side of (\ref{clf_dyn_single_sum}) as a function in $(t-s)\in\C_+$. This function is bounded and continuous in $\C_+$ and is holomorphic in the interior of $\C_+$, because the functions $\boldsymbol\varphi_m(x)$ for fixed $x$ and $m\to+\infty$ decay as $\mathrm{Const}\cdot m^{-x-\frac12}\xi^{\frac m2}$ (this can be readily observed from the analytic expression (\ref{phw_al,xi})). We are interested in the restriction of the right-hand side of (\ref{clf_dyn_single_sum}) to real nonnegative values of $(t-s)$. Observe that this is exactly the kernel $\boldsymbol\Phi_{\al,\xi}(s,x;t,y)$ (\ref{Pfk_dyn}). By application of Lemma \ref{lemma:holomorphic_functions}, we see that formula (\ref{dynamical_Pfaffian_formula_purely_imaginary}) holds for real nonnegative $t_{2,1},\dots,t_{n,n-1}$, that is, for $0\le t_1\le \dots\le t_{n}$. This fact together with Proposition \ref{prop:dynamical_correlation_Delta_TX} implies Theorem \ref{thm:dynamical_Pfaffian}. 

Thus, we have finished the proof of Theorem \ref{thm:B} from \S\ref{sec:model_and_results}.

\subsection{Skew-symmetric matrices in Pfaffian formulas}
\label{subsection:skew_symm}

In the right-hand sides of our Pfaffian formulas (\ref{static_Pfaffian_formula_thm_formula}) and (\ref{dyn_corr_Pfaff_formula}) for static and dynamical correlation functions we see certain skew-symmetric $2n\times 2n$ matrices constructed using Pfaffian kernels $\boldsymbol\Phi_{\al,\xi}(x,y)$ and $\boldsymbol\Phi_{\al,\xi}(s,x;t,y)$, respectively. It is clear from (\ref{Pfk_static_xy_all}) and (\ref{Pfk_dyn}) that for $t=s$, the dynamical kernel $\boldsymbol\Phi_{\al,\xi}(s,x;t,y)$ turns into the static one. (In fact, this is the reason why we use the same notation for these kernels.) However, the matrix of (\ref{dyn_corr_Pfaff_formula}) for $t_1=\dots=t_n$ does not become the one from (\ref{static_Pfaffian_formula_thm_formula}). Let us explain how one can transform (\ref{dyn_corr_Pfaff_formula}) to get the expected behavior in the static picture.

One can readily verify that conjugating the matrix $\boldsymbol\Phi_{\al,\xi}\llbracket T,X\rrbracket$ from (\ref{dyn_corr_Pfaff_formula}) by the matrix $C$ of the permutation $(1,2n,2,2n-1,\dots,n,n+1)$, we get the following $2n\times 2n$ skew-symmetric matrix:
\begin{align}\label{Pfk_matr_conjug}
  &
  (C\boldsymbol\Phi_{\al,\xi}\llbracket T,X\rrbracket C^T)_{i,j}\\&
  \qquad
  =
  \left\{
  \begin{array}{ll}
    \boldsymbol\Phi_{\al,\xi}(t_i,x_i;t_j,x_j),
    &\mbox{if $1\le i<j\le n$};\\
    \boldsymbol\Phi_{\al,\xi}(t_i,x_i;t_{j'},-x_{j'}),
    &\mbox{if $1\le i\le n<j\le 2n$ and 
    $i\le j'$};\\
    -\boldsymbol\Phi_{\al,\xi}(t_{j'},-x_{j'};t_i,x_i),
    &\mbox{if $1\le i\le n<j\le 2n$
    and $i>j'$};\\
    -\boldsymbol\Phi_{\al,\xi}(t_{j'},-x_{j'};t_{i'},-x_{i'}),
    &\mbox{if $n< i<j\le 2n$},        
  \end{array}
  \right.
  \nonumber
\end{align}
where $1\le i<j\le 2n$, and $i':=2n+1-i$, $j':=2n+1-j$. The permutation matrix $C$ has determinant one (cf. how (\ref{hX_in_order}) is obtained), so the Pfaffian does not change under such a conjugation. It is worth noting that matrices similar to (\ref{Pfk_matr_conjug}) appeared in \cite[Thm. 3.1]{Matsumoto2005} and \cite[Thm. 2.2]{Vuletic2007shifted}). Using Corollary \ref{corollary:reduction_formulas}, it is clear that for $t_1=\dots=t_n$, (\ref{Pfk_matr_conjug}) becomes the matrix $\hat{\boldsymbol\Phi}_{\al,\xi}\llbracket X\rrbracket$. Observe that Corollary \ref{corollary:reduction_formulas}.(2) does not hold in the dynamical case, so one cannot put a matrix of the form $\hat{\boldsymbol\Phi}_{\al,\xi}\llbracket T,X\rrbracket$ in the right-hand side of (\ref{dyn_corr_Pfaff_formula}).



\section{Dynamical Pfaffian kernel} 
\label{sec:dynamical_pfaffian_kernel}

In this section we discuss some properties of the extended Pfaffian hyper\-geo\-met\-ric-type kernel $\boldsymbol\Phi_{\al,\xi} (s,x;t,y)$ (obtained in the previous section) of our Markov processes $(\boldsymbol\la_{\al,\xi}(t))_{t\ge0}$ on strict partitions.

\subsection{Expression through the extended discrete hypergeometric kernel}\label{subsection:Pfk_dyn_Kzz}

The extended discrete hypergeometric kernel introduced in \cite{Borodin2006} serves as a determinantal kernel for a Markov dynamics preserving the $z$-measures on ordinary partitions. It is given by \cite[Thm.~A (Part~2)]{Borodin2006}:
\begin{equation*}
  \underline K_{z,z',\xi}
  (t,\sh x;s,\sh y)=\pm\sum\nolimits_{\sh a\in\Z'_+}
  e^{-\sh a|t-s|}
  \psi_{\pm \sh a}(\sh x;z,z',\xi)
  \psi_{\pm \sh a}(\sh y;z,z',\xi),
\end{equation*}
where $\sh x,\sh y\in\Z'=\Z+\frac12$, the ``$+$'' sign is taken for $t\ge s$, and the ``$-$'' sign is taken for $t<s$. Our kernel $\boldsymbol\Phi_{\al,\xi}(s,x;t,y)$ can be expressed through $\underline K_{z,z',\xi} (t,\sh x;s,\sh y)$:

\begin{prop}\label{prop:Pfkd_Kzz}
  For all $x,y\in\Z$ and $s\le t$, we have
  \begin{align}\nonumber
    \boldsymbol\Phi_{\al,\xi}
    (s,x;t,y)&=
    \tfrac12
    (-1)^{x\wedge0+y\vee0}
    \Big[
    e^{-\frac12(t-s)}
    \underline K_{\nu(\al)-\frac12,
    -\nu(\al)-\frac12,\xi}
    (t,x+\tfrac12;s,-y+\tfrac12)\\&\qquad+
    e^{\frac12(t-s)}
    \underline K_{\nu(\al)+\frac12,
    -\nu(\al)+\frac12,\xi}
    (t,x-\tfrac12;s,-y-\tfrac12)
    \Big],
    \label{Pfkd_Kzz_text}
  \end{align}
\end{prop}

\begin{proof}
  This is established similarly to (\ref{Pfk_Kzz}): using the above formula for the kernel $\underline K_{z,z',\xi} (t,\sh x;s,\sh y)$ and (\ref{phw_psi}), one can write two identities analogous to (\ref{Kzz_sum_1}) and (\ref{Kzz_sum_2}), then take their half sum and use (\ref{Pfk_dyn}).
\end{proof}

When $t=s$, (\ref{Pfkd_Kzz_text}) reduces to the static version (\ref{Pfk_Kzz}). Observe that the parameters $z,z'$ (depending on $\al$) in (\ref{Pfkd_Kzz_text}) are admissible (see p. \pageref{param_zz'_assumptions}). However, similarly to the static identity, (\ref{Pfkd_Kzz_text}) seems to imply no probabilistic connections between the dynamics related to the $z$-measures \cite{Borodin2006} and our Markov processes $\boldsymbol\la_{\al,\xi}$.

\subsection{Plancherel degeneration}
\label{subsection:Plancherel_degeneration_dyn}

\begin{thm}
  \label{thm:dynamical_Plancherel}
  Under the Plancherel degeneration (\ref{Pl_degen}), the extended hyper\-geo\-metric-type kernel $\boldsymbol\Phi_{\al,\xi}$ has a pointwise limit. The limiting kernel is expressed through the Bessel function (\ref{BesselJ}):
  \begin{align*}
    \boldsymbol\Phi_{\te}(s,x;t,y)=
    (-1)^{x\wedge0+y\vee0}
    \sum\nolimits_{m=0}^{\infty}
    2^{-\delta(m)}
    e^{-m(t-s)}J_{m+x}(2\sqrt\te)J_{m-y}(2\sqrt\te)
  \end{align*} 
  (here $x,y\in\Z$ and $0\le s\le t$).
\end{thm}

\begin{proof}
  This is a consequence of (\ref{Pfk_dyn}) and the fact that under the Plancherel degeneration (\ref{Pl_degen}) one has $\boldsymbol\varphi_m(x;\al,\xi)\to J_{x+m}$.
\end{proof}

In the static case ($t=s$), the above theorem reduces to a Pfaffian formula for correlation functions of the poissonized Plancherel measure with the kernel $\boldsymbol\Phi_{\te}(x,y)$. This Pfaffian formula then can be written as a determinantal one, see Theorem \ref{thm:Plancherel_static}.

\subsection{``Whittaker'' limit}
\label{subsection:whit_limit}

Here we consider the limit of our dynamical kernel $\boldsymbol\Phi_{\al,\xi}(s,x;t,y)$ which corresponds to studying the dynamics of the scaled largest rows in a strict partition. We embed the half-lattice $\N$ into the half-line $\R_{>0}$ as $x\mapsto(1-\xi)x$, where $x\in\N$, and then pass to the limit as $\xi\nearrow1$. For the static picture this limit transition was described in \cite[\S3.2]{Petrov2010}, where the corresponding limit of the hypergeometric-type kernel $\mathbf{K}_{\al,\xi}(x,y)$ was given.

Let $\mathbf{w}_m(u;\al)$ be the functions indexed by $m\in\Z$ with argument $u\in\R_{>0}$, which are expressed through the classical Whittaker functions $W_{\kappa,\mu}(u)$ (e.g., see \cite[Ch. 6.9]{Erdelyi1953}) as follows:
\begin{equation*}
  \mathbf{w}_m(u;\al):=\big(\Gamma(\tfrac12-m-\nu(\al))
  \Gamma(\tfrac12-m+\nu(\al))\big)^{-\frac12}
  u^{-\frac12}W_{-m,\nu(\al)}(u).
\end{equation*}
These functions $\mathbf{w}_m(u;\al)$ are related to the functions $w_{\sh a}(u;z,z')$ \cite[(9.1)]{Borodin2006} in the same manner as in (\ref{phw_psi}):
\begin{equation}\label{whit_w_w_identity}
  \mathbf{w}_m(u;\al)=
  w_{m+\frac12+d}(u;\nu(\al)+
  \tfrac12+d,-\nu(\al)+\tfrac12+d),
  \quad
  m\in\Z,\ u>0,
\end{equation}
where $d\in\Z$ is arbitrary. Thus, from \cite[Prop. 9.1]{Borodin2006} it follows that if $\xi\nearrow1$ and $x\in\Z$ goes to $+\infty$ so that $(1-\xi)x\to u>0$, then $\boldsymbol\varphi_m(x;\al,\xi)\sim (1-\xi)^{\frac12}\mathbf{w}_m(u;\al)$. 

\begin{thm}\label{thm:whit_lim}
  As $\xi\nearrow1$, $x,y\to \infty$ (inside $\Z$), and $(1-\xi)x\to u$, $(1-\xi)y\to v$, where $u,v\in\R_{\ne0}$, there exists a limit of the extended Pfaffian kernel:
  \begin{equation*}
    \boldsymbol\Phi_{\al}^W(s,u;t,v)=
    \lim\nolimits_{\xi\nearrow1}(1-\xi)^{-1}
    \boldsymbol\Phi_{\al,\xi}(s,x;t,y),\qquad 0\le s\le t
  \end{equation*}
  (the prefactor $(1-\xi)^{-1}$ is due to the rescaling of the space $\N$). For $s<t$, the limiting kernel $\boldsymbol\Phi_{\al}^W(s,u;t,v)$ is given for $u,v>0$ by 
  \begin{align*}
    \boldsymbol\Phi_{\al}^W(s,u;t,v)&=
    \sum\nolimits_{m=0}^\infty
    2^{-\delta(m)}e^{-m(t-s)}(-1)^m
    \mathbf{w}_m(u;\al)\mathbf{w}_{-m}(v;\al),
    \\
    \boldsymbol\Phi_{\al}^W(s,u;t,-v)&=
    \sum\nolimits_{m=0}^\infty
    2^{-\delta(m)}e^{-m(t-s)}
    \mathbf{w}_m(u;\al)\mathbf{w}_{m}(v;\al),
    \\
    \boldsymbol\Phi_{\al}^W(s,-u;t,v)&=
    \sum\nolimits_{m=0}^\infty
    2^{-\delta(m)}e^{-m(t-s)}
    \mathbf{w}_{-m}(u;\al)\mathbf{w}_{-m}(v;\al),
    \\
    \boldsymbol\Phi_{\al}^W(s,-u;t,-v)&=
    \sum\nolimits_{m=0}^\infty
    2^{-\delta(m)}e^{-m(t-s)}(-1)^m 
    \mathbf{w}_{-m}(u;\al)\mathbf{w}_{m}(v;\al).
  \end{align*}
  For $s=t$, one should take limits in the expression for the static kernel given by Proposition \ref{prop:Pfk_integrable} (the above expressions for $\boldsymbol\Phi_{\al}^W(s,u;s,v)$ and $\boldsymbol\Phi_{\al}^W(s,-u;s,-v)$ do not converge). This leads to an ``integrable'' expression for the limiting static Pfaffian kernel $\boldsymbol\Phi_{\al}^W(s,u;s,v)$. This kernel is reduced to the determinantal Macdonald kernel given in \cite[Theorem 3.2]{Petrov2010}.
\end{thm}

\begin{proof}
  This is a consequence of results of \cite[\S9]{Borodin2006}. Formulas for $\boldsymbol\Phi_{\al}^W$ are obtained using formula (\ref{Pfk_dyn}) for the pre-limit kernel $\boldsymbol\Phi_{\al,\xi}$. One should consider various cases of signs of $x,y$, and with the help of (\ref{phw_symm2}) make the arguments of the functions $\boldsymbol\varphi_m(\cdot;\al,\xi)$ positive. After that one should replace each $\boldsymbol\varphi_m(\cdot;\al,\xi)$ by the corresponding $\mathbf{w}_m(\cdot;\al)$. All limit transitions can be justified using expression (\ref{Pfkd_Kzz_text}) of our kernel through the dynamical kernel of \cite{Borodin2006}. 
\end{proof}

\begin{rmk}\label{rmk:no_whit_limit}
  It is known \cite[\S8]{Borodin2006} that the discrete hypergeometric kernel $\underline K_{z,z',\xi}(s,x;t,y)$ itself (even in the fixed time $s=t$ picture) does not converge in the limit described in the above theorem. For this kernel to have a limit one should perform a certain particle-hole involution, see \cite[\S8]{Borodin2006}. However, we see that a combination of the values of the kernel $\underline K_{z,z',\xi}$ (with suitably chosen parameters) as in (\ref{Pfkd_Kzz_text}) does have a limit in this regime, which is given in the above theorem. 
\end{rmk}

\begin{rmk}\label{rmk:knuxi_Kzz_whit_limit}
  It is interesting to compare the limit behavior of the relation (\ref{knuxi_Kzz}) in the above scaling limit as $\xi\nearrow1$ with the observation made in \cite[Rmk. 6]{Petrov2010}. Using the approach of \cite[\S8]{Borodin2006}, one can see that in the limit (\ref{knuxi_Kzz}) becomes 
  \begin{equation*}
    \sqrt{\frac uv}\mathbf{K}^{\mathrm{Mac}}_\al(u,v)
    =\mathcal{K}^{++}_{\nu(\al)+\frac12,
    -\nu(\al)+\frac12}(u,v)
    -\mathcal{K}^{+-}_{\nu(\al)-\frac12,
    -\nu(\al)-\frac12}(u,v),\qquad u,v>0,
  \end{equation*}
  where $\mathbf{K}^{\mathrm{Mac}}_\al$ is the Macdonald kernel of \cite[Thm. 3.2]{Petrov2010}, and $\mathcal{K}^{++}_{z,z'}$ and $\mathcal{K}^{+-}_{z,z'}$ are the blocks of the matrix Whittaker kernel of \cite[\S5]{Borodin2000a}.
  
  On the other hand, \cite[Remark 6]{Petrov2010} states
  \begin{equation*}
    \mathbf{K}^{\mathrm{Mac}}_\al(u,v)
    =\mathcal{K}^{++}_{\nu(\al)-\frac12,
    -\nu(\al)-\frac12}(u,v)
    -i\mathcal{K}^{+-}_{\nu(\al)-\frac12,
    -\nu(\al)-\frac12}(u,v),\qquad u,v>0.
  \end{equation*}
  The advantage of the first relation is that it involves admissible parameters $z,z'$ (see p. \pageref{param_zz'_assumptions}) in contrast to the second relation. One can show that the two above relations are equivalent by manipulations with the confluent hypergeometric function using formulas from \cite[Ch. VI]{Erdelyi1953}.
\end{rmk}

\subsection{``Gamma'' limit}  
\label{subsection:gamma_limit}

The second limit regime we consider corresponds to studying the dynamics of the smallest rows in a strict partition. We stay on the lattice $\N$ and pass to the limit as $\xi\nearrow1$. An appropriate scaling of time is needed. For the static picture the corresponding limit of the determinantal kernel $\mathbf{K}_{\al,\xi}(x,y)$ was given in \cite[\S3.1]{Petrov2010}.

\begin{thm}\label{thm:gamma_lim}
  Let $x,y\in\Z$ and $s=(1-\xi)\si$, $t=(1-\xi)\tau$, where $\si,\tau\in\R$, $\si\le\tau$. As $\xi\nearrow1$, there exists a limit of the extended Pfaffian kernel:
  \begin{equation*}
    \boldsymbol\Phi_{\al}^{\mathrm{gamma}}(\si,x;\tau,y)=
    \lim\nolimits_{\xi\nearrow1}
    \boldsymbol\Phi_{\al,\xi}
    \big((1-\xi)\si,x;(1-\xi)\tau,y\big).
  \end{equation*}
  The limiting kernel has the form
  \begin{equation*}
    \boldsymbol\Phi_{\al}^{\mathrm{gamma}}(\si,x;\tau,y)=
    (-1)^{x\wedge0+y\vee0}\int_{0}^{+\infty}
    e^{-u(\tau-\si)}\mathbf{w}_x(u;\al)\mathbf{w}_{-y}(u;\al)du.
  \end{equation*}
  For $\si=\tau$, the static Pfaffian kernel $\boldsymbol\Phi_{\al}^{\mathrm{gamma}}(\si,x;\si,y)$ admits a simpler ``integrable'' form which corresponds to the determinantal kernel given in \cite[Theorem 3.1]{Petrov2010}.
\end{thm}
\begin{proof}
  After a simple computation, this follows from the result of \cite[Thm. 10.1]{Borodin2006} together with (\ref{Pfkd_Kzz_text}) and (\ref{whit_w_w_identity}).
\end{proof}
 Observe that in contrast to the scaling limit regime of the previous subsection (see Remark \ref{rmk:no_whit_limit}), here the kernel $\underline K_{z,z',\xi}(s,x;t,y)$ itself has a limit in the regime described in the above theorem.


\begin{rmk}\label{rmk:oint}
  One could also prove Theorems \ref{thm:whit_lim} and \ref{thm:gamma_lim} by writing double contour integrals for the pre-limit kernel $\boldsymbol\Phi_{\al,\xi}(s,x;t,y)$, and passing to the corresponding limits. In this way one can obtain double contour integral expressions for the kernels $\boldsymbol\Phi_{\al}^W(s,u;t,v)$ and $\boldsymbol\Phi_{\al}^{\mathrm{gamma}}(\si,x;\tau,y)$ similar to the ones of Theorems 9.4 and 10.1 in \cite{Borodin2006}, respectively. On the other hand, these double contour integral expressions for $\boldsymbol\Phi_{\al}^W$ and $\boldsymbol\Phi_{\al}^{\mathrm{gamma}}$ readily follow from our formulas of Theorems \ref{thm:whit_lim} and \ref{thm:gamma_lim}. See also \cite[\S8--10]{Borodin2006} for more details about the ``Whittaker'' and ``gamma'' limit transitions of correlation kernels in the model of the $z$-measures.
\end{rmk}



\appendix
\section{Reduction of Pfaffians to determinants} 
\label{sec:reduction_of_pfaffians_to_determinants}

Let us first recall basic definitions and properties related to Pfaffians. We use the following notations for matrices. Let $\mathfrak{X}$ be an abstract finite space of indices and $\mathbf{a}=(a_1,\dots,a_{2n})$ be a sequence of length $2n$ of points of $\mathfrak{X}$. Let $F\colon\mathfrak{X}\times\mathfrak{X}\to\C$ be some function. Form a $2n\times 2n$ skew-symmetric matrix 
\begin{equation*}
  \left(
  \begin{array}{ccccc}
    0&F(a_1,a_2)&\dots&F(a_1,a_{2n-1})&F(a_{1},a_{2n})\\
    -F(a_1,a_2)&0&\dots&F(a_2,a_{2n-1})&F(a_2,a_{2n})\\
    \multicolumn{5}{c}\dotfill\\
    -F(a_1,a_{2n-1})&-F(a_2,a_{2n-1})&
    \dots&0&F(a_{2n-1},a_{2n})\\
    -F(a_1,a_{2n})&-F(a_2,a_{2n})&\dots&-F(a_{2n-1},a_{2n})&0
  \end{array}
  \right).
\end{equation*}
Denote this matrix by ${F\llbracket\mathbf{a}\rrbracket}$. This skew-symmetric matrix has rows and columns indexed by $a_1,\dots,a_{2n}$, such that the $ij$th element above the main diagonal is equal to $F(a_i,a_j)$ (here $1\le i<j\le 2n$).

\begin{df}
  Let $\mathbf{a}=(a_1,\dots,a_{2n})$ and ${F\llbracket\mathbf{a}\rrbracket}$ be as defined above. The determinant $\det({F\llbracket\mathbf{a}\rrbracket})$ is a perfect square as a polynomial in $F(a_i,a_j)$ (where $i<j$). The Pfaffian of ${F\llbracket\mathbf{a}\rrbracket}$, denoted by $\Pf\big({F\llbracket\mathbf{a}\rrbracket}$\big), is defined to be the square root of $\det{F\llbracket\mathbf{a}\rrbracket}$ having the ``$+$'' sign by the monomial $F(a_1,a_2)\dots F(a_{2n-1},a_{2n})$.  
\end{df}
The following properties of Pfaffians are well known:
\begin{enumerate}[$\bullet$]
  \item Let $A$ be a skew-symmetric $2n\times 2n$ matrix and $B$ be any $2n\times 2n$ matrix, then
    \begin{equation}\label{property1_det_Pf}
      \Pf(BAB^T)=\det B\cdot\Pf(A).
    \end{equation}
    where $B^T$ means the transposed matrix;
  \item If $M$ is any $n\times n$ matrix, then
    \begin{equation}\label{property2_block_Pf}
      \Pf\left(
      \begin{array}{cc}
        0&M\\
        -M^T&0
      \end{array}
      \right)=(-1)^{n(n-1)/2}\det M.
    \end{equation}
\end{enumerate}

Now we give a sufficient condition under which a $2n\times 2n$ Pfaffian can be reduced to a certain $n\times n$ determinant. Assume that the set $\mathfrak{X}$ is divided into two parts $\mathfrak{X}=\mathfrak{X}_+\sqcup\mathfrak{X}_-$, and there exists a bijection between $\mathfrak{X}_+$ and $\mathfrak{X}_-$. By $a\mapsto \hat a$ we denote the corresponding involution of the space $\mathfrak{X}$ that interchanges $\mathfrak{X}_+$ and $\mathfrak{X}_-$. Let $\mathbf{a}:=\left( a_1,\dots,a_n,\hat a_n,\dots, \hat a_1\right)$, and $a_i\in \mathfrak{X}_+$ (so $\hat a_i\in\mathfrak{X}_-$), $i=1,\dots,n$.

\begin{prop}\label{prop:A1_Determ_reduction}
  Suppose that the function $F$ on $\mathfrak{X}\times\mathfrak{X}$ satisfies the following properties:\footnote{cf. Corollary \ref{corollary:reduction_formulas}.}
  \begin{enumerate}[{\rm{}(1)\/}]
    \item $F(a,\hat b)=F(b,\hat a)$ for any $a,b\in\mathfrak{X}$.
    \item $F(a,b)=-F(b,a)$ for any $a,b\in\mathfrak{X}$ such that $a\ne \hat b$.
    \item There exists a strictly positive function $f\colon\mathfrak{X}_+\to\R$ with the property $f(a)\ne f(b)$ if $a\ne b$, such that
      \begin{equation*}
        \left( f(a)-f(b) \right)F(a,\hat b)=
        (f(a)+f(b))F(a,b)\qquad
        \mbox{for any $a,b\in\mathfrak{X}^+$}.
      \end{equation*}
  \end{enumerate}
  Then 
  \begin{equation*}
    \Pf\big({F\llbracket a_1,\dots,a_n,\hat a_n,\dots,
    \hat a_1\rrbracket}\big)=
    \det [K(a_r,a_s)]_{r,s=1}^n,
  \end{equation*}
  where $K$ has the form
  \begin{equation}\label{App1_det_kernel_1}
    K(u,v)=\frac{2F(u,\hat v)\sqrt{f(u)f(v)}}
    {f(u)+f(v)},\qquad u,v\in\mathfrak{X}_+.
  \end{equation}
\end{prop}
Note that the third property above implies that $F(a,a)=0$ for all $a\in\mathfrak{X}_+$.
\begin{proof}
  In this proof we denote the matrix ${F\llbracket a_1,\dots,a_n,\hat a_n,\dots,\hat a_1\rrbracket}$ simply by $F$.

  We act on $F$ by $SL(2,\C)^{n}$: each $j$th copy of $SL(2,\C)$ acts as $F\mapsto C_jFC_j^T$, where $C_j$ is the $2n\times 2n$ identity matrix except for the $2\times 2$ submatrix with determinant $1$ formed by rows and columns with numbers $j$ and $2n+1-j$. By (\ref{property1_det_Pf}), this action of $SL(2,\C)^n$ does not change the Pfaffian of $F$. We want to choose $C\in SL(2,\C)^n$ such that the matrix $CFC^T$ becomes a block matrix as in (\ref{property2_block_Pf}).

  Define $g(a):=\frac12\log f(a)$, $a\in\mathfrak{X}_+$. As the $j$th element in $C\in SL(2,\C)^n$ we take the hyperbolic rotation $\left(
    \begin{array}{cc}
      \cosh g(a_j) &\sinh g(a_j)\\
      \sinh g(a_j) &\cosh g(a_j)\\
    \end{array}
\right)$. The whole matrix $C$ looks as 
  \begin{equation*}
    C=\left(
    \begin{array}{ccc|ccc}
      \cosh g(a_1)&\dots&0&0&\dots&\sinh g(a_1)\\
      \multicolumn{3}{c|}\dotfill&\multicolumn{3}{c}\dotfill\\
      0&\dots&\cosh g(a_n)&\sinh g(a_n)&\dots&0\\
      \hline  
      0&\dots&\sinh g(a_n)&\cosh g(a_n)&\dots&0\\
      \multicolumn{3}{c|}\dotfill&\multicolumn{3}{c}\dotfill\\
      \sinh g(a_1)&\dots&0&0&\dots&\cosh g(a_1)
    \end{array}
    \right).
  \end{equation*}
  It can be readily verified using the properties of $F$ that
$CFC^T=\left(
    \begin{array}{cc}
      0&M\\
      -M^T&0
    \end{array}
\right)$, where the rows of $M$ are indexed by $i=1,2,\dots,n$, and columns are indexed by $j=n+1,\dots,2n$, and 
  \begin{equation*}
    M_{ij}=\left\{
    \begin{array}{ll}\displaystyle
      \frac{2F(a_i,a_{2n+1-j})
      \sqrt{f(a_i)f(a_{2n+1-j})}}{f(a_i)-f(a_{2n+1-j})},
      &\qquad\mbox{if $i+j\ne 2n$},\\\rule{0pt}{16pt}
      F(a_i,\hat a_i),&\qquad\mbox{otherwise}.
    \end{array}
    \right.
  \end{equation*}
  Set, for $r,s=1,\dots,n$,
  \begin{equation}\label{app_K_bad}
    K(a_r,a_s):=M_{r,2n+s-1},
  \end{equation}
  and note that $\det \left[ K(a_r,a_s) \right]_{r,s=1}^{n}=  (-1)^{n(n-1)/2}\det M$. Thus, from (\ref{property1_det_Pf}) and (\ref{property2_block_Pf}) we get $\Pf (F)=\Pf(CFC^T)= (-1)^{n(n-1)/2}\det M=\det \left[ K(a_r,a_s)\right]_{r,s=1}^{n}$. It remains to observe that $K(\cdot,\cdot)$ (\ref{app_K_bad}) that now has the form
  \begin{equation*}
  K(u,v)=\left\{
    \begin{array}{ll}
      \displaystyle
      \frac{2 F(u,v)
      \sqrt{f(u)f(v)}}{f(u)-f(v)},&\qquad\mbox{if $u\ne v$},
      \\
      \rule{0pt}{16pt}
      F(u,\hat u),&\qquad\mbox{otherwise},
    \end{array}
    \right.
  \end{equation*}
  where $u,v\in\mathfrak{X}_+$, can be rewritten as (\ref{App1_det_kernel_1}) using the properties of $F$.
\end{proof}


\providecommand{\bysame}{\leavevmode\hbox to3em{\hrulefill}\thinspace}
\providecommand{\MR}{\relax\ifhmode\unskip\space\fi MR }
\providecommand{\MRhref}[2]{%
  \href{http://www.ams.org/mathscinet-getitem?mr=#1}{#2}
}
\providecommand{\href}[2]{#2}


\begin{thebibliography}{DJKM82}

\bibitem[ANvM10]{adler-nordenstam2010dyson}
M.~Adler, E.~Nordenstam, and P.~van Moerbeke, \emph{{The Dyson Brownian minor
  process}}, 2010, arXiv:1006.2956 [math.PR].

\bibitem[BDJ99]{baik1999distribution}
J.~Baik, P.~Deift, and K.~Johansson, \emph{{On the distribution of the length
  of the longest increasing subsequence of random permutations}}, Journal of
  the American Mathematical Society \textbf{12} (1999), no.~4, 1119--1178,
  arXiv:math/9810105 [math.CO].

\bibitem[BDJ00]{Baik1999}
\bysame, \emph{{On the distribution of the length of the second row of a Young
  diagram under Plancherel measure}}, Geometric And Functional Analysis
  \textbf{10} (2000), no.~4, 702--731, arXiv:math/9901118 [math.CO].

\bibitem[BGR10]{borodin-gr2009q}
A.~Borodin, V.~Gorin, and E.M. Rains, \emph{{q-Distributions on boxed plane
  partitions}}, Selecta Mathematica, New Series \textbf{16} (2010), no.~4,
  731--789, arXiv:0905.0679 [math-ph].

\bibitem[BO98]{Borodin1998}
A.~Borodin and G.~Olshanski, \emph{Point processes and the infinite symmetric
  group}, Math. Res. Lett. \textbf{5} (1998), 799--816.

\bibitem[BO00]{Borodin2000a}
\bysame, \emph{Distributions on partitions, point processes, and the
  hypergeometric kernel}, Commun. Math. Phys. \textbf{211} (2000), no.~2,
  335--358, arXiv:math/9904010 [math.RT].

\bibitem[BO06a]{Borodin2006}
\bysame, \emph{Markov processes on partitions}, Probab. Theory Related Fields
  \textbf{135} (2006), no.~1, 84--152, arXiv:math-ph/0409075.

\bibitem[BO06b]{borodin2006meixner}
\bysame, \emph{{Meixner polynomials and random partitions}}, Moscow
  Mathematical Journal \textbf{6} (2006), no.~4, 629--655, arXiv:math/0609806
  [math.PR].

\bibitem[BO06c]{Borodin2006stochastic}
\bysame, \emph{{Stochastic dynamics related to Plancherel measure on
  partitions}}, Representation Theory, Dynamical Systems, and Asymptotic
  Combinatorics (V.~Kaimanovich and A.~Lodkin, eds.), 2, vol. 217, Transl. AMS,
  2006, pp.~9--22, arXiv:math--ph/0402064.

\bibitem[BO09]{Borodin2007}
\bysame, \emph{Infinite-dimensional diffusions as limits of random walks on
  partitions}, Prob. Theor. Rel. Fields \textbf{144} (2009), no.~1, 281--318,
  arXiv:0706.1034 [math.PR].

\bibitem[BOO00]{Borodin2000b}
A.~Borodin, A.~Okounkov, and G.~Olshanski, \emph{{Asymptotics of Plancherel
  measures for symmetric groups}}, J. Amer. Math. Soc. \textbf{13} (2000),
  no.~3, 481--515, arXiv:math/9905032 [math.CO].

\bibitem[Bor99]{Borodin1997}
A.~Borodin, \emph{Multiplicative central measures on the {S}chur graph}, Jour.
  Math. Sci. (New York) \textbf{96} (1999), no.~5, 3472--3477, in Russian: Zap.
  Nauchn. Sem. POMI {\bf{}240\/} (1997), 44--52, 290--291.

\bibitem[Bor00]{Borodin2000riemann}
\bysame, \emph{{Riemann-Hilbert problem and the discrete Bessel Kernel}},
  International Mathematics Research Notices \textbf{2000} (2000), no.~9,
  467--494, arXiv:math/9912093 [math.CO].

\bibitem[Bor09]{Borodin2009}
\bysame, \emph{Determinantal point processes}, 2009, arXiv:0911.1153 [math.PR].

\bibitem[BR05]{borodin2005eynard}
A.~Borodin and E.M. Rains, \emph{{Eynard--Mehta theorem, Schur process, and
  their Pfaffian analogs}}, Journal of Statistical Physics \textbf{121} (2005),
  no.~3, 291--317, arXiv:math-ph/0409059.

\bibitem[BS06]{Borodin2006averages}
A.~Borodin and E.~Strahov, \emph{{Averages of characteristic polynomials in
  random matrix theory}}, Communications on Pure and Applied Mathematics
  \textbf{59} (2006), no.~2, 161--253, arXiv:math-ph/0407065.

\bibitem[BS09]{Borodin2009correlation}
\bysame, \emph{{Correlation kernels for discrete symplectic and orthogonal
  ensembles}}, Communications in Mathematical Physics \textbf{286} (2009),
  no.~3, 933--977, arXiv:0712.1693 [math-ph].

\bibitem[Dei99]{deift1999integrable}
P.~Deift, \emph{{Integrable operators}}, Differential operators and spectral
  theory: M. Sh. Birman's 70th Anniversay Collection, Transl. AMS, 1999, p.~69.

\bibitem[DJKM82]{Date1982transformation}
E.~Date, M.~Jimbo, M.~Kashiwara, and T.~Miwa, \emph{{Transformation groups for
  soliton equations. IV. A new hierarchy of soliton equations of KP-type}},
  Physica D \textbf{4} (1982), 343--365.

\bibitem[Dys70]{Dyson1970correlations}
F.J. Dyson, \emph{{Correlations between eigenvalues of a random matrix}},
  Communications in Mathematical Physics \textbf{19} (1970), no.~3, 235--250.

\bibitem[Erd53]{Erdelyi1953}
A.~Erd{\'e}lyi (ed.), \emph{{Higher transcendental functions}}, McGraw--Hill,
  1953.

\bibitem[Fer04]{ferrari2004polynuclear}
P.L. Ferrari, \emph{{Polynuclear growth on a flat substrate and edge scaling of
  GOE eigenvalues}}, Communications in Mathematical Physics \textbf{252}
  (2004), no.~1, 77--109, arXiv:math-ph/0402053.

\bibitem[Fom94]{Fomin1994duality}
S.~Fomin, \emph{{Duality of graded graphs}}, Journal of Algebraic Combinatorics
  \textbf{3} (1994), no.~4, 357--404.

\bibitem[Ful05]{Fulman2005}
J.~Fulman, \emph{Stein's method and {P}lancherel measure of the symmetric
  group}, Trans. Amer. Math. Soc. \textbf{357} (2005), no.~2, 555--570,
  arXiv:math/0305423 [math.RT].

\bibitem[Ful09]{Fulman2007}
\bysame, \emph{Commutation relations and {M}arkov chains}, Prob. Theory Rel.
  Fields \textbf{144} (2009), no.~1, 99--136, arXiv:0712.1375 [math.PR].

\bibitem[GSK04]{gikhman2004theoryII}
I.I. Gikhman, A.V. Skorokhod, and S.~Kotz, \emph{{The theory of stochastic
  processes II}}, Springer Verlag, 2004.

\bibitem[HH92]{Hoffman1992}
P.N. Hoffman and J.F. Humphreys, \emph{Projective representations of the
  symmetric groups}, Oxford Univ. Press, 1992.

\bibitem[HKPV06]{Peres2006determinantal}
J.B. Hough, M.~Krishnapur, Y.~Peres, and B.~Vir{\'a}g, \emph{{Determinantal
  processes and independence}}, Probability Surveys \textbf{3} (2006),
  206--229, arXiv:math/0503110 [math.PR].

\bibitem[IIKS90]{its1990differential}
A.R. Its, A.G. Izergin, V.E. Korepin, and N.A. Slavnov, \emph{{Differential
  equations for quantum correlation functions}}, Int. J. Mod. Phys. B
  \textbf{4} (1990), no.~5, 1003--1037.

\bibitem[Iva99]{IvanovNewYork3517-3530}
V.~Ivanov, \emph{The {D}imension of {S}kew {S}hifted {Y}oung {D}iagrams, and
  {P}rojective {C}haracters of the {I}nfinite {S}ymmetric {G}roup}, Jour. Math.
  Sci. (New York) \textbf{96} (1999), no.~5, 3517--3530, in Russian: Zap.
  Nauchn. Sem. POMI {\bf{}240\/} (1997), 115-135, arXiv:math/0303169 [math.CO].

\bibitem[Iva06]{Ivanov2006plancherel}
\bysame, \emph{{Plancherel measure on shifted Young diagrams}}, Representation
  theory, dynamical systems, and asymptotic combinatorics, 2, vol. 217, Transl.
  AMS, 2006, pp.~73--86.

\bibitem[JN06]{johansson2006eigenvalues}
K.~Johansson and E.~Nordenstam, \emph{{Eigenvalues of GUE minors}}, Electron.
  J. Probab \textbf{11} (2006), no.~50, 1342--1371, arXiv:math/0606760
  [math.PR].

\bibitem[Joh01]{Johansson1999}
K.~Johansson, \emph{{Discrete orthogonal polynomial ensembles and the
  Plancherel measure}}, Annals of Mathematics \textbf{153} (2001), no.~1,
  259--296, arXiv:math/9906120 [math.CO].

\bibitem[Joh02]{johansson2002non}
\bysame, \emph{{Non-intersecting paths, random tilings and random matrices}},
  Probability theory and related fields \textbf{123} (2002), no.~2, 225--280,
  arXiv:math/0011250 [math.PR].

\bibitem[Joh05]{johansson2005non}
\bysame, \emph{{Non-intersecting, simple, symmetric random walks and the
  extended Hahn kernel}}, Annales de l'institut Fourier \textbf{55} (2005),
  no.~6, 2129--2145, arXiv:math/0409013 [math.PR].

\bibitem[Kat05]{Katori2005PfDyn}
M.~Katori, \emph{{Non-colliding system of Brownian particles as Pfaffian
  process}}, RIMS Kokyuroku \textbf{1422} (2005), 12--25, arXiv:math/0506186
  [math.PR].

\bibitem[KM57]{KMG57BDClassif}
S.~Karlin and J.~McGregor, \emph{The classification of birth and death
  processes}, Trans. Amer. Math. Soc. \textbf{86} (1957), 366--400.

\bibitem[KM58]{KMG58Linear}
\bysame, \emph{Linear growth, birth and death processes}, J. Math. Mech.
  \textbf{7} (1958), 643--662.

\bibitem[KNT04]{NagaoKatoriTanemura2004PfDyn}
M.~Katori, T.~Nagao, and H.~Tanemura, \emph{{Infinite systems of non-colliding
  Brownian particles}}, {Adv. Stud. in Pure Math. \textbf{39} ``Stochastic
  Analysis on Large Scale Interacting Systems''}, Mathematical Society of
  Japan, 2004, pp.~283--306, arXiv:math.PR/0301143.

\bibitem[KOO98]{Kerov1998}
S.~Kerov, A.~Okounkov, and G.~Olshanski, \emph{{T}he boundary of {Y}oung graph
  with {J}ack edge multiplicities}, Intern. Math. Research Notices \textbf{4}
  (1998), 173--199, arXiv:q-alg/9703037.

\bibitem[KOV93]{Kerov1993}
S.~Kerov, G.~Olshanski, and A.~Vershik, \emph{Harmonic analysis on the infinite
  symmetric group. {A} deformation of the regular representation}, Comptes
  Rendus Acad. Sci. Paris Ser. I \textbf{316} (1993), 773--778.

\bibitem[KOV04]{Kerov2004}
\bysame, \emph{Harmonic analysis on the infinite symmetric group}, Invent.
  Math. \textbf{158} (2004), no.~3, 551--642, arXiv:math/0312270 [math.RT].

\bibitem[KS96]{Koekoek1996}
R.~Koekoek and R.F. Swarttouw, \emph{{The Askey-scheme of hypergeometric
  orthogonal polynomials and its q-analogue}}, Tech. report, Delft University
  of Technology and Free University of Amsterdam, 1996.

\bibitem[KT99]{karlin1981second}
S.~Karlin and H.M. Taylor, \emph{{A second course in stochastic processes}},
  Academic press, 1999.

\bibitem[Lan85]{Lang1985sl2}
S.~Lang, \emph{{$SL_2 (\mathbb{R})$}}, Springer, 1985.

\bibitem[Mac95]{Macdonald1995}
I.G. Macdonald, \emph{Symmetric functions and {H}all polynomials}, 2nd ed.,
  Oxford University Press, 1995.

\bibitem[Mat05]{Matsumoto2005}
S.~Matsumoto, \emph{{Correlation functions of the shifted Schur measure}}, J.
  Math. Soc. Japan, vol. \textbf{57} (2005), no.~3, 619--637,
  arXiv:math/0312373 [math.CO].

\bibitem[Meh04]{mehta2004random}
M.L. Mehta, \emph{{Random matrices}}, Academic press, 2004.

\bibitem[MP83a]{Pandey1983gaussian}
M.L. Mehta and A.~Pandey, \emph{{Gaussian ensembles of random Hermitian
  matrices intermediate between orthogonal and unitary ones}}, Communications
  in Mathematical Physics \textbf{87} (1983), no.~4, 449--468.

\bibitem[MP83b]{Mehta1983some}
\bysame, \emph{{On some Gaussian ensembles of Hermitian matrices}}, Journal of
  Physics A: Mathematical and General \textbf{16} (1983), 2655--2684.

\bibitem[Nag07]{Nagao2007pfaffian}
T.~Nagao, \emph{{Pfaffian Expressions for Random Matrix Correlation
  Functions}}, Journal of Statistical Physics \textbf{129} (2007), no.~5,
  1137--1158, arXiv:0708.2036 [math-ph].

\bibitem[Naz92]{Nazarov1992}
M.L. Nazarov, \emph{Projective representations of the infinite symmetric
  group}, Representation theory and dynamical systems (A. M. Vershik, ed.),
  Advances in Soviet Mathematics, Amer. Math. Soc. \textbf{9} (1992), 115--130.

\bibitem[Nel59]{nelson1959analytic}
E.~Nelson, \emph{{Analytic vectors}}, Ann. Math. \textbf{2} (1959), no.~70,
  572--615.

\bibitem[NF98]{nagao1998multilevel}
T.~Nagao and P.J. Forrester, \emph{{Multilevel dynamical correlation functions
  for Dyson's Brownian motion model of random matrices}}, Physics Letters A
  \textbf{247} (1998), no.~1-2, 42--46.

\bibitem[NW91a]{nagaoWadati1991}
T.~Nagao and M.~Wadati, \emph{{Correlation functions of random matrix ensembles
  related to classical orthogonal polynomials}}, J. Phys. Soc. Japan
  \textbf{60} (1991), 3298--3322.

\bibitem[NW91b]{nagaoWadati1992}
\bysame, \emph{{Correlation functions of random matrix ensembles related to
  classical orthogonal polynomials II, III}}, J. Phys. Soc. Japan \textbf{61}
  (1991), 78--88, 1910--1918.

\bibitem[Oko00]{okounkov2000random}
A.~Okounkov, \emph{{Random matrices and random permutations}}, International
  Mathematics Research Notices \textbf{2000} (2000), no.~20, 1043--1095,
  arXiv:math/9903176 [math.CO].

\bibitem[Oko01a]{okounkov2001infinite}
\bysame, \emph{{Infinite wedge and random partitions}}, Selecta Mathematica,
  New Series \textbf{7} (2001), no.~1, 57--81, arXiv:math/9907127 [math.RT].

\bibitem[Oko01b]{Okounkov2001a}
\bysame, \emph{{S}{L}(2) and z-measures}, Random matrix models and their
  applications (P.~M. Bleher and A.~R. Its, eds.), Mathematical Sciences
  Research Institute Publications, vol. {\bf{}40\/}, pp.~407--420, Cambridge
  Univ. Press, 2001, arXiv:math/0002135 [math.RT].

\bibitem[Ols08a]{Olshansk2008-difference}
G.~Olshanski, \emph{Difference operators and determinantal point processes},
  Functional Analysis and Its Applications \textbf{42} (2008), no.~4, 317--329,
  arXiv:0810.3751 [math.PR].

\bibitem[Ols08b]{Olshanski-fockone}
\bysame, \emph{{Fock Space and Time-dependent Determinantal Point Processes}},
  unpublished work, 2008.

\bibitem[Ols10]{Olshanski2009}
\bysame, \emph{Anisotropic {Y}oung diagrams and infinite-dimensional diffusion
  processes with the {J}ack parameter}, International Mathematics Research
  Notices \textbf{2010} (2010), no.~6, 1102--1166, arXiv:0902.3395 [math.PR].

\bibitem[OR03]{okounkov2003correlation}
A.~Okounkov and N.~Reshetikhin, \emph{{Correlation function of Schur process
  with application to local geometry of a random 3-dimensional Young diagram}},
  Journal of the American Mathematical Society \textbf{16} (2003), no.~3,
  581--603, arXiv:math/0107056 [math.CO].

\bibitem[Pet09]{Petrov2007}
L.~Petrov, \emph{A two-parameter family of infinite-dimensional diffusions in
  the {K}ingman simplex}, Functional Analysis and Its Applications \textbf{43}
  (2009), no.~4, 279--296, arXiv:0708.1930 [math.PR].

\bibitem[Pet10a]{Petrov2010}
\bysame, \emph{{Random Strict Partitions and Determinantal Point Processes}},
  Electronic Communications in Probability \textbf{15} (2010), 162--175,
  arXiv:1002.2714 [math.PR].

\bibitem[Pet10b]{Petrov2009eng}
\bysame, \emph{{Random walks on strict partitions}}, Journal of Mathematical
  Sciences \textbf{168} (2010), no.~3, 437--463, in Russian: Zap. Nauchn. Sem.
  POMI {\bf{}373\/} (2009), 226--272, arXiv:0904.1823 [math.PR].

\bibitem[PS02]{PhahoferSpohn2002}
M.~Praehofer and H.~Spohn, \emph{{Scale invariance of the PNG droplet and the
  Airy process}}, J. Stat. Phys. \textbf{108} (2002), 1071--1106, arXiv:
  math.PR/0105240.

\bibitem[Puk64]{Pukanszky_SL2_1964}
L.~Pukanszky, \emph{{The Plancherel formula for the universal covering group of
  $SL(2,\mathbb{R})$}}, Mathematische Annalen \textbf{156} (1964), no.~2,
  96--143.

\bibitem[Rai00]{Rains2000}
E.M. Rains, \emph{Correlation functions for symmetrized increasing
  subsequences}, 2000, arXiv:math/0006097 [math.CO].

\bibitem[Roz99]{Rozhkovskaya1997multiplicative}
N.~Rozhkovskaya, \emph{{Multiplicative distributions on Young graph}}, Jour.
  Math. Sci. (New York) \textbf{96} (1999), no.~5, 3600--3608, in Russian: Zap.
  Nauchn. Sem. POMI {\bf{}240\/} (1997), 245--256.

\bibitem[Sag87]{Sag87}
B.E. Sagan, \emph{{Shifted tableaux, Schur Q-functions, and a conjecture of
  Stanley}}, J. Comb. Theo. A \textbf{45} (1987), 62--103.

\bibitem[Sch11]{Schur1911}
I.~Schur, \emph{{\"U}ber die {D}arstellung der symmetrischen und der
  alternierenden {G}ruppe durch gebrocheme lineare {S}ubstitionen}, J. Reine
  Angew. Math. \textbf{139} (1911), 155--250.

\bibitem[Sos00]{Soshnikov2000}
A.~Soshnikov, \emph{Determinantal random point fields}, Russian Mathematical
  Surveys \textbf{55} (2000), no.~5, 923--975, arXiv:math/0002099 [math.PR].

\bibitem[Sta88]{stanley1988differential}
R.~Stanley, \emph{Differential posets}, Journal of the American Mathematical
  Society \textbf{1} (1988), no.~4, 919--961.

\bibitem[Str10a]{Strahov2009}
E.~Strahov, \emph{{The z-measures on partitions, Pfaffian point processes, and
  the matrix hypergeometric kernel}}, Advances in mathematics \textbf{224}
  (2010), no.~1, 130--168, arXiv:0905.1994 [math-ph].

\bibitem[Str10b]{Strahov2009z}
\bysame, \emph{{Z-measures on partitions related to the infinite Gelfand pair
  $(S (2\infty), H (\infty))$}}, Journal of Algebra \textbf{323} (2010), no.~2,
  349--370, arXiv:0904.1719 [math.RT].

\bibitem[TW96]{Tracy1996orthogonal}
C.A. Tracy and H.~Widom, \emph{{On orthogonal and symplectic matrix
  ensembles}}, Communications in Mathematical Physics \textbf{177} (1996),
  no.~3, 727--754, arXiv:solv-int/9509007.

\bibitem[Ver96]{Vershik1996StatMech}
A.M. Vershik, \emph{Statistical mechanics of combinatorial partitions, and
  their limit shapes}, Funct. Anal. Appl. \textbf{30} (1996), 90--105.

\bibitem[VK87]{vershik1987locally}
A.~Vershik and S.~Kerov, \emph{{Locally semisimple algebras. Combinatorial
  theory and the $K_0$-functor}}, Journal of Mathematical Sciences \textbf{38}
  (1987), no.~2, 1701--1733.

\bibitem[VK88]{Vilenkin-Klimyk-DAN_UKR_1988}
N.Y. Vilenkin and A.U. Klimyk, \emph{{Representations of the group $SU(1, 1)$,
  and the Krawtchouk-Meixner functions}}, Dokl. Akad. Nauk Ukrain. SSR Ser. A
  (1988), no.~6, 12--16.

\bibitem[VK95]{Vilenkin-Klimyk-ITOGI1995-en}
\bysame, \emph{{Representations of Lie groups and special functions}},
  {Representation Theory and Noncommutative Harmonic Analysis II} (A.A.
  Kirillov, ed.), Springer, 1995, (translation of VINITI vol. 59, 1990),
  pp.~137--259.

\bibitem[Vul07]{Vuletic2007shifted}
M.~Vuletic, \emph{{Shifted Schur Process and Asymptotics of Large Random Strict
  Plane Partitions}}, International Mathematics Research Notices \textbf{2007}
  (2007), no.~rnm043, arXiv:math-ph/0702068.

\bibitem[War07]{warren2005dyson}
J.~Warren, \emph{{Dyson's Brownian motions, intertwining and interlacing}},
  Electron. J. Probab. \textbf{12} (2007), no.~19, 573--590, arXiv:math/0509720
  [math.PR].

\bibitem[Wor84]{worley1984theory}
D.R. Worley, \emph{{A theory of shifted Young tableaux}}, Ph.D. thesis, MIT,
  Dept. of Mathematics, 1984.

\end{thebibliography}
\end{document}